\documentclass[12pt]{amsart}

\usepackage{geometry}
\usepackage[dvipsnames]{xcolor}
\usepackage[shortlabels]{enumitem}
\usepackage{graphicx}
\usepackage{booktabs}
\usepackage[utf8]{inputenc}
\usepackage{graphicx}
\usepackage{amsmath,amssymb,amsthm,mathtools}
\usepackage{booktabs}
\usepackage{MnSymbol}
\usepackage{latexsym, verbatim}
\usepackage{appendix}
\usepackage{units}
\usepackage{amsmath}
\usepackage[normalem]{ulem}
\usepackage{hyperref}
\usepackage{algorithm}
\usepackage{algpseudocode}
\usepackage{soul}
\usepackage{nicematrix}
\usepackage{tikz}
\usepackage{pgfplots}
\pgfplotsset{compat=1.18}
\usetikzlibrary{pgfplots.groupplots}
\usepackage{subfigure}
\usepackage{forest}
\usepackage{standalone}
\usepackage{tikz-3dplot}
\usepackage{caption}
\usepackage{subcaption}

\mathtoolsset{showonlyrefs}
\newtheorem{theorem}{Theorem}[section]
 
\newtheorem{lemma}[theorem]{Lemma}
\newtheorem{proposition}[theorem]{Proposition}

\newtheorem{remark}[theorem]{Remark}

\theoremstyle{definition}
\newtheorem{example}[theorem]{Example}

\newtheorem*{theorem**}{Theorem\theoremnum}
\newenvironment{theorem*}[1][]{%
  \edef\theoremnum{\if\relax\detokenize{#1}\relax\else~#1\fi}% Store theorem number
  \begin{theorem**}
}{%
  \end{theorem**}
}

\newcommand{\fa}{\mathbf{a}}
\newcommand{\fb}{\mathbf{b}}
\newcommand{\fx}{\mathbf{x}}

\newcommand{\PP}{\mathbb{P}}
\newcommand{\RR}{\mathbb{R}}

\DeclareMathOperator{\Span}{Span}

\DeclareMathOperator{\rank}{rank}

\DeclareMathOperator{\Vect}{Vect}

\DeclareMathOperator{\Mat}{Mat}

\newcommand{\T}{^\mathsf{T}}

\newcommand{\bx}{\mathbf{x}}
\newcommand{\by}{\mathbf{y}}
\newcommand{\bz}{\mathbf{z}}

\newcommand{\fc}{\mathbf{c}}

\newcommand{\bs}{\mathbf{s}}

\begin{document}

\title[Contrastive Independent Component Analysis]{Contrastive independent component analysis}
% Use letters for affiliations, numbers to show equal authorship (if applicable) and to indicate the corresponding author
% \author[a]{Kexin Wang}
% \author[b]{Aida Maraj}
% \author[a]{Anna Seigal\textsuperscript{1,}}

% \affil[a]{Harvard University, School Of Engineering And Applied Sciences, 29 Oxford Street, Cambridge, MA 02138
% USA}
% \affil[b]{Max Planck Institute of Molecular Cell Biology and Genetics and Center for Systems Biology Dresden}

% % Please give the surname of the lead author for the running footer
% \leadauthor{Wang}

\author{Kexin Wang}
\address{Harvard University, Pierce Hall, 29 Oxford Street, Cambridge, MA 02138, USA}
\email{kexin\_wang@g.harvard.edu}

\author{Aida Maraj}
\address{Max Planck Institute of Molecular Cell Biology and Genetics and Center for Systems Biology, Dresden, Germany}
\email{maraj@mpi-cbg.de}

\author{Anna Seigal}
\address{Harvard University, Pierce Hall, 29 Oxford Street, Cambridge, MA 02138, USA}
\email{aseigal@seas.harvard.edu} 

% % Please add a significance statement to explain the relevance of your work
% \significancestatement{Visualizing data and finding patterns in data are ubiquitous problems in the sciences. Increasingly, applications 
% seek signal and structure in a contrastive setting: a foreground dataset relative to a background dataset. 
% The goal is to learn patterns and visualize the foreground after ``subtracting off" the effect of the background. 
% For this purpose, we propose contrastive independent component analysis (cICA).
% We investigate cICA theoretically and computationally. We find that, relative to other approaches, cICA is more expressive; that is, able to model a broader range of settings, while simultaneously being identifiable, able to recover patterns uniquely.
% }

% % Please include corresponding author, author contribution and author declaration information
% \authorcontributions{K.W., A.M., A.S. devised the framework, analyzed data, and wrote the paper.}
% %\authordeclaration{Please declare any competing interests here.}
% %\equalauthors{\textsuperscript{1}A.O.(Author One) contributed equally to this work with A.T. (Author Two) (remove if not applicable).}
% \correspondingauthor{\textsuperscript{1}To whom correspondence should be addressed. E-mail: aseigal@seas.harvard.edu}

% At least three keywords are required at submission. Please provide three to five keywords, separated by the pipe symbol.

\begin{abstract}
In recent years, there has been growing interest in jointly analyzing a foreground dataset, representing an experimental group, and a background dataset, representing a control group. 
The goal of such contrastive investigations is to identify salient features in the experimental group relative to the control.
Independent component analysis (ICA) is a powerful tool for learning independent patterns in a dataset.
We generalize it to contrastive ICA (cICA).
For this purpose, we devise a new linear algebra based tensor decomposition algorithm, which is more expressive but just as efficient and identifiable as other linear algebra based algorithms.
We establish the identifiability of cICA and demonstrate its performance in finding patterns and visualizing data, using synthetic, semi-synthetic, and real-world datasets, comparing the approach to existing methods.
\end{abstract}

\maketitle

\section{Introduction}

Finding and understanding patterns in data is fundamental in various scientific fields. 
Often, data have been collected under two different settings, such as a group of patients receiving treatment and a control group, or a group of patients with a certain disease and a group without the disease. 
The goal is to understand the effect of the treatment or to understand the genetic changes that describe the disease. 
While standard data analysis methods can be used, which restrict attention to one of the datasets or combine them together, an alternate view is offered by contrastive methods. Contrastive methods view the two settings as a foreground and a background. They seek to learn patterns in the foreground after accounting for (or, ``subtracting off") the background. The hope is that such patterns encode useful structures and offer a good basis for dimensionality reduction and visualization of the data, to identify fine-grained structures and clusters particular to the foreground.

Back in the 1980s, Flury initiated the idea of comparing covariance matrices and finding principal components across multiple datasets~\cite{flury1983some,flury1984common,flury1987two}. The contrastive viewpoint was then addressed and formalized in \cite{zou2013contrastive}, where the authors discussed contrastive topic modeling and contrastive hidden Markov models.
Principal component analysis (PCA) was generalized to
contrastive PCA (cPCA) in~\cite{abid2017contrastive,abid2018exploring}.
A latent variable model perspective is taken in \cite{li2020probabilistic,severson2019unsupervised}.
The present work extends such methods, specifically cPCA, to a more expressive and identifiable setting. Specifically, it removes simplifying assumptions that amount of each background signal present in the foreground is the same \cite{zou2013contrastive,abid2017contrastive,abid2018exploring,severson2019unsupervised,li2020probabilistic}, that the latent variables are Gaussians \cite{abid2017contrastive,abid2018exploring,severson2019unsupervised,li2020probabilistic}, and that the salient patterns in the foreground data are orthogonal \cite{abid2017contrastive,abid2018exploring}. 
The greater expressivity and identifiability are achieved using the higher-order cumulant tensors of the foreground and background data, which encode more fine-grained structure than the covariance.

We call the method contrastive independent component analysis (cICA). 
 Independent component analysis (ICA) is a blind source separation method, which seeks to recover latent sources and unknown mixing from observations of mixtures of signals~\cite{comon2010handbook}.
 ICA assumes that latent sources are independent. In extending ICA to the contrastive setting, the idea is that background data is generated by mixing of independent sources while foreground data is generated by the background mixing together with a foreground mixing of independent sources.

 We show using connections to classical algebraic geometry that cICA has strong identifiability properties. 
This enables the contribution of each background pattern to the foreground to be found uniquely, avoiding the need for a sweep of hyperparameters to find the best multiple of the background to subtract from the foreground and avoids the assumption that the background contribution to the foreground is via a single scalar multiple, both of which are required in~\cite{zou2013contrastive,abid2017contrastive,abid2018exploring,li2020probabilistic}. 

To implement cICA, we devise a new hierarchical tensor decomposition based on recursive eigendecompositions. 
 The decomposition encourages (rather than imposes) orthogonality between the rank one summands.
We show that it recovers accurate patterns for synthetic data. 
  We turn cICA into a dimensionality reduction tool and investigate its performance on real-world data, comparing the plots to those obtained with other contrastive methods to see its competitiveness.

The paper is organized as follows. 
We define cICA in Section \ref{sec:cica}. 
We introduce the new hierarchical tensor decomposition in Section \ref{sec: HTD}.
We study identifiability and algorithms for cICA in Section~\ref{sec:id ang alg}.
Numerical results are in Section~\ref{sec:numerical result}.

\section{From ICA to contrastive ICA}\label{sec:cica}

Blind source separation seeks to recover latent sources and unknown mixing from observations of mixtures of signals~\cite{comon2010handbook}. A special case is independent component analysis (ICA),
which assumes that the latent sources are independent.
ICA was introduced in 1985 \cite{ans1985architectures} and popularized by Comon in his paper \cite{comon1994independent}.

ICA can be viewed as a generalization of PCA, where instead of finding uncorrelated components, it goes a step further by aiming to make the components statistically independent and instead of decomposing second-order information (covariance matrices), it decomposes higher-order statistics (via the cumulant tensors).

ICA studies observations that are a linear mixture of independent source variables. 
Applications include recovering speech and
brain signals~\cite{bartlett2002face,jung2001imaging}, causal discovery~\cite{shimizu2006linear}, and image decomposition~\cite{hyvarinen1999fast}.
We write the ICA  model as 
\begin{equation}
    \label{eqn:usual_ica}
    \by = A \bz,
\end{equation}
where $\bz$ is a vector of $r$ independent latent random variables, the mixing matrix is $A \in \RR^{p \times r}$, and $\by$ is a vector of $p$ observed variables. The $i$-th column of $A$ records a pattern in the data: the contribution of variable $z_i$ to each of the $p$ observed variables. 
The identifiability of ICA refers to the uniqueness of the mixing matrix $A$ and sometimes also of the variables $\bz$; see~\cite{1306473,comon1994independent,wang2024identifiability}.

Many algorithms for ICA proceed via tensor decomposition, see e.g.~\cite{comon2010handbook,cardoso1993blind,de2001independent,de2007fourth}. 
The cumulants of a distribution are symmetric tensors that encode it.
The $d$-th cumulant $\kappa_d(\by)$ of $\by$ is a symmetric order $d$ tensor of format $p \times \cdots \times p$ 
whose entry at position $(j_1, \ldots, j_d)$ is
\begin{equation}
 \label{eqn:ICA}
     \sum_{i=1}^r \lambda_i (\fa_i)_{j_1} \cdots (\fa_i)_{j_d},
 \end{equation}
 where the scalar $\lambda_i$ is the $d$-th cumulant of $z_i$ and the vector $\fa_i \in \RR^p$ is the $i$-th column of $A$.
 We denote this by
 \begin{equation}
 \label{eqn:ICA1}
     \kappa_{d}(\by)=\sum_{i=1}^r \lambda_i \fa_i^{\otimes d}.
 \end{equation}
This decomposition~\eqref{eqn:ICA1} follows from the multi-linear properties of cumulants and the fact that cumulant tensors of independent variables are diagonal, see~\cite[Chapter 2]{mccullagh2018tensor}.
The matrix $A$ can be recovered using tensor decomposition of the cumulant tensor~\eqref{eqn:ICA}. If the tensor decomposition is identifiable, then the columns $\fa_i$ with $\lambda_i \neq 0$ can be recovered uniquely up to permutation and scaling of columns.  Thus tensor decomposition of higher-order cumulant tensors gives an algorithm for ICA, provided no source is Gaussian (this is required for non-zero higher-order cumulants).

In this paper, we extend ICA, and tensor decomposition for ICA, to the comparison of two distributions. 
We call this contrastive ICA (cICA), 
by analogy with cPCA~\cite{abid2018exploring}. 
We have two observed distributions, a foreground, and a background.
Both are assumed to be linear mixtures of independent source variables. Our cICA model expresses the background $\by$ and foreground $\bx$ as 
\begin{equation}
\label{eqn:cica}
    \by = A \mathbf{z}\qquad \text{and} \qquad \bx = A \mathbf{z}' + B \mathbf{s}. 
\end{equation}
The background distribution $\by$ is a linear mixture of a random vector $\bz$ of $r$ independent random variables, as in~\eqref{eqn:usual_ica}. 
The foreground $\bx$ is a mixture of $r + \ell$ independent variables $\bz' = (z_1', \ldots, z_r')$ and $\bs = (s_1, \ldots, s_\ell)$. 
The columns of $A$ are the patterns in the background: column $\fa_i \in \RR^p$ records how source variable $z_i$ appears among the $p$ background variables as well as how source variable $z_i'$ appears among the $p$ foreground variables. The columns of $B$ are patterns that appear only in the foreground. 
They correspond to the variables $s_i$, referred to as the salient variables in~\cite{abid2019contrastive}.

We propose a tensor decomposition algorithm to recover mixing matrices $A$ and $B$ from~\eqref{eqn:cica}. These matrices record the patterns that encode our background and foreground distributions. We apply the algorithm to empirical cumulant tensors of $\bx$ and $\by$ obtained from sample data. We order the columns of matrix $B$ to obtain a dimensionality reduction tool.
We work under the assumption that $\bz,\bz^\prime, \bs$ are non-Gaussian, an assumption that also appears for usual ICA. This can likely be relaxed to that at most one source is Gaussian, cf.~\cite{comon1994independent,wang2024identifiability}.

Under the model~\eqref{eqn:cica}, the $d$-th cumulants of the background and foreground data are, respectively,
\begin{equation}
    \label{eqn:tensor_decomp}
    \kappa_{d}(\by)=\sum_{i=1}^r \lambda_i \fa_i^{\otimes d}, \qquad \quad \kappa_{d}(\bx)=\sum_{i=1}^r \lambda_i' \fa_i^{\otimes d}+\sum_{j=1}^\ell \nu_j \fb_j^{\otimes d},
\end{equation}
where $\lambda_i$ is the $d$-th cumulant of $z_i$, $\lambda_i'$ is the $d$-th cumulant of $z_i'$, and $\nu_j$ is the $d$-th cumulant of~$s_j$. 
This follows from the multilinearity of cumulants and that cumulant tensors of independent sources are diagonal, as for usual ICA.
See Figure \ref{fig: decompose foreground cumulant} for an illustration of $\kappa_3(\bx)=\sum_{i=1}^r \lambda_i' \fa_i^{\otimes d}+\sum_{j=1}^\ell \nu_j \fb_j^{\otimes d}$ when $d=3$.

\begin{figure}[htbp]
\centering
\includegraphics[scale = 0.4]{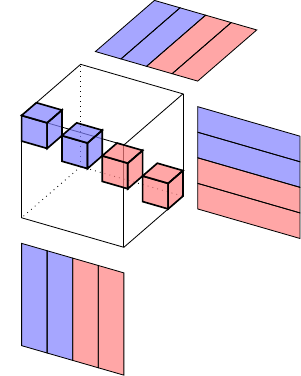}
\caption{Tensor decomposition for $\kappa_3(\bx)=\sum_{i=1}^r \lambda_i' \fa_i^{\otimes d}+\sum_{j=1}^\ell \nu_j \fb_j^{\otimes d}$ when $d=3$ and $r = \ell = 2$.
The central $4 \times 4 \times 4$ diagonal tensor is multiplied along each index by a matrix with four columns, whose first two columns (blue) are the background patterns and second two (red) are the foreground patterns. }
\label{fig: decompose foreground cumulant}
\end{figure}

To recover $A$ and $B$, we compute a joint decomposition of the cumulant tensors $\kappa_d(\by)$ and $\kappa_d(\bx)$~\eqref{eqn:tensor_decomp}, via three steps:
\begin{enumerate}
    \item Compute a symmetric tensor decomposition of $\kappa_d(\by)$ to learn $A$.
    \item  Find the coefficients $\lambda_i'$ of each $\fa_i^{\otimes d}$ in $\kappa_d(\bx)$ to obtain $\sum_{j=1}^\ell \nu_j \fb_j^{\otimes d}$.
    \item Compute a symmetric tensor decomposition of $\sum_{j=1}^\ell \nu_j \fb_j^{\otimes d}$ to learn $B$.
\end{enumerate}

We work with the fourth order cumulants $d=4$, since the tensor decomposition we use works better for an even order symmetric tensor. 
For the third step of our approach, we require a tensor decomposition method that is efficient and promotes orthogonality among the rank-1 components, which aids interpretability and improves visualizations.
To address this, we propose a hierarchical eigendecomposition based algorithm, which we describe in more detail in the next section. 
The algorithm uses linear algebra and can handle tensors of rank up to $p^2$ (compared to rank $p$ for other linear algebra-based methods \cite{harshman1970foundations,kolda2015symmetric}).

\subsection{Related Work}

We relate cICA to other contrastive models.
In cPCA, the contrastive patterns are principal components of the foreground covariance matrix minus a scalar multiple of the background covariance matrix~\cite{abid2017contrastive,abid2018exploring}.  
We can specialize cICA to cPCA by setting $\mathbf{z}'=\gamma \mathbf{z}$ and studying observed distributions $\bx$ and $\by$ via their covariance matrices ($d=2$).   
Probabilistic contrastive PCA (PCPCA) is introduced in \cite{li2020probabilistic}, where foreground patterns are inferred by maximizing a likelihood ratio of linear Gaussian mixtures. 
Contrastive ICA also relates to PCPCA~\cite{li2020probabilistic} but we do not impose distributional assumptions, beyond independence and non-Gaussianity, on the variables $\mathbf{z}$ and $(\mathbf{z}', \mathbf{s})$.
The paper \cite{severson2019unsupervised} studies a linear contrastive latent variable model. 
The contrastive ICA model aligns with the framework of the contrastive latent variable model proposed in~\cite{severson2019unsupervised}, but it does not assume any relationship between $\mathbf{z}$ and $\mathbf{z}'$ while the contrastive latent variable model assumes $\mathbf{z}=\mathbf{z}'$.

The setting of cICA relates to usual ICA, with block structure on the mixing matrix:
\[
\begin{aligned}
&\text{if $\bz'\!$, $\bz$, $\bs$ are independent,} \quad 
\begin{pmatrix}
    \bx \\ \by
\end{pmatrix} = 
\begin{pmatrix}
    0 & A & B \\
    A & 0 & 0 
\end{pmatrix}
\begin{pmatrix}
    \bz\phantom{'} \\ \bz' \\ \bs\phantom{'}
\end{pmatrix}; \\ 
&\text{if $\bz' = \gamma \bz$,} \quad
\begin{pmatrix}
    \bx \\ \by
\end{pmatrix} = 
\begin{pmatrix}
    \gamma A & B \\
    A & 0 
\end{pmatrix}
\begin{pmatrix}
    \bz \\ \bs
\end{pmatrix}.
\end{aligned}
\]
Identifiability can be characterized using~\cite{comon1994independent}, or using~\cite{1306473,wang2024identifiability} if the model is overcomplete (i.e. the number of sources exceeds the number of observations, which occurs for $2r + \ell > 2p$).
However, learning parameters via usual ICA requires access to the joint distribution of $(\bx, \by)$, 
which is generally unavailable because the data from the two datasets are unpaired. 
For example, single-cell RNA data for patients with a disease (foreground) and a control group (background), has each person assigned to either the foreground set or the background.

In~\cite{sturma2024unpaired}, the authors study multi-modal linear ICA. They recover the mixing matrices from each mode via usual linear ICA and use a hypothesis test to decide which latent variables are shared across modes. 
Our method differs from this as we seek patterns unique to the foreground rather than shared patterns.

Nonlinear contrastive methods have been explored in the literature.
Nonlinear ICA is studied using contrastive learning~\cite{hyvarinen2016unsupervised,hyvarinen2019nonlinear,lyu2022finite}. 
Here contrastive is used in a different context: it describes a method to train a network to distinguish two datasets. 
A nonlinear contrastive method called a contrastive variational autoencoder (cVAE) is introduced in~\cite{abid2019contrastive,severson2019unsupervised}.
The paper \cite{weinberger2022moment} presents a method for cVAE using maximum mean discrepancy to prevent leakage of information between the two sets of latent variables.
Identifiability of cVAE is studied using connections to nonlinear ICA in~\cite{lopez2024toward}. These works produce a nonlinear latent encoding of data, whereas our focus is on linear pattern vectors.

\section{Hierarchical tensor decomposition}\label{sec: HTD}

ICA has seen limited application in data visualization, one notable exception being~\cite{lim2008cumulant}. 
Existing algorithms to compute a symmetric tensor decomposition usually have randomness due to initialization and the details of the optimization process, such as the step size in gradient-based optimization. Another challenge is that the resulting vectors may be nearly parallel~\cite{landsberg2011tensors}, which yields a suboptimal basis for projecting the data and 
hinders its interpretability.
We overcome these difficulties with our proposed hierarchical tensor decomposition (HTD). Its output is deterministic and the components learned are almost orthogonal.

HTD decomposes an order four tensor via recursive eigendecompositions. 
The idea is to find a low-rank approximation of a tensor, whose rank one summands offer an interpretable basis on which to project data.
Later, we use the decomposition for cICA. In this section, we define the decomposition and study its properties. 
HTD for a tensor in $(\RR^{p})^{\otimes 4}$ uses linear structure in the space $(\RR^{p})^{\otimes 2}$ rather than $\RR^{p}$, so it handles tensors of rank up to $p^2$ (unlike $p$ in other linear algebra-based methods \cite{harshman1970foundations,kolda2015symmetric}).
The detailed comparison with other tensor decomposition methods is in Section 1 of the Appendix.

\subsection{The HTD algorithm}

Consider a symmetric tensor $T$ of format $p \times p \times p \times p$. We compute a rank $r$ approximation,
\begin{equation}
    \label{eqn:Tapprox}
    T \approx \sum_{i=1}^r \nu_i \fb_i^{\otimes 4}, 
\end{equation} 
as follows. 
Let $\Mat(T)$ be the flattening of $T$ that rearranges its $p^4$ entries into a matrix of size $p^2 \times p^2$.
The entries of $\Mat(T)$ are indexed $((i_1, i_2), (j_1, j_2))$, where $i_1, i_2, j_1, j_2 \in [p] :=\{1,\ldots,p\}$.
We compute the approximation~\eqref{eqn:Tapprox} by first computing the eigendecomposition of $\Mat(T)$, whose eigenvectors lie in $\RR^{p^2}$, and then by reshaping these eigenvectors into $p \times p$ matrices and computing their top eigenvalue and corresponding eigenvector.  By top eigenvalue we mean those of highest magnitude. This decomposition has not to our knowledge been studied before but has connections to the hierarchical tensor representations of~\cite[Chapter 11]{hackbusch2012tensor} and the PARATREE model in~\cite{salmi2009sequential}, see Section \ref{app:comparison} of the Appendix. 
See Figure \ref{fig:HTD} for an illustration of the steps of HTD on a $2\times 2 \times 2 \times 2$ tensor. Here is the HTD algorithm.

\begin{algorithm}[htbp]
%\scriptsize
\caption{Compute unit vectors $\fb_1,\ldots,\fb_r$ such that $T \approx \sum_{i=1}^r \nu_i \fb_i^{\otimes 4}$} \label{alg:hierarchical}
\begin{algorithmic}[1]
\renewcommand{\algorithmicrequire}{\textbf{Input:}}
\Require Symmetric tensor $T$ of format $p \times p \times p \times p$ and rank~$r$.
\State Compute the eigendecomposition of the $p^2 \times p^2$ flattening $\Mat(T)$. Take the top
$r$ 
eigenvalues $\mu_1,\ldots,\mu_r$, with corresponding 
eigenvectors $\mathbf{v}_1,\ldots,\mathbf{v}_r \in \RR^{p^2}$ of unit length.
\State For each $i \in [r]$, reshape $\mathbf{v}_i \in \RR^{p^2}$ to $M_i \in \RR^{p \times p}$. 
\State For each $M_i$, find the top eigenvalue $\beta_i$ and a corresponding unit length eigenvector $\fb_i \in \RR^p$. 
\renewcommand{\algorithmicrequire}{\textbf{Output:}}
\Require Rank $r$ decomposition $\sum_{i=1}^r (\mu_i \beta_i^2 )\mathbf{b}_i^{\otimes 4}$.
\end{algorithmic}
\end{algorithm}

\begin{figure}[htbp]
    \centering
    \includegraphics[scale = 0.3]{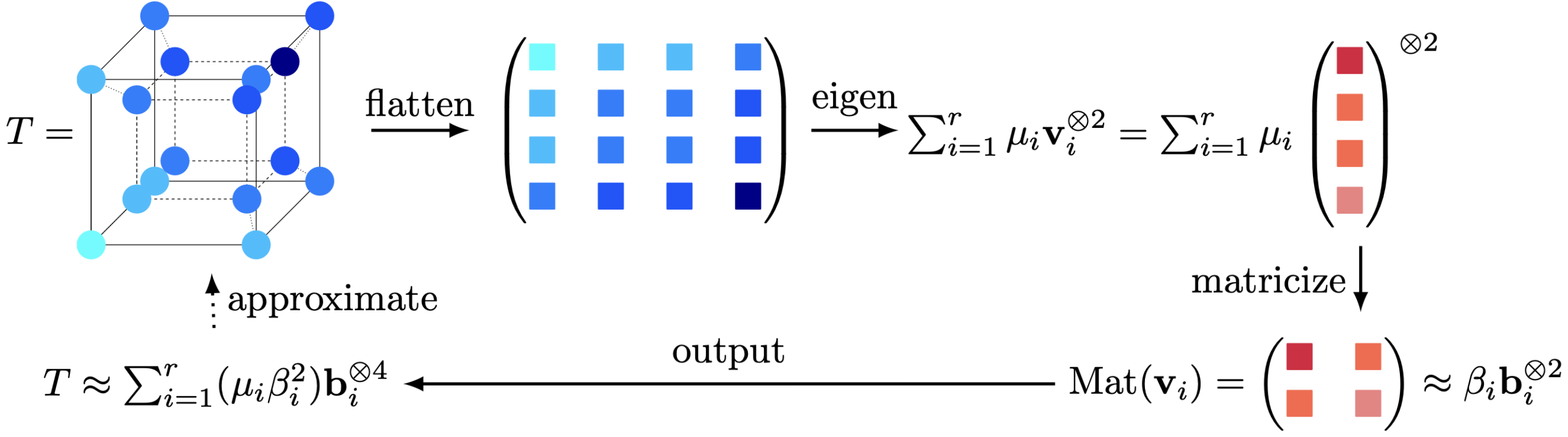}
    \caption{Steps in the HTD algorithm: input tensor $T$, matrix flattening $\Mat(T)$, best rank $r$ approximation $\Mat(T)\approx \sum_{i=1}^r \mu_i \mathbf{v}_i^{\otimes 2}$, best rank one approximation of each $\Mat(\mathbf{v}_i)$ and the output rank $r$ approximation for $T$.}
    \label{fig:HTD}
\end{figure}

We record some observations about Algorithm~\ref{alg:hierarchical}. The matrix $\Mat(T) \in \RR^{p^2 \times p^2}$ is symmetric since $T$ is symmetric. The matrices $M_1, \ldots, M_r \in \RR^{p \times p}$ are also symmetric, because the vectors $\mathbf{v}_1, \ldots, \mathbf{v}_r$ are in the column space of $\Mat(T)$, whose $(i_1, i_2)$-th row equals its $(i_2, i_1)$-th row. Although the output vectors $\fb_i$ are in general not orthogonal, as each is an eigenvector of a distinct matrix, they can be nearly orthogonal in practice, see Section~\ref{sec:properties_of_HTD}. This is because they are the leading eigenvectors of matrices that have been reshaped from orthogonal vectors $\mathbf{v}_i$.

\begin{example}[$2\times 2 \times 2\times 2$ example]
\label{ex:2x2x2x2}
Let $r = 2$. Fix \[ T=2\begin{bmatrix}
       1 \\0
    \end{bmatrix}^{\otimes 4}+\begin{bmatrix}
        0.0998\\0.995
    \end{bmatrix}^{\otimes 4}.\]  
    Then \[\mathrm{Mat}(T)=\begin{bmatrix} 2.0001 & 0.0010 & 0.0010 & 0.0099 \\  0.0010 & 0.0099 & 0.0099 & 0.0983 \\  0.0010 & 0.0099 & 0.0099 & 0.0983 \\  0.0099 & 0.0983 & 0.0983 & 0.9801 \end{bmatrix}\]
 with eigenvalues $\mu_1=2.00019, \mu_2=0.99977$ and associated eigenvectors
\[
\begin{aligned}
\mathbf{v}_1\T & \approx \begin{bmatrix}
   0.99995 & 0.00098 & 0.00098 & 0.00985
\end{bmatrix}, \\ 
\mathbf{v}_2\T & \approx \begin{bmatrix}
-0.00995 & 0.0993 & 0.0993 & 0.99003
\end{bmatrix}.
\end{aligned}
\]
Their corresponding matrices $M_1, M_2 \in \RR^{2 \times 2}$ are symmetric with top eigenvalues $\beta_1=0.99995$ and $\beta_2=0.9998$, respectively, with eigenvectors $\fb_1\T =\begin{bmatrix}
    0.99999 & 0.00099
\end{bmatrix}$ and
$\fb_2\T =\begin{bmatrix}
   0.09787 &  0.99519
\end{bmatrix}$. 
The HTD algorithm with input $T$ and $r=2$ thus outputs 
\begin{equation} \sum_{i=1}^2(\mu_i\beta_i^2)\fb_i^{\otimes 4} = 1.99999\begin{bmatrix}
      0.99999 \\ 0.00099
\end{bmatrix}^{\otimes 4} + 0.99937\begin{bmatrix}
  0.09787 \\  0.99519
\end{bmatrix}^{\otimes 4}.
\end{equation}
We note the similarity to the input tensor $T$.
\end{example}

\subsection{Properties of the decomposition}
\label{sec:properties_of_HTD}

The HTD algorithm  outputs a rank $r$ approximation of a tensor. 
In certain cases, the output closely approximates the input tensor, as in Example~\ref{ex:2x2x2x2}.
We bound the distance between the HTD approximation and the input tensor. 
We give a bound that applies to all tensors in Proposition~\ref{prop:bound_in_general}. We show that the input and output coincide for orthogonally decomposable tensors in Proposition~\ref{prop:orthogonal_case}. Our main result is Theorem~\ref{thm:nearly_orthogonal}, which bounds the distance between an input and output tensor for a tensor decomposition involving vectors that are close to orthogonal. 

The norm $\| \cdot \|_F$ refers to the Frobenius norm for matrices and tensors and the $2$-norm for vectors; i.e., the square root of the sum of the squares of the entries.
The $2$-norm of a matrix is denoted by $\| \cdot \|_2$.

\begin{proposition}
\label{prop:bound_in_general}
Let \(T\) be a symmetric tensor of format \(p \times p \times p \times p\).
Let \(T'=\sum_{i=1}^r (\mu_i \beta_i^2) \mathbf{b}_i^{\otimes 4}\) be the rank \(r\) HTD  approximation of \(T\).
Then  
\[
\|T'-T\|_F \leq {\left( \sum_{i=r+1}^{q} \mu_i^2 \right) }^{\frac12} + \sum_{i=1}^r |\mu_i| (1+ |\beta_i|){\left( \sum_{j = 2}^{r_i} (\beta_i^{(j)})^2 \right)}^{\frac12},
\]
where 
\(q\) is the rank of \(\Mat(T)\), \(r_i\) is the rank of \(M_i\),
the numbers \(\mu_1, \ldots, \mu_{r}\) are the eigenvalues of \(\Mat(T)\) in descending order of magnitude, and \(\beta_i := \beta_i^{(1)}\) is the highest magnitude eigenvalue of \(M_i\) with \(\beta_i^{(2)}, \ldots, \beta_i^{(r_i)}\) the other eigenvalues.
\end{proposition}
\begin{proof}
We use the notation from Algorithm~\ref{alg:hierarchical}.
We have
\[\| \Mat(T) - \sum_{i=1}^r \mu_i \mathbf{v}_i^{\otimes 2} \|_F^2 = \sum_{i=r+1}^q \mu_i^2 ,\,
\| M_i -\beta_i \mathbf{b}_i^{\otimes 2} \|_F^2 = \sum_{j = 2}^{r_i} (\beta_i^{(j)})^2,
\]
from the properties of the eigendecomposition of a symmetric matrix and the Frobenius norm.
Let \(T''\) be the \(p \times p \times p \times p\) tensor obtained from reshaping the truncated eigendecomposition \(\sum_{i=1}^r \mu_i \mathbf{v}_i^{\otimes 2}\) of \(\Mat(T)\). 
Then
$\|T- T''\|_F^2 = \sum_{i=r+1}^q \mu_i^2 $.
Let $\mathbf{B}_i \in \RR^{p^2}$ be the vectorization of $\mathbf{b}_i^{\otimes 2} \in \RR^{p \times p}$.
Then
\begin{align*}
\|T''-T'\|_F =& \| \sum_{i=1}^r \mu_i(\mathbf{v}_i^{\otimes 2}-\beta_i^2 \mathbf{B}_i^{\otimes 2})\|_F \\
\leq & \sum_{i=1}^r |\mu_i| \|\mathbf{v}_i^{\otimes 2} - \beta_i^2 \mathbf{B}_i^{\otimes 2}\|_F  \\
\leq &\sum_{i=1}^r |\mu_i| (\|\mathbf{v}_i^{\otimes 2} - \beta_i \mathbf{B}_i\otimes \mathbf{v}_i\|_F+ \|\beta_i^2\mathbf{B}_i^{\otimes 2} - \beta_i \mathbf{B}_i\otimes \mathbf{v}_i\|_F)\\
=  & \sum_{i=1}^r |\mu_i| (\| \mathbf{v}_i\| + |\beta_i| \|\mathbf{B}_i \| ) \| \mathbf{v}_i-\beta_i \mathbf{B}_i\| \\
= & \sum_{i=1}^r |\mu_i| (1+ |\beta_i|){\left( \sum_{j = 2}^{r_i} (\beta_i^{(j)})^2 \right)}^{\frac12},
\end{align*} 
where the penultimate equality follows from $\| \mathbf{x} \otimes \mathbf{y} \| = \| \mathbf{x} \| \cdot \| \mathbf{y} \|$ and the last equality uses $\| \mathbf{v}_i \| = \| \mathbf{B}_i \| = 1$.
We conclude with the triangle inequality 
$
\|T-T'\|_F \leq \| T-T''\|_F + \| T''-T'\|_F$.
\end{proof} 

The quantity in Proposition~\ref{prop:bound_in_general} is small if $\Mat(T)$ is well-approximated by a matrix of rank $r$, and each $M_i$ is well-approximated by a matrix of rank one.
Orthogonally decomposable tensors are those with a decomposition into orthogonal rank one terms; that is, a decomposition $T=\sum_{i=1}^r \nu_i \fb_i^{\otimes 4}$, where $\fb_1,\ldots,\fb_r$ are orthonormal~\cite{robeva2016orthogonal}. For orthogonally decomposable tensors, HTD recovers the exact decomposition. 

\begin{proposition}
\label{prop:orthogonal_case}
Let $T=\sum_{i=1}^r \nu_i \fb_i^{\otimes 4}$, where the vectors $\fb_1,\ldots,\fb_r$ are orthonormal and the coefficients $\nu_1,\ldots,\nu_r$ are distinct. Then the rank $r$ HTD approximation is the tensor~$T$.
\end{proposition}
\begin{proof}
The flattening $\Mat(T)$ has decomposition  $\sum_{i=1}^r \nu_i \mathbf{B}_i^{\otimes 2}$, where $\mathbf{B}_i \in \RR^{p^2}$ is the vectorization of $\mathbf{b}_i^{\otimes 2} \in \RR^{p \times p}$.
We have 
$\langle \mathbf{B}_i , \mathbf{B}_j \rangle = \langle \fb_i, \fb_j \rangle^2 = 0$ for all $i \neq j$,
since the vectors $\fb_i,\fb_j$ are orthogonal.
Hence this expression for $\Mat(T)$ is a sum of outer products of orthogonal vectors, so it is the eigendecomposition of $\Mat(T)$. The matrix reshaped from the eigenvector~$\mathbf{B}_i$ is $M_i = \fb_i^{\otimes 2}$. It has top eigenvalue $1$ with corresponding eigenvector $\fb_i$. Hence the output of HTD is $\sum_{i=1}^r \nu_i \fb_i^{\otimes 4}$.
\end{proof}

We extend Proposition~\ref{prop:orthogonal_case} to decompositions where the vectors $\fb_i$ are close to orthogonal and the input tensor is noisy.
The condition that the matrices $\fb_1^{\otimes2},\ldots, \fb_r^{\otimes 2}$ are linearly independent ensures that $\Mat(T)$ has rank~$r$. 
This condition holds for generic vectors $\fb_i$, provided $r\leq{{p+1 \choose 2}}$.
The quantity 
\(\min \{ \| \fb_i- \fb_i'\|, \| \fb_i + \fb_i'\| \}\) arises because of the sign indeterminacy in the vectors in the decompositions, due to the equality \((- \mathbf{b}_i)^{\otimes d} = \mathbf{b}_i^{\otimes d}\) for \(d\) even.

We sketch the proof of Theorem~\ref{thm:nearly_orthogonal}. The full proof is in Section \ref{app: thm properties of HTD} of the Appendix.

\begin{theorem}\label{thm:nearly_orthogonal}
Fix vectors $\fb_1, \ldots, \fb_\ell \in \mathbb{R}^p$ with
$
|\langle \fb_i, \fb_j \rangle| \leq \epsilon$ for all $i \neq j$.
Let
$$
T = \sum_{i=1}^\ell \nu_i \fb_i^{\otimes 4},
$$
where $\nu_1 > \cdots > \nu_\ell$, $\ell \leq p$, and $\fb_1^{\otimes 2},\ldots, \fb_\ell^{\otimes 2}$ are linearly independent. 
Fix $\hat{T}$ with
$
\| \hat{T} - T \|_F \leq \delta$. 
Let $\fc_i$ be the output patterns of the HTD algorithm with input tensor $\hat{T}$ and $\mu_i$ the corresponding recovered scalars ordered so that $\mu_1 > \cdots > \mu_\ell$.
Then for any $i \in [\ell]$,
$$
|\nu_i - \mu_i| = O(\epsilon^2) + O(\delta), \quad \text{and} \quad 
$$
$$
\min\left\{ \|\fb_i - \fc_i\|, \|\fb_i + \fc_i\| \right\} = O(\epsilon^2) + O(\delta).
$$
\end{theorem}

\begin{proof}[Proof Sketch]
Fix
$
M = \Mat(T)$. Then $M = \sum_{i=1}^r \nu_i \mathbf{B}_i^{\otimes 2}
$,
where $\mathbf{B}_i~=~\Vect(\fb_i^{\otimes 2})$.
Using Gram-Schmidt orthogonalization, we can construct a matrix $M'$ in $\mathbb{R}^{p^2 \times p^2}$ with eigendecomposition
$
\sum_{i=1}^\ell \nu_i (\mathbf{B}_i')^{\otimes 2}
$
such that
\begin{equation}\label{eq:B' B}
\| \mathbf{B}_i' - \mathbf{B}_i \| \leq 2(\ell-1)\epsilon^2+ O(\epsilon^4),
\end{equation}
\begin{equation}\label{eq: M',M}
\|M - M'\|_F \leq K \epsilon^2 + O(\epsilon^4),
\end{equation}
where $K =  \sqrt{8}\sum_{i=1}^\ell |\nu_i|(i-1).$
Suppose $\hat{M} = \text{Mat}(\hat{T})$ has eigendecomposition
$
\hat{M} = \sum_{i=1}^\ell \hat{\nu}_i \hat{\mathbf{B}}_i^{\otimes 2}.
$
The difference between $\hat{M}$ and $M'$ is bounded by 
$$
\|\hat{M}-M'\|_F \leq \|\hat{M}-M\|_F + \|M-M'\|_F \leq K\epsilon^2 + \delta + O(\epsilon^4)
$$
using the triangle inequality.
We thus obtain
\begin{equation}\label{eq:nu}
|\nu_i - \hat{\nu}_i| \leq \delta + K\epsilon^2 + O(\epsilon^4),
\end{equation}
\begin{equation}\label{eq: bhat bprime}
\| \hat{\mathbf{B}}_i - \mathbf{B}_i' \| \leq \frac{2^{\frac{3}{2}}} {\nu} ( \delta + K\epsilon^2 + O(\epsilon^4)),
\end{equation}
by Weyl's Theorem and the variant of Davis-Kahan Theorem in \cite{yu2015useful}, where $\nu~=~\min_{i \neq j} \{ |\nu_i - \nu_j|, |\nu_i| \}$.
We bound the difference between $\mathbf{B}_i$ and $\hat{\mathbf{B}}_i$ using \eqref{eq:B' B} and~\eqref{eq: bhat bprime}:
\begin{equation}\label{eq:bprime bhat}
\| \mathbf{B}_i - \hat{\mathbf{B}}_i \| \leq \| \mathbf{B}_i' - \mathbf{B}_i \| + \| \mathbf{B}_i' - \hat{\mathbf{B}}_i \| \leq L \epsilon^2 + \frac{2^\frac{3}{2}} {\nu} \delta + O(\epsilon^4),
\end{equation}
where $L = 2^{3/2} \frac{K}{\nu} + 2\ell-2$.
Then, by Weyl's theorem,
\begin{equation}\label{eq:alpha and 1}
|\alpha - 1| \leq \| \mathbf{B}_i - \hat{\mathbf{B}}_i \|,
\end{equation}
where $\alpha$ is the top eigenvalue of $\Mat(\hat{\mathbf{\mathbf{B}}}_i)$. 
HTD implies $\mu_i = \alpha^2 \hat{\nu}_i$.
The bound on $|\mu_i - \nu_i|$ then follows from \eqref{eq:nu} and \eqref{eq:alpha and 1}.
The bound of
$
\min\{ \|\fb_i - \fc_i\|, \|\fb_i + \fc_i\| \}
$
follows from \eqref{eq:bprime bhat} and~\cite{yu2015useful}, 
since $\fc_i$ is the top eigenvector of $\hat{\mathbf{B}}_i$.
\end{proof}

\section{Tensor decompositions for cICA}\label{sec:id ang alg}

Our cICA model assumes 
$\by = A \mathbf{z}$ and $\bx = A \mathbf{z}' + B \mathbf{s}$,
for $A \in \RR^{p \times r}$ and $B \in \RR^{p \times \ell}$, 
see~\eqref{eqn:cica}. This leads to the cICA tensor decompositions~\eqref{eqn:tensor_decomp}. 
One does not assume a relationship between $\bz$ and $\bz'$. 
We discuss the algorithm and identifiability of cICA in subsection \ref{sec:general_cica}.
We explain how to use cICA for dimensionality reduction in Section~\ref{sec:ranking_bs}. This projects data onto a subspace given by certain columns of the foreground mixing $B$.
We bound the end-to-end error of our algorithm in Section~\ref{sec:error analysis}.
When $\mathbf{z}' = \gamma \mathbf{z}$ for some scalar $\gamma$, we discuss an alternative algorithm in Section \ref{app: proportional cICA} of the Appendix and its performance for various datasets in Section \ref{app:simulations_details} of the Appendix.

\subsection{cICA Algorithm and Identifiability}
\label{sec:general_cica}

We present Algorithm~\ref{alg:generic a b} for cICA. 
Steps 1 and 3 both decompose a symmetric order four tensor. We use the subspace power method~\cite{kileel2019subspace} in Step 1 to prioritize the accuracy of the tensor decomposition. We use Algorithm~\ref{alg:hierarchical} in Step 3 to prioritize interpretability and efficiency.
We provide numerical experiments to justify these choices of algorithm in Section \ref{sec:SPM HTD best}.

\begin{algorithm}[htbp]
%\scriptsize
\caption{Recover background mixing $A$ and foreground mixing $B$ from the fourth cumulants of the background and foreground}\label{alg:generic a b}
\begin{algorithmic}[1]
\renewcommand{\algorithmicrequire}{\textbf{Input:}}
\Require tensors $\kappa_{4}(\bx),\kappa_{4}(\by)$ and positive integers $r$ and $\ell$.  
\State \textbf{Recover $A$:} Compute the symmetric tensor decomposition of $\kappa_{4}(\by)$ via the subspace power method~\cite{kileel2019subspace}. This recovers $A$ up to permutation and scaling of columns.
\State \textbf{Subtract background from $\kappa_{4}(\bx)$:} 
Learn the coefficients $\lambda_i^\prime$ of $\fa_1^{\otimes 4},\ldots,\fa_r^{\otimes 4}$ in $\kappa_4(\bx)$ using the deflation step of the subspace power method.
\State \textbf{Recover $B$:} Compute the symmetric tensor decomposition of $\sum_{i=1}^\ell\nu_i\fb^{\otimes 4}=\kappa_{4}(\bx)-\sum_{i=1}^r \lambda_i^\prime \fa_i^{\otimes 4}$, using Algorithm~\ref{alg:hierarchical}.
\renewcommand{\algorithmicrequire}{\textbf{Output:}}
\Require Mixing matrices $A$ and $B$.
\end{algorithmic}
\end{algorithm}

We study the identifiability of the algorithm, that is, the uniqueness of the vectors and scalars it outputs, assuming genericity. Our genericity assumption holds almost surely in the space of parameters.
We use the following lemma.  
\begin{lemma}
\label{lem:rankq}
Let vectors $\fa_i \in \RR^p$ and scalars $\lambda_i \in \RR$ be generic. Then the decomposition $T = \sum_{i=1}^q \lambda_i \fa_i^{\otimes d}$ of a symmetric $p \times p \times p \times p$ tensor $T$ is unique for 
\[ q \leq  \begin{cases} 
\lceil \frac{1}{p}{p+3\choose 4} -1 \rceil & \text{for } p\notin \{ 3,4,5 \}, \\
\lceil \frac{1}{p}{p+3\choose 4} \rceil & \text{for } p \in \{ 3,5 \}, \\ 
9 & \text{for } p=4, \text{ provided $q \neq 8$.}
\end{cases} \] 
\end{lemma}

\begin{proof}
    The rank of a generic $p \times p \times p \times p$ symmetric tensor is $\lceil \frac{1}{p} { p+3\choose 4} \rceil$ for $p\notin \{ 3,4,5\}$ and $\lceil \frac{1}{p} {p+3\choose 4} \rceil+1$ for $p \in \{ 3,4,5 \}$, by the Alexander-Hirschowitz theorem~\cite{alexander1995poly}.
    Generic rank $q$ tensors in this space, with $q$ strictly below the generic rank, have unique symmetric tensor decomposition for $(p,q)\neq (4,8)$ and two tensor decompositions for $p=4,q=8$ by \cite[Theorem 1.1]{chiantini2017generic}. 
\end{proof}

\begin{proposition}[Identifiability of the cICA tensor decomposition]
\label{prop:iden_cica_tensor}
The joint decomposition 
\begin{equation}
    \kappa_{4}(\by)=\sum_{i=1}^r \lambda_i \fa_i^{\otimes 4}, \qquad \quad \kappa_{4}(\bx)=\sum_{i=1}^r \lambda_i' \fa_i^{\otimes 4}+\sum_{j=1}^\ell \nu_j \fb_j^{\otimes 4},
\end{equation}
is unique for generic $\fa_i, \fb_j, \lambda_i,\lambda_i',\nu_j$, where  $i\in [r]$ and $j\in [\ell]$, when $r+\ell<\lceil \frac{1}{p}{p+3\choose 4}  \rceil$ for $p\neq 3,4,5$, $r+\ell\leq \lceil \frac{1}{p}{p+3\choose 4} \rceil$ for $p=3,5$, and when $r+\ell\leq 9, r+\ell\neq 8$ for $p=4$.
\end{proposition}

\begin{proof}
The tensor decomposition for cICA in the statement is identifiable when the symmetric tensor decomposition of $\kappa_4(\bx)$ is unique, as follows. 
The tensor decomposition of $\kappa_4(\bx)$, gives vectors $\fa_i,\fb_j$ up to permutation and scaling. Then we can solve a linear system to find the decomposition $\kappa_4(\by)=\sum_{i=1}^r \lambda_i \fa_i^{\otimes 4}$. 
It remains to study the identifiability of the decomposition of $\kappa_4(\mathbf{x})$. 
    It is a symmetric $p \times p \times p \times p$ tensor of rank $r + \ell$. Hence the uniqueness follows from Lemma~\ref{lem:rankq}, setting $q = r + \ell$.
\end{proof}

When $(\lambda_1,\ldots,\lambda_r)$ and $(\lambda_1',\ldots,\lambda_r')$ are proportional as vectors in $\RR^r$, we have a stronger identifiability result than the one for two separate tensor decompositions in Proposition~\ref{prop:iden_cica_tensor}.

\begin{proposition}\label{prop: new identifiability}
Consider the joint decomposition
$$
\kappa_4(\by) = \sum_{i=1}^{r} \lambda_i \fa_i^{\otimes 4}, \quad 
\kappa_4(\bx) = \sum_{i=1}^{r} \lambda_i' \fa_i^{\otimes 4} + \sum_{j=1}^{\ell} \nu_j \fb_j^{\otimes 4}.
$$
Suppose that $(\lambda_1', \dots, \lambda_{r}') = \mu(\lambda_1, \dots, \lambda_r)$ for some $\mu\in \RR\setminus \{0\}$. 
Suppose further that $\fa_i, \fb_j, \lambda_i, \mu, \nu_j$ for $i\in [r]$ and $j\in [\ell]$ are generic.
Then, the joint decomposition of $\kappa_4(\by)$ and $\kappa_4(\bx)$ is unique provided
$$
\max \left\{ \frac{1}{p} \left\lceil \frac{r}{\ell}\right\rceil + \ell , r \right\} < \frac{1}{p} \binom{p+3}{4}.
$$
\end{proposition}

\begin{proof}
We can assume that $r$ is a multiple of $\ell$: if the joint decomposition is unique with $r$ replaced by the possibly larger number $\lceil r/\ell \rceil \ell$, then the original joint decomposition with $r$ terms is also unique. 

Let $k = \frac{r}{\ell}$ and define the tensors $T_1, \ldots, T_k$ by taking a subset of $\ell$ consecutive terms from $\kappa_4(\by)$: $T_j = \sum_{i=(j-1)\ell + 1}^{j\ell}\nu_i \fb_i^{\otimes 4}$. Define
\begin{align*}
W &= \Span\left\{\kappa_4(\fx), T_1,  \dots,  T_k \right\}.
\end{align*}
Then $\sum_{i=1}^\ell \nu_i \fb_i^{\otimes 4}\in W$, since the difference between it and $\kappa_4(\fx)$ is $\mu(T_1 + \cdots + T_k)$.

Let $X\in \PP^N$ be the variety of symmetric border rank at most $\ell$ tensors in $(\mathbb{R}^{p})^{\otimes 4}$, where $N = \binom{p+3}{4}-1$. 
The tensors
\begin{equation}\label{eq: rank l tensors}
\sum_{j=1}^{\ell}\nu_j \fb_j^{\otimes 4}, T_1 ,   \dots, T_k 
\end{equation}
are generic points on $X$, 
since $\fa_i,\fb_j,\lambda_i,\nu_j$ are generic for $i\in[r], j\in[\ell]$. 
We have projective dimensions $\dim X \leq \ell p - 1$ and $\dim W = k$.
When $k+\ell p < \binom{p+3}{4}$, we have the inequality
\[
\dim X + \dim W < N.
\]
Thus the intersection $W \cap X$ contains only the points in \eqref{eq: rank l tensors}, by the Generalized Trisecant Lemma~\cite[Proposition 2.6]{chiantini2002weakly}. 
The rank $r$ satisfies the condition in Lemma \ref{lem:rankq}, since $rp< \binom{p+3}{4}$, so we can uniquely recover $T_1, \ldots, T_k$. 
We can thus recover the linear space $W$ and therefore we can recover $\sum_{j=1}^{\ell}\nu_j \fb_j^{\otimes 4}$ from $W\cap X$.
The decomposition of $\sum_{i=1}^{\ell}\nu_i b_i^{\otimes 4}$ is unique, 
since $\ell p< {\binom{p+3}{4}}$, and $\nu_j, \fb_j$ are generic for $j\in [\ell]$. Hence, the overall joint decomposition is unique.
\end{proof}

\begin{remark}
An alternative approach to study the identifiability of the joint decomposition is to stack $\kappa_4(\bx)$ and $\kappa_4(\by)$ to form a partially symmetric tensor of size $2 \times p \times p \times p \times p$. This connects to the study of Segre-Veronese varieties~\cite{abo2024non}. However, existing results do not apply to our setting, 
because the pair $(\kappa_4(\bx), \kappa_4(\by))$ has additional structure:
Proposition~\ref{prop: new identifiability} is a first step towards identifiability for partially symmetric tensors with rank-one components that appear in a subset of slices.
\end{remark}

We say that Algorithm~\ref{alg:generic a b} is identifiable if, for generic $\fa_i,\fb_j,\lambda_i,\lambda_i',\nu_j$ where $i\in [r]$, $j\in [\ell]$, we can uniquely recover the vectors $\fa_1,\ldots,\fa_r$, the coefficients $\lambda_1', \ldots, \lambda_r'$, and the vectors $\fb_1,\ldots,\fb_\ell$.

\begin{proposition}
\label{prop:iden_alg2}
Algorithm~\ref{alg:generic a b} is identifiable 
when $r+\ell \leq {p+1 \choose 2}$ for $p\neq 4$ and $r+\ell \leq 9, r,\ell \neq 8$ for $p=4$.
\end{proposition}

To prove Proposition~\ref{prop:iden_alg2} we use the following linear algebra result. See \cite[Lemma B.1]{kileel2019subspace} for a proof.

\begin{lemma}
\label{lem:unique_C}
Let \(M \in \RR^{n\times n}, U \in \RR^{n\times k}\) and \(V \in \RR^{n\times k}\) be full-rank matrices with \(k\leq n\). Let \(C^\ast=(V\T M^{-1}U)^\dagger\), where \(\dagger\) denotes the pseudo-inverse, and \(d=\rank (C^\ast)\). Then 
\[\rank (M-UCV\T )\geq n-d,\] with equality if and only if \(C=C^\ast\).
\end{lemma}

\begin{proof}[Proof of Proposition~\ref{prop:iden_alg2}]
Tensors $\sum_{i=1}^r \lambda_i \mathbf{a}_i^{\otimes 4}$ and $\sum_{j=1}^\ell \nu_j \mathbf{b}_j^{\otimes 4}$ have generic rank $r$ and rank $\ell$, respectively. So, the identifiability of Steps 1 and 3 of Algorithm~\ref{alg:generic a b} hold if 
$r,\ell<\lceil \frac{1}{p}{p+3\choose 4}  \rceil$ for $p\notin \{ 3,4,5 \}$ or $r,\ell\leq \lceil \frac{1}{p}{p+3\choose 4} \rceil$ for $p \in \{ 3,5 \}$ or $r,\ell\leq 9, r,\ell\neq 8$ for $p=4$, setting $q = r$ and $q = \ell$ in Lemma~\ref{lem:rankq}.

It remains to consider Step 2, learning the coefficients $\lambda_i'$ of $\mathbf{a}_i^{\otimes 4}$ in $\kappa_4(\mathbf{x})$.
The flattening of $\kappa_4(\bx)$
has the form 
$M= \sum_{i=1}^r \lambda_i' \mathbf{A}_i^{\otimes 2}+ \sum_{j=1}^\ell \nu_j \mathbf{B}_j^{\otimes 2} \in \RR^{p^2 \times p^2}$, where  
$\mathbf{A}_i, \mathbf{B}_j \in \RR^{p^2}$ vectorize $\fa_i^{\otimes 2}$ and $\fb_j^{\otimes 2}$, respectively. 
The scalar $\lambda_i'$ is unique if $\rank(M-\lambda_i' \mathbf{A}_i\otimes \mathbf{A}_i) = \rank (M) -1$, 
by Lemma~\ref{lem:unique_C}. It is $((\mathbf{A}_i\T V) D^{-1} (\mathbf{A}_i\T V)\T)^{-1}$, where $VDV\T$ is the thin eigendecomposition of $M$. In particular, the coefficient $\lambda_i'$ is unique when 
\[ \fa_i^{\otimes 2}\notin \Span( \{ \fa_1^{\otimes 2},\ldots, \fa_{i-1}^{\otimes 2}, \fa_{i+1}^{\otimes 2}, \fa_r^{\otimes 2},\fb_1^{\otimes 2},\ldots,\fb_\ell^{\otimes 2} \}). \] 
For generic $\fa_i$ and $\fb_j$, this holds provided $r+\ell$ is at most ${p+1\choose 2}$, the dimension of the space of $p\times p$ symmetric matrices.
Inequalities ${p+1 \choose 2}\leq \lceil \frac{1}{p}{p+3\choose 4} \rceil$ for $p\notin \{ 3,4,5 \}$ and ${p+1 \choose 2}\leq \lceil \frac{1}{p}{p+3\choose 4} \rceil+1$ for $p \in \{ 3,4,5 \}$ hold. Combining the above conditions, Algorithm~\ref{alg:generic a b} is identifiable when $r+\ell \leq {p+1 \choose 2}$ for $p\neq 4$ and $r+\ell \leq 9, r,\ell \neq 8$ for $p=4$.
\end{proof}

In some settings, we assume that the vectors $\fb_1,\ldots,\fb_\ell$ are orthogonal.
In particular,~$\ell\leq p$. 
This assumption is natural for visualization purposes since the projection onto foreground patterns is orthogonal.
In this case, HTD gives an exact decomposition, by Proposition~\ref{prop:orthogonal_case}. 
The identifiability requirements are the same as in Propositions~\ref{prop:iden_cica_tensor} and \ref{prop:iden_alg2}, as follows. 
The identifiability conditions in the two propositions are unchanged under a change of basis by an invertible $p \times p$ matrix. 
When $\ell \leq p$, we can apply a change of basis to $\kappa_4(\bx)$ so that the vectors $\mathbf{b}_1,\ldots,\mathbf{b}_\ell$ become orthogonal. We apply the same change of basis to $\kappa_4(\by)$.

\subsection{cICA for dimensionality reduction}
\label{sec:ranking_bs}

Usual ICA has been used as a tool to project data, see~\cite{domino2018use,geng2020npsa,lim2008cumulant}.
We extend this to cICA.
In practice, the input to cICA consists of samples from the foreground $\bx$ and background $\by$. These samples comprise the foreground data $X \in \RR^{n \times p}$ and the background data $Y \in \RR^{m \times p}$, where $n$ and $m$ are the numbers of samples in the foreground and background datasets, respectively. We then construct the sample cumulants $\kappa_4(\bx)$ and $\kappa_4(\by)$ as follows. 

A dataset of $n$ samples in $\RR^p$ gives a data matrix $X \in \RR^{n \times p}$. Its fourth cumulant is computed as follows. Let $\bar{X}\in \RR^p$ denote the mean vector over all observations. The  $p \times p$ sample covariance matrix $\Sigma$ for $X$ has entries \( \sigma_{ij} = \frac{1}{n} \sum_{t=1}^{n} (X_{ti} - \bar{X}_i)(X_{tj} - \bar{X}_j)\). The fourth-order central sample moment is a $p \times p \times p \times p$ tensor with entries \( M_{ijkl} = \frac{1}{n} \sum_{t=1}^{n} (X_{ti} - \bar{X}_i)(X_{tj} - \bar{X}_j)(X_{tk} - \bar{X}_k)(X_{tl} - \bar{X}_l).\) Entry $(i,j,k,l)$ of the fourth-order sample cumulant is \( M_{ijkl} - \sigma_{ij}\sigma_{kl} - \sigma_{ik}\sigma_{jl} - \sigma_{il}\sigma_{jk}. \) If the data $X$ are samples from a distribution $\bx$, this sample cumulant approximates $\kappa_4(\bx)$. The computation for $\kappa_4(\by)$ is similar. 

When $p$ is large, forming the fourth cumulants may be prohibitively expensive. 
To get around this, one can reduce the dimension before forming the cumulants, as follows. 
We combine the foreground and background datasets to form a single dataset, a matrix of size $(m + n) \times p$. 
Let $U \in \RR^{p \times k}$ have as its columns the top $k$ principal components of this combined data. The background and foreground transformed variables are then
\begin{equation}
\label{eqn:cica_pca}
     U\T A \mathbf{z}\qquad \text{and} \qquad  U\T A \mathbf{z}' + U\T B \mathbf{s},
\end{equation}
respectively,
where $U\T A \in \RR^{k \times r}$ and $U\T B \in \RR^{k \times \ell}$.
The recovered foreground patterns from cICA are the columns of $U\T B$. 
The columns of $U U\T  B \in \RR^{p \times \ell}$ convert these projected foreground patterns back into the original space. 

In practice, for our data visualization in Section~\ref{sec:visualization}, 
we choose the number $k$ of PCA components to be 30 or the number of components that explains at least $90\%$ variance, whichever comes first.

We compute the mixing matrix $B \in \RR^{p \times \ell}$ with columns $\fb_1,\ldots,\fb_\ell$ using Algorithm~\ref{alg:generic a b}. 
 When employing cICA for dimensionality reduction, we project the foreground data $X$ onto $X B$. 
 For a two-dimensional plot, we plot the projections $(X\fb_i, X \fb_j)$ for a pair $i, j$. To select the most relevant vectors out of our $\ell$ recovered vectors $\fb_i \in \RR^\ell$, we order them by the ratio \begin{equation} \label{eqn:kb} k(\fb) := \frac{\fb\T \kappa_2(\bx) \fb}{\fb\T \kappa_2(\by)\fb}. \end{equation} We justify this ranking and interpret the axes of a cICA dimensionality reduction plot in Section \ref{app: practicalities} of the Appendix.

\subsection{Error Analysis for cICA}\label{sec:error analysis}

Suppose we are in the setting of cICA, where the foreground and background datasets are described by ICA models
$$
\by = A \mathbf{z}, \qquad \bx = A \mathbf{z}' + B \mathbf{s}
$$
and the population cumulant tensors are
$$
\kappa_4(\by) = \sum_{i=1}^r \lambda_i \fa_i^{\otimes 4}, \quad \kappa_4(\bx) = \sum_{i=1}^r \lambda_i' \fa_i^{\otimes 4} + \sum_{i=1}^\ell \nu_i \fb_i^{\otimes 4}.
$$
Let $\hat{\kappa}_4(\by), \hat{\kappa}_4(\bx)$ be the sample cumulant tensors for the two datasets. We prove the following upper bound on the error of estimating $\sum_{i=1}^\ell \nu_i \fb_i^{\otimes 4}$.

\begin{theorem}\label{thm: error for T}
Let
$T = \sum_{i=1}^\ell \nu_i \fb_i^{\otimes 4}$
and let $\hat{T}$ be the tensor obtained after Steps 1 and 2 
of Algorithm \ref{alg:generic a b} with input sample cumulant tensors $\hat{\kappa}_4(\bx), \hat{\kappa}_4(\by)$.
Let
$\rho~= \max_{i\neq j} |\langle \fa_i, \fa_j \rangle|, M_y =\Mat(\kappa_4(\by))
$
and 
$
\Delta_M = \| M_y - \Mat(\hat{\kappa_4}(\by)) \|_2. 
$
Let $\sigma_r(M_y)$ denote the $r$-th largest singular value of $M_y$.
Define
$$
\Delta_A = \frac{\Delta_M}{\sigma_r(M_y) - \Delta_M}, \quad \lambda = \min_i |\lambda_i|,\quad \lambda' = \lambda(1-(r-1)\rho).
$$
Under the assumptions that $(r-1)\rho= o(1)$,
that $\Delta_M < \frac{\lambda}{45} + O(\rho)$,
and moreover that $\max_i |\lambda_i'| \frac{2\sqrt{\Delta_A} + 3 \Delta_A}{\lambda'} = o(1)$,
we have
$$
\|\hat{T} - T\|_F \leq \|\hat{\kappa}_4(\bx) - \kappa_4(\bx)\|_F + \beta\sqrt{\Delta_M} + O(\Delta_M),
$$
where $\beta= \sum_{i=1}^r(|\lambda_i'|\sqrt{\frac 2 {\lambda'}} +  |\lambda_i'^2| 2 \lambda'^{-\frac 3 2}).$
\end{theorem}

\begin{proof}[Sketch Proof]
Let $\fa_i'$ be the estimate of $\fa_i$ obtained via Step 1 of Algorithm \ref{alg:generic a b}, and let $\mu_i$ be the estimate of $\lambda_i'$ via Step 2 of Algorithm \ref{alg:generic a b}.
Then $\|\hat{T} - T\|_F$ is at most
\begin{align}
 \|\hat{\kappa}_4(\bx) - \kappa_4(\bx)\|_F +\sum_{i=1}^r 2|\mu_i| \|\fa_i  - \fa_i' \| + \sum_{i=1}^r |\lambda_i' - \mu_i|,
 \label{eq: T That}
\end{align}
as can be shown using the triangle inequality and by comparing $\|\fa_i^{\otimes 4} - \fa_i'^{\otimes 4}\|$ and $\|\fa_i - \fa_i'\|$ for vectors $\fa_i, \fa_i'$. 
We will obtain bounds on the second and third terms in the sum~\eqref{eq: T That}.

The distances between the numbers $\frac{1}{\lambda_i'}, \frac{1}{\mu_i} $ and between the vectors $\|\fa_i - \fa_i'\|$ 
can be bounded
as
\begin{equation}
\|\fa_i - \fa_i'\| \leq \sqrt{\frac{\Delta_A}{2}},
\quad
\left| \frac{1}{\lambda_i'} - \frac{1}{\mu_i} \right| \leq \frac{2}{\lambda'} \sqrt{\Delta_A} + O(\Delta_A),
\label{eq: a,a', lambda mu}
\end{equation}
by applying results from the study of the optimization landscape of tensor decomposition \cite{kileel2021landscape}.
One can also show that 
\begin{equation}\sigma_r(M) \geq \lambda' = \lambda + O(\rho),\quad
\Delta_A = \frac{\Delta_M}{\lambda'} + O(\Delta_M^2),
\label{eq: delta m delta a}
\end{equation}
by relating $\sigma_r(M_y)$ to $\sigma_r(G_2)$, where $G_2\in \RR^{r\times r}$ is the matrix with $(i,j)$ entry $\langle \fa_i,\fa_j\rangle^2$ and relating
$\sigma_r(G_2)$ to $\rho$.
Substituting \eqref{eq: a,a', lambda mu} and \eqref{eq: delta m delta a} into \eqref{eq: T That}, we obtain the result. 
\end{proof}

We obtain the following end-to-end error bound for recovery of the foreground patterns and its coefficients via cICA, 
by combining Theorem \ref{thm:nearly_orthogonal}, Theorem \ref{thm: error for T}, and sample complexity results for cumulant tensors \cite{anandkumar2014sample}.

\begin{theorem}\label{thm: end to end error}
Suppose we have $N_1$ samples for the background dataset and $N_2$ samples for the foreground dataset.
We can shift and scale our latent variables $z_i,z_i',s_j$ for $i, i' \in [r], j \in [\ell]$, so we assume without loss of generality that 
\begin{itemize}
\item $\mathbb{E}[z_i] = \mathbb{E}[z_i'] = \mathbb{E}[s_j] = 0$,
\item $\mathbb{E}[z_i^2] = \mathbb{E}[z_i'^2] = \mathbb{E}[s_j^2] = 1$.
\end{itemize}
Assume moreover that
the fourth cumulants of $z_i, z_i', s_j$ are nonzero, and
that the variables $z_i, z_i', s_j$ are sub-Gaussian.
Suppose $\fc_i$ 
are the output patterns of the cICA algorithm, with corresponding recovered scalars $\mu_i$, obtained from the tensor of foreground patterns~$T = \sum_{i=1}^\ell\nu_i \fb_i^{\otimes 4}$.
Under the assumptions of Theorem \ref{thm:nearly_orthogonal} and Theorem \ref{thm: error for T}, 
we have
$$
|\nu_i - \mu_i| \leq O(\epsilon^2) + \widetilde{O}(\delta),
$$
$$
\min\{\|\fb_i - \fc_i\|, \|\fb_i + \fc_i\|\} \leq O( \epsilon^2) + \widetilde{O}(\delta)
$$
where
$$
|\langle \fb_i, \fb_j \rangle| \leq \epsilon \quad \text{for all } i\neq j,
$$
$$
\delta =  \widetilde{O}\left( \frac{p^\frac{3}{2} \ell'^2}{N_2} + \sqrt{ \frac{\ell'^4}{N_2} } + \sqrt{\frac{pr'^2}{N_1} + \sqrt{ \frac{r'^4}{p N_1} }}  \right),
$$
$r' = \max\{r,p\}, \ell' = \max\{\ell,p\}$, and
$\widetilde{O}$ absorbs polylog terms.
\end{theorem}

\begin{remark}
The $O(\epsilon^2)$ term in Theorem~\ref{thm: end to end error} captures model mismatch from the non-orthogonality of the true components. The $\widetilde{O}(\delta)$ term is error due to finite sample estimation of foreground patterns. 
Assuming $r$ and $\ell$ are $O(p)$, the $\widetilde{O}(\delta)$ term scales as
$$
\widetilde{O}\left( \frac{p^\frac{7}{2}}{N_2} + \sqrt{ \frac{p^4}{N_2} } + \sqrt{\frac{p^3}{N_1} + \sqrt{ \frac{p^3}{N_1} }}  \right).
$$
We thus obtain a constant accuracy guarantee for recovering the foreground patterns and their coefficients
if the background and foreground sample sizes satisfy
\[
N_1=\widetilde{O}(p^3), \qquad N_2=\widetilde{O}(p^4).
\]
These sample size requirements are beyond the optimal \(O(p^2)\) sample complexity achievable by polynomial-time methods in~\cite{auddy2025large}. 
The gap is due to two steps in our analysis
that introduce dimension-dependent factors: (i) bounding the spectral norm of \(\hat\kappa_4({\by})-\kappa_4(\by)\) by that of its flattening, and (ii) converting between spectral and Frobenius norms for \(\hat\kappa_4({\bx})-\kappa_4(\bx)\).

An interesting direction for future work is to improve the sample efficiency, for instance using the structure of the tensors $\hat\kappa_4({\by})-\kappa_4(\by)$ and $ \hat\kappa_4({\bx})-\kappa_4(\bx)$, by pre-whitening the data, or by decomposing the stacked foreground and background cumulant tensors as a single tensor of size $p \times p \times p \times p \times 2$ to avoid the three-step procedure.
\end{remark}

\section{Numerical experiments}\label{sec:numerical result}

We compare Algorithm \ref{alg:generic a b} with other tensor decompositions and ICA methods to illustrate the necessity of HTD  (Section~\ref{sec:SPM HTD best}).
We investigate the performance of cICA for finding patterns in data (Section~\ref{sec:patterns}) and for data visualization (Section~\ref{sec:visualization}). 
Our code is available on GitHub at 
\url{https://github.com/QWE123665/cICA}.

\subsection{Choices of Methods in Algorithm \ref{alg:generic a b}}\label{sec:SPM HTD best}

We evaluate Algorithm~\ref{alg:generic a b}. We compare our method (SPM-HTD) against several alternatives involving SPM \cite{kileel2019subspace}, HTD (Algorithm~\ref{alg:hierarchical}), FastICA \cite{hyvarinen1999fast}, FOOBI \cite{de2007fourth}, and JADE \cite{cardoso1993blind}. The evaluated combinations include SPM-HTD, HTD-HTD, SPM-SPM, SPM-JADE, JADE-HTD, JADE-JADE, FOOBI-HTD, SPM-FOOBI, FOOBI-FOOBI, and FastICA-HTD.  %

Our setup has
three backgroud patterns and two foreground patterns.
The background patterns are three independent uniform random variables. The foreground patterns are two mixtures of beta distributions $0.5 B(2,5) + 0.5 B(5,4)$.
The foreground mixing matrix $B\in \RR^{5\times 2}$ consists of the last two columns of the identity matrix $I_5$. The background mixing matrix $A\in \RR^{5\times 3}$ is randomly generated and adjusted to ensure small inner products with columns of $B$.

We generate foreground and background datasets, each with 200 samples. Their projections to the leading two principal components are the first two subplots of Figure~\ref{fig:four clusters}. Projecting the foreground dataset via matrix 
$B$ reveals four distinct clusters, see the top-right subplot of Figure~\ref{fig:four clusters}.

We illustrate the performance of our algorithm SPM-HTD and the variants SPM-SPM, HTD-HTD in the second row of Figure~\ref{fig:four clusters}. SPM-HTD is the only method of the three to recover the four clusters. 
The performance of the other competing methods is in Section \ref{app: algorithm choice} of the Appendix. All methods that find the four clusters use an ICA or tensor decomposition method in Step 1 and HTD in Step 3.

\begin{figure}
    \centering
    \includegraphics[width=0.8\linewidth]{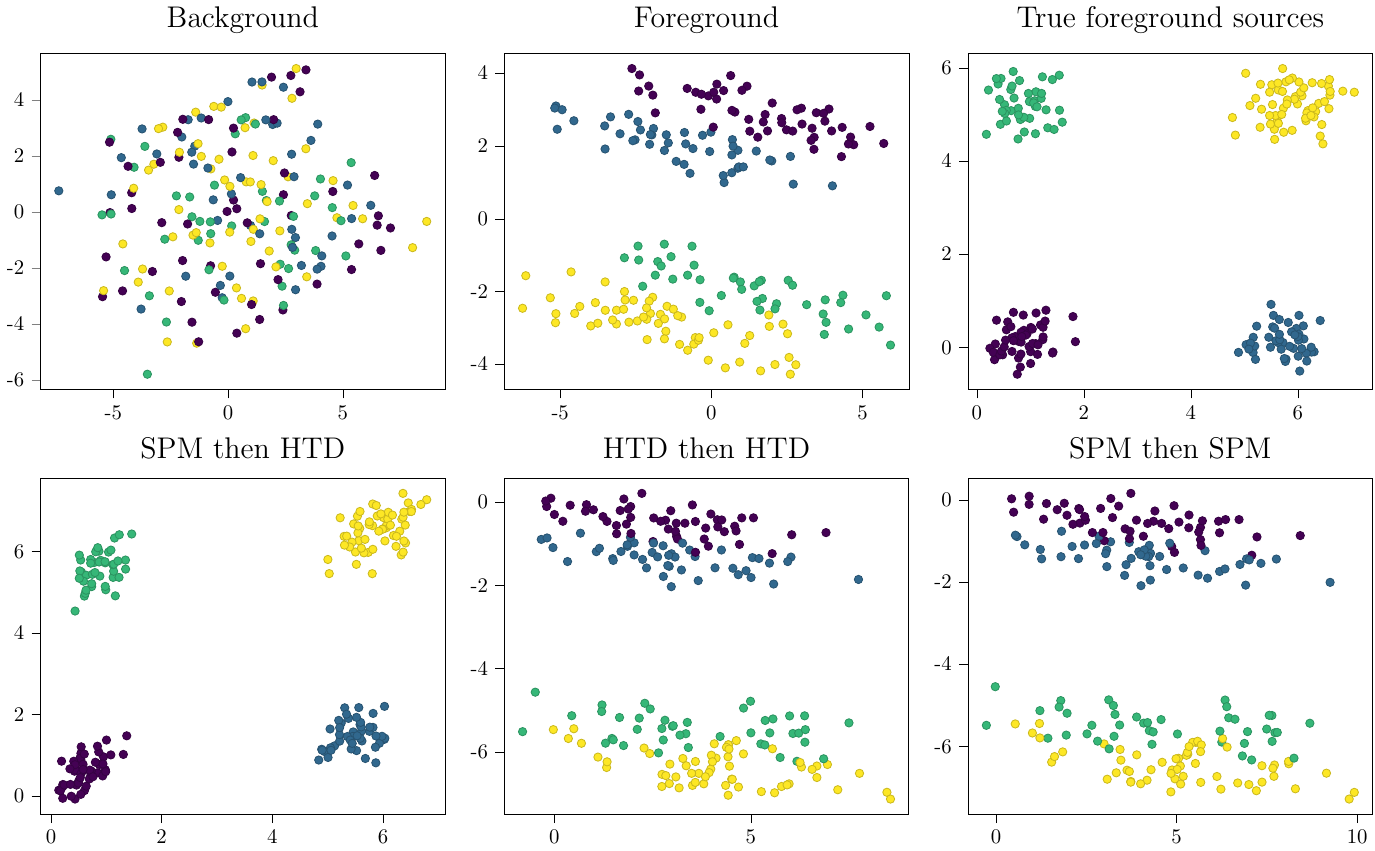}
    \caption{
  We compare our algorithm SPM-HTD against other ICA and tensor decomposition methods in a synthetic setting to to justify our algorithmic choices. 
  The top-left and top-middle subplots illustrate the background and foreground datasets, each consisting of 200 samples in $\mathbb{R}^5$, projected onto their two leading principal components. 
The top-right subplot shows the foreground dataset projected onto the true foreground mixing matrix $B\in\mathbb{R}^{5\times 2}$, revealing four clusters. 
In the bottom row, we compare our algorithm (SPM-HTD) with applying HTD in both Steps 1 and 3, and applying SPM in both steps. 
Only our method (SPM-HTD, bottom-left) recovers the four clusters.}
    \label{fig:four clusters}
\end{figure}

We vary the sample size of both datasets from 100 to 1000. 
For each sample size, we repeat the experiment 20 times by randomly drawing datasets, applying all~11 methods to estimate the matrix $B$, and computing the silhouette score on the foreground data projected via the estimated $B$. A higher silhouette score indicates better recovery of the four clusters. To mitigate randomness, we record the best silhouette score from 20 independent runs for each method and then average these across experiments.

Figure~\ref{fig:SPM-HTD} compares silhouette scores for methods that apply an ICA or tensor decomposition approach in the first step followed by HTD (JADE-HTD, SPM-HTD, FOOBI-HTD, FastICA-HTD) to methods that do not use HTD in the third step.
It shows that methods using tensor decomposition or an ICA approach followed by HTD achieve superior silhouette scores, highlighting the importance of HTD in Step~3.
The HTD-HTD method underperforms approaches combining another tensor decomposition method with HTD, revealing the necessity of an accurate decomposition in Step 1. 
The best choice in Step 1 cycles between FOOBI, FastICA and SPM. We choose SPM for compatibility with Step 2. FastICA does not directly process cumulant tensors, making it unsuitable for Step 2.

\begin{figure}
    \centering
    \includegraphics[scale = 0.6]{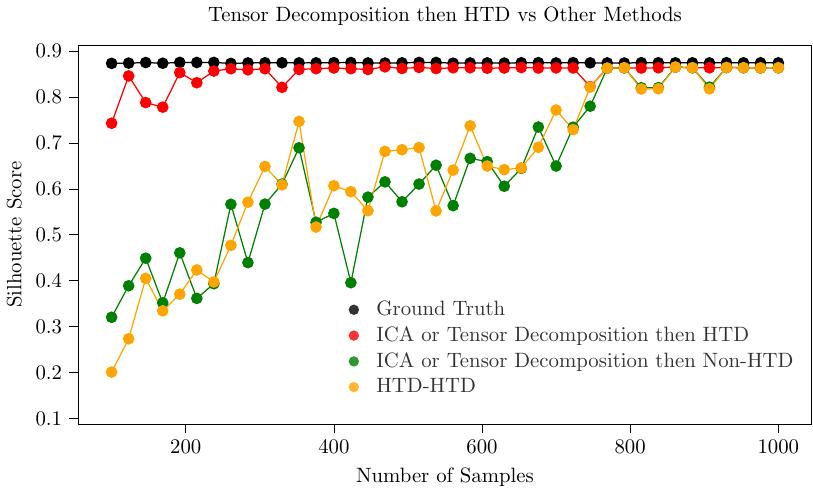}
    \caption{We study the accuracy of different approaches to cICA as the number of samples varies. We compare methods using an ICA or tensor decomposition method followed by HTD (red), ICA or tensor decomposition methods followed by non-HTD alternatives (green) and HTD-HTD (yellow). Performance is evaluated using the silhouette score, which measures how effectively the estimated matrix recovers the four clusters shown in the top-right plot of Figure~\ref{fig:four clusters}.
Methods using ICA or tensor decomposition method followed by HTD outperform both non-HTD approaches and the HTD-HTD combination. This justifies our decision to use SPM in Step 1 and HTD in Step 3 of our cICA algorithm.}
    \label{fig:SPM-HTD}
\end{figure}

\subsection{Salient patterns}
\label{sec:patterns}

The cICA patterns are the foreground vectors $\fb_i$. We investigate the interpretability of the cICA patterns on synthetic, semi-synthetic, and real-world datasets.
We demonstrate that cICA recovers foreground patterns accurately for synthetic data, with comparisons to cPCA~\cite{abid2017contrastive} and PCPCA~\cite{li2020probabilistic}. 
Our semi-synthetic setup has background dataset consisting of images of grass and clouds from \cite{deng2009imagenet}. The foreground dataset consists of digits $0$ and $1$ superimposed, with varying intensity, onto images of grass and clouds. We find that, unlike other methods, cICA is able to recover as top two foreground patterns the digits~$0$ and $1$.
Additionally, we apply cICA to gene expression data from \cite{suresh2023comparative}, using monkey gene expression as the background and human gene expression as the foreground.
We compare the cICA foreground patterns to results to identify genes responsible for human evolution.

\subsubsection{Synthetic data}
\label{sec:simulated_patterns}

We use synthetic data to assess the accuracy of the patterns recovered by cICA. 
We compare against cPCA and PCPCA, illustrating that cICA algorithms recover the foreground patterns more accurately when generated under a model~\eqref{eqn:cica} that assumes independence of latent variables, see Figure~\ref{fig:general setting}. The details of the simulations are in Section \ref{app: synthetic} of the Appendix.

\begin{figure}[htb]
\centering
\includegraphics{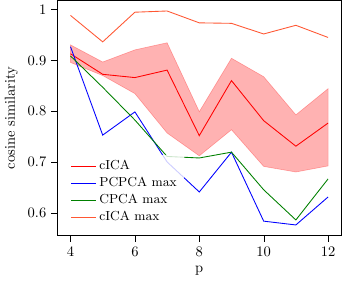}
\includegraphics{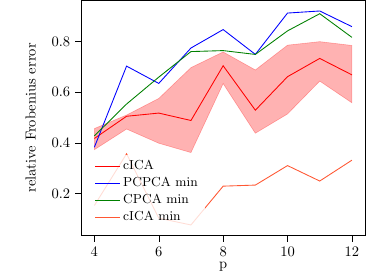}
\caption{
The similarity of the recovered vs. true foreground patterns (i.e. the accuracy of recovering matrix \(B\)), measured via 
cosine similarity (top)
and relative Frobenius error (bottom), via cICA in Algorithm~\ref{alg:generic a b}. The interquartile range over 100 runs is shaded in red, with the best run shown as the red line. 
For cPCA and PCPCA, we 
test 100 hyperparameter values and plot the one with the lowest error. 
}
\label{fig:general setting}
\end{figure}

We see from Figure~\ref{fig:general setting} that cICA outperforms cPCA and PCPCA in recovering the foreground patterns.
Figure~\ref{fig:general setting}(top) shows that the interquartile range for cICA in Algorithm~\ref{alg:generic a b} is above the maximum cosine similarity results for cPCA and PCPCA. The best performing cICA has cosine similarity above 0.9 for all tested $p$. 
Figure~\ref{fig:general setting}(bottom) shows analogous results with accuracy measured via relative Frobenius norm. 
The variability as $p$ changes is due to randomness in the matrix $A$.
The method outperforms cPCA and PCPCA, with the added benefit that no selection of hyperparameters is necessary.

\subsubsection{Corrupted MNIST dataset with continuous strength}
\label{sec:mixed corrupted mnist}
We superimpose hand-written digits $0$ and $1$ from MNIST \cite{deng2012mnist} onto grass and cloud images from \cite{deng2009imagenet}. The background dataset consists of 5000 cloud images and 5000 grass images. 
For the foreground dataset, we sample 8000 grass and 2000 cloud images.
Next, we sample 10000 pairs of images of digits $0$ and $1$ and superimpose them on the foreground grass and cloud images with independent strength following $\text{Uniform}[0,1]$. 
Digits 0 and 1 images are expected to be the foreground patterns. The background patterns come from decomposing grass and cloud images and the ratios of grass and cloud images in the background versus foreground reflects that the coefficient of the background signals may not be proportional, which often happens in reality.
That is, the foreground-to-background ratio $\lambda_i'/\lambda_i$ from equation \eqref{eqn:tensor_decomp} would be  
$0.4 = 2000/5000$ for a patterns in the clouds and $1.6=8000/5000$ for a pattern in the grass.
Samples of the foreground and background images are shown in Figure \ref{fig: MNIST foreground continuous}.

\begin{figure}[htbp]
    \centering
    \includegraphics[scale=0.2]{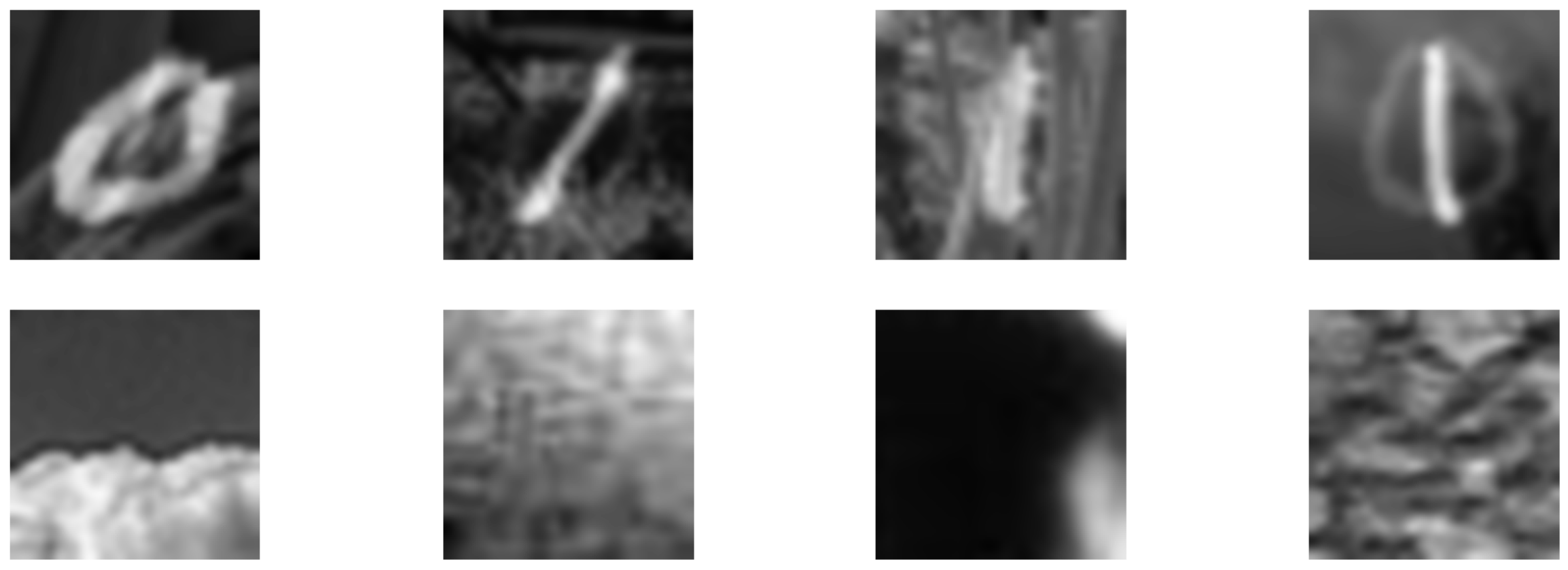}
    \caption{
    Foreground (top) and background images (bottom) for the corrupted MNIST dataset.
    } 
    \label{fig: MNIST foreground continuous}
\end{figure}

To interpret the cICA patterns, we plot the vectors as grayscale images. 
We expect the images from the top two cICA patterns to look like 0 and 1. 
We also plot the top two images for cPCA and PCPCA for comparison, see Figure \ref{fig:patterns as images}. 
The cICA images most closely resemble the images obtained from averaging the sampled digits $0$ and $1$ images. %weighted by uniform strength. 
In the other methods, one component is a combination of $0$ and $1$.
For details, see Section~\ref{app:corrupted MNIST continuous} of the Appendix.

\begin{figure}[htbp]
    \centering
    \includegraphics[width = 0.22\linewidth]{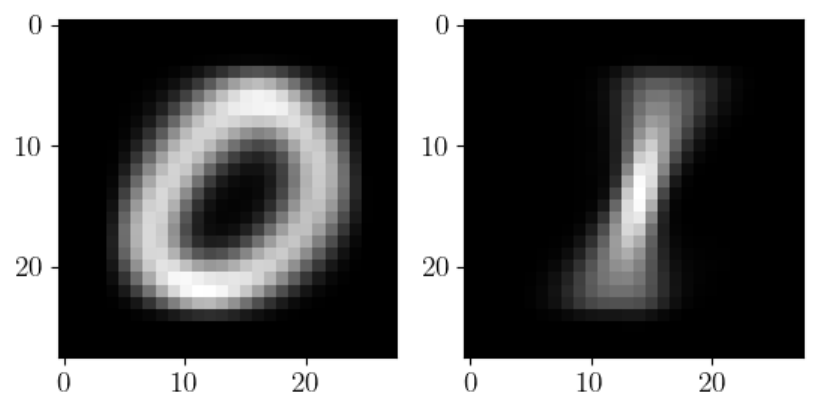} \quad 
    \includegraphics[width=0.22\linewidth]{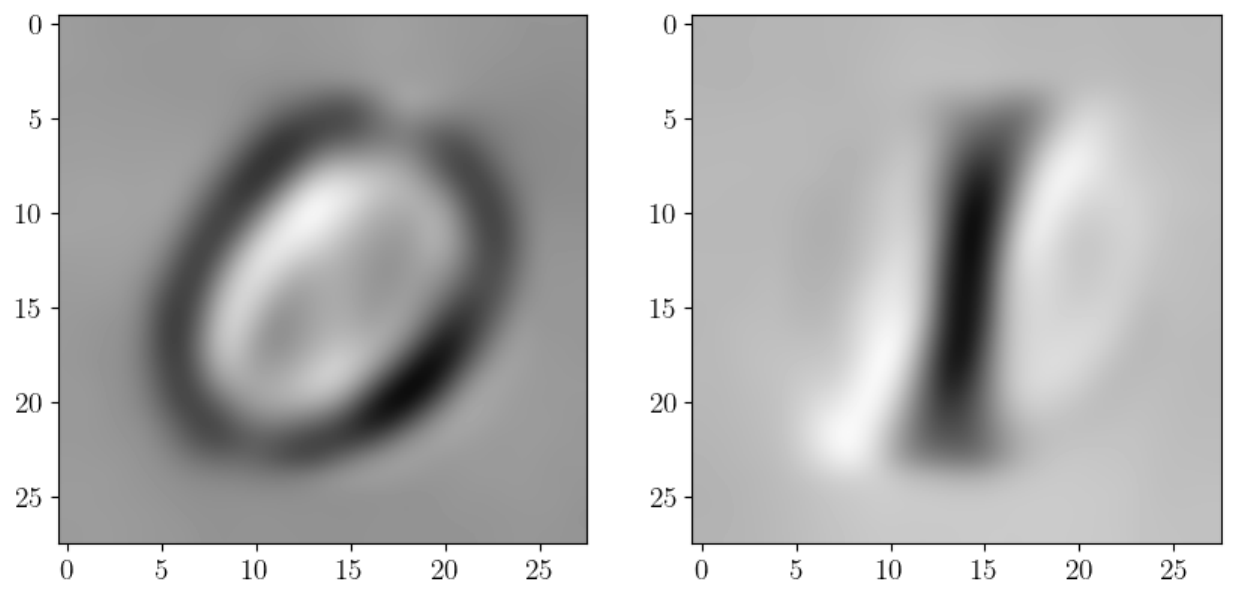}  \quad 
    \includegraphics[width = 0.22\linewidth]{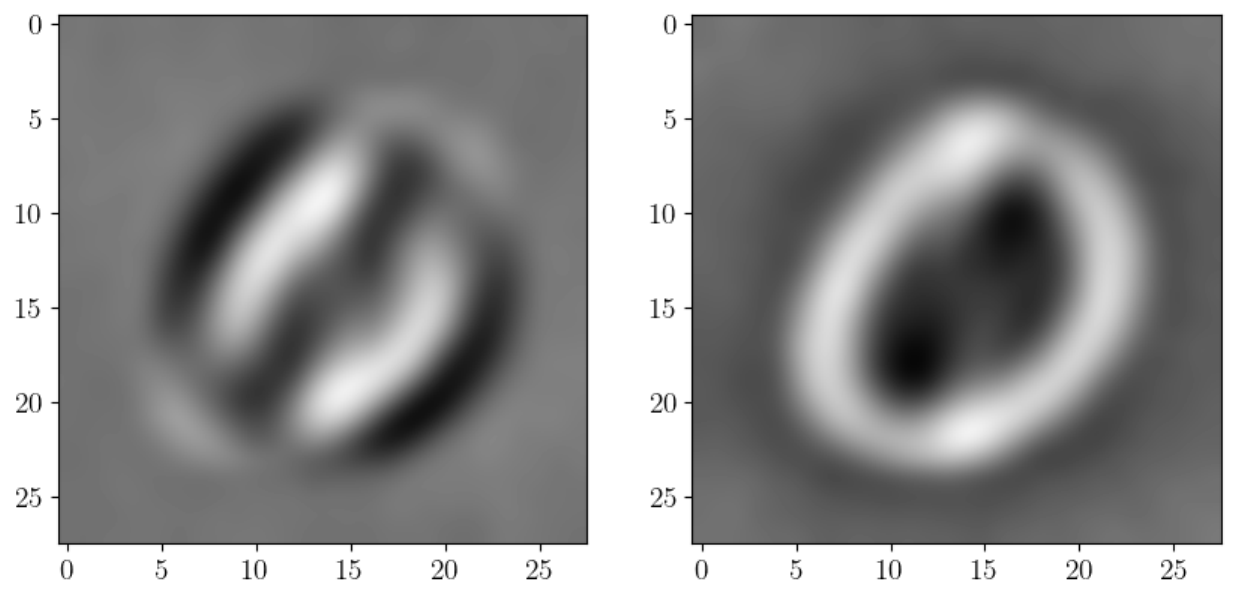} \quad 
    \includegraphics[width=0.22\linewidth]{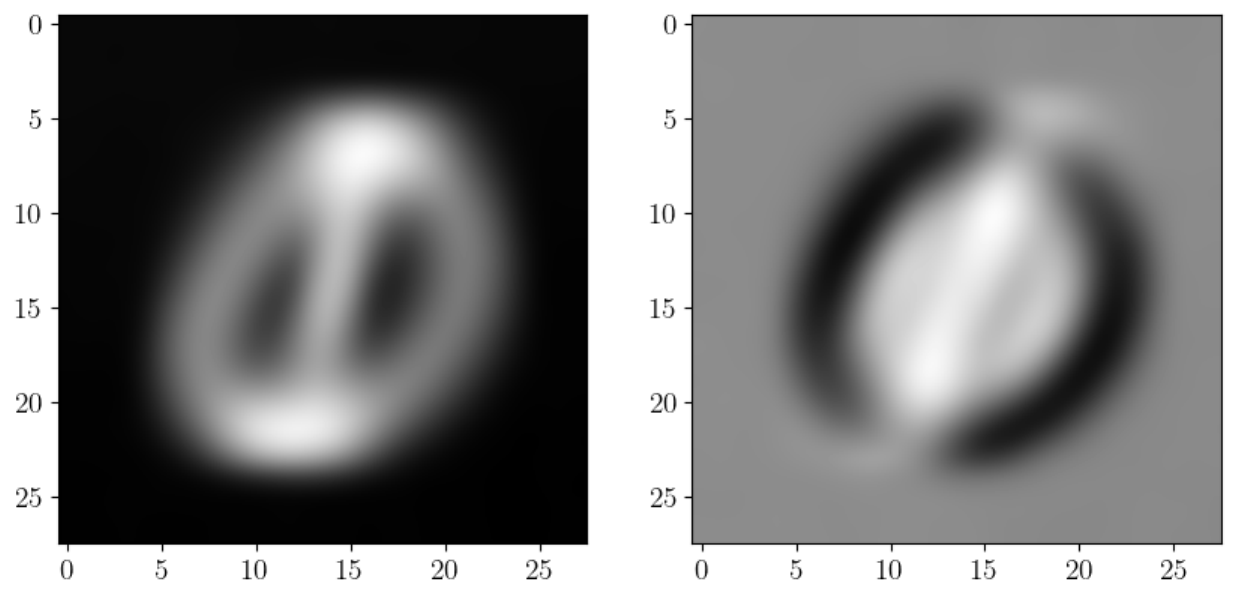}
    \caption{Average images for digits 0 and 1 (first two images). 
    Patterns recovered for cICA (second two), cPCA (third two) and PCPCA (fourth two).}
    \label{fig:patterns as images}
\end{figure}

\subsubsection{Human and monkey gene expression data}
\label{sec:human_and_monkey}

We apply cICA to a dataset of human and monkey gene expression from~\cite{suresh2023comparative}, in which the authors analyze human, chimp, gorilla, macaque, and marmoset datasets to identify genes that are responsible for evolutionary change. Out of 14131 genes, they identify 3383 genes with extensive differences between human and non-human primates, of which they identify a subset of 139 with deeply conserved co-expression across all non-human animals, and strongly divergent co-expression relationships in humans. 

The idea is that the foreground patterns 
should be gene modules (considered as linear combinations of genes) that contribute to the human dataset but not the monkey dataset. By analogy to the MNIST dataset in the previous subsection, the foreground gene modules correspond to the digits 0 and 1. We evaluate the quality of the foreground patterns by testing its consistency with \cite{suresh2023comparative}.

We select the 15 most variable genes among the 139 selected genes and the 15 most variable genes among the other $3244 = 3383 - 139$ genes. 
We combine 10000 chimp and 10000 gorilla data points to form the background dataset $Y \in \RR^{20000 \times 30}$ and 10000 human gene expression data points for the foreground dataset $X \in \RR^{10000 \times 30}$.
Then we apply cICA as in Algorithm~\ref{alg:generic a b} and use~\eqref{eqn:kb} to order the $\fb_i$ and extract the first two vectors $\fb_1, \fb_2 \in \RR^{30}$.
We observe
that the 15 genes with the highest absolute values in $\fb_1$ (resp. $\fb_2$) have 10 (resp. 13) genes among the 15 selected genes that come from the subset of 139 in~\cite{suresh2023comparative}. 
This demonstrates consistency with the results from~\cite{suresh2023comparative}: the vectors $\fb_i$ assign higher weights to the genes from the subset of 139. In comparison, cPCA identifies 9 and 10 genes in its first two patterns and PCPCA identifies 10 and 11 genes.

We also report the number of genes misclassified by the methods, the size of the intersection of the 
$3244 = 3383 - 139$ evolution-irrelevant genes with the
two sets of 15 genes in the foreground patterns (those with largest absolute values for $\fb_1,\fb_2$).
The result can be found in Table~\ref{table: monkey human}. 
We see that cICA outperforms the other methods, with more recovered genes and fewer misclassified genes.
The details of the experiments are in Section \ref{app:monkey_and_human} of the Appendix.

\begin{table}[htbp]
\centering
\begin{tabular}{|c|c|c|c|}
\hline
 method   & $\#$ misclassified genes\\
 \hline
cICA   &5 \\
\hline
ICA  & 7 \\
\hline 
PCPCA & 7\\
\hline
cPCA & 9\\
\hline 
PCA & 9\\
\hline
\end{tabular}

\caption{Number of genes misclassified for the human-monkey gene expression data.}
\label{table: monkey human}
\end{table}

\subsection{Dimensionality reduction}
\label{sec:visualization}

We use cICA for dimensionality reduction and data visualization, as described in Section~\ref{sec:ranking_bs}.
We investigate the performance of cICA on two datasets: mouse protein expression and corrupted MNIST images with discrete strength. Additional numerical experiments on transplant gene expression data are in Section~\ref{app: RNA} of the Appendix.
We quantify the performance of the methods using the silhouette score~\cite{rousseeuw1987silhouettes} of the projected data; 
higher values indicate better clustering of points.

\subsubsection{Mouse protein data}

We study the mouse protein dataset from \cite{higuera2015self}. The foreground data measure protein expression in the cortex of mice subjected to shock therapy, some of whom have Down syndrome. The
background dataset consists of protein expression measurements from mice without
Down Syndrome who did not receive shock therapy. We compare cICA, ICA, as well as cPCA and PCPCA. 
All four algorithms can separate the two clusters in the foreground data, corresponding to mice with Down syndrome and those without, though the projections differ: cICA has the highest Silhouette score (0.606), followed by ICA (0.604), then cPCA (0.421), and then PCPCA (0.220), see Figure~\ref{fig:mice_data}. 
We consider the absolute values of the foreground-to-background cumulant ratios $|\lambda'_i/\lambda_i|$, for $\lambda_i,\lambda_i'$ defined in equation \eqref{eqn:tensor_decomp}.
For $\fa_1, \ldots, \fa_r$, these range from $1.3 \times 10^{-4}$ to $0.12$. 
Moreover, the foreground cumulants for $\fa_1,\ldots,\fa_r$ are in the range [0.1,30] while the foreground cumulants for $\fb_1,\ldots,\fb_5$ are much larger (in the range [200,10000]).
This implies that the background patterns are not obvious in the foreground dataset $X$ and explains the small difference between the experimental results for cICA and ICA.
See Section~\ref{app:mouse} of the Appendix for details.

\begin{figure}[htbp]
    \centering
    \subfigure[]{\includegraphics[scale = 0.61]{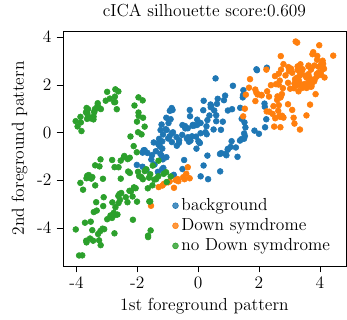} }
    \subfigure[]{\includegraphics[scale = 0.61]{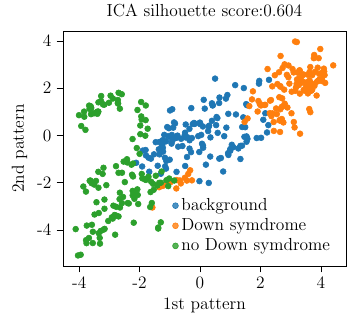} }
    \subfigure[]{\includegraphics[scale = 0.61]{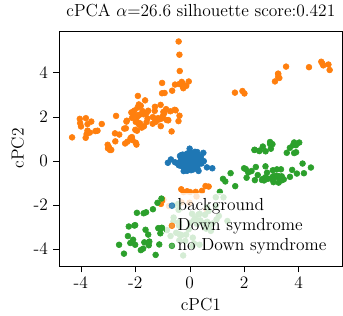}}
    \subfigure[]{\includegraphics[scale = 0.61]{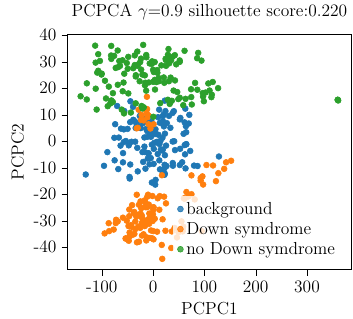}}
    \caption{Dimensionality reduction of the mouse protein data~\cite{higuera2015self} via (a) cICA (b) ICA (c) cPCA (d) PCPCA. 
    For (a), we fix a random seed. For (b), (c), and (d), we plot the projection with the best silhouette score over 100 hyperparameter values.
    }
    \label{fig:mice_data}
\end{figure}

\subsubsection{Corrupted MNIST data with discrete strength}
We superimpose hand-written digits 0 and 1 from MNIST \cite{deng2012mnist} onto grass and cloud images from \cite{deng2009imagenet}. The background dataset consists of 5000 cloud images and 5000 grass images. 
For the foreground dataset, we sample 8000 grass and 2000 cloud images to create different foreground-to-background cumulant ratios for $\lambda_i'/\lambda_i$ in equation \eqref{eqn:tensor_decomp}. 
Similar to the corrupted MNIST data with continuous strength, we expect a ratio of $0.4$, while for grass images, we expect a ratio of $1.6$.
Next, we sample 2500 digit 0, 2500 digit 1 images and form 2500 images consisting of both digit 0 and digit 1. 
We then superimpose 2500 digit 0, 2500 digit 1, and 2500 combined digit 0 and 1 images onto a randomly chosen subset of the background, as shown in the top row of Figure~\ref{fig: MNIST foreground}.
The inclusion of digits 0, 1, both, and none is to make the images of 0 and 1 independent patterns. 
Each image is of size $28\times 28$.

\begin{figure}[htbp]
    \centering
    \includegraphics[scale=0.2]{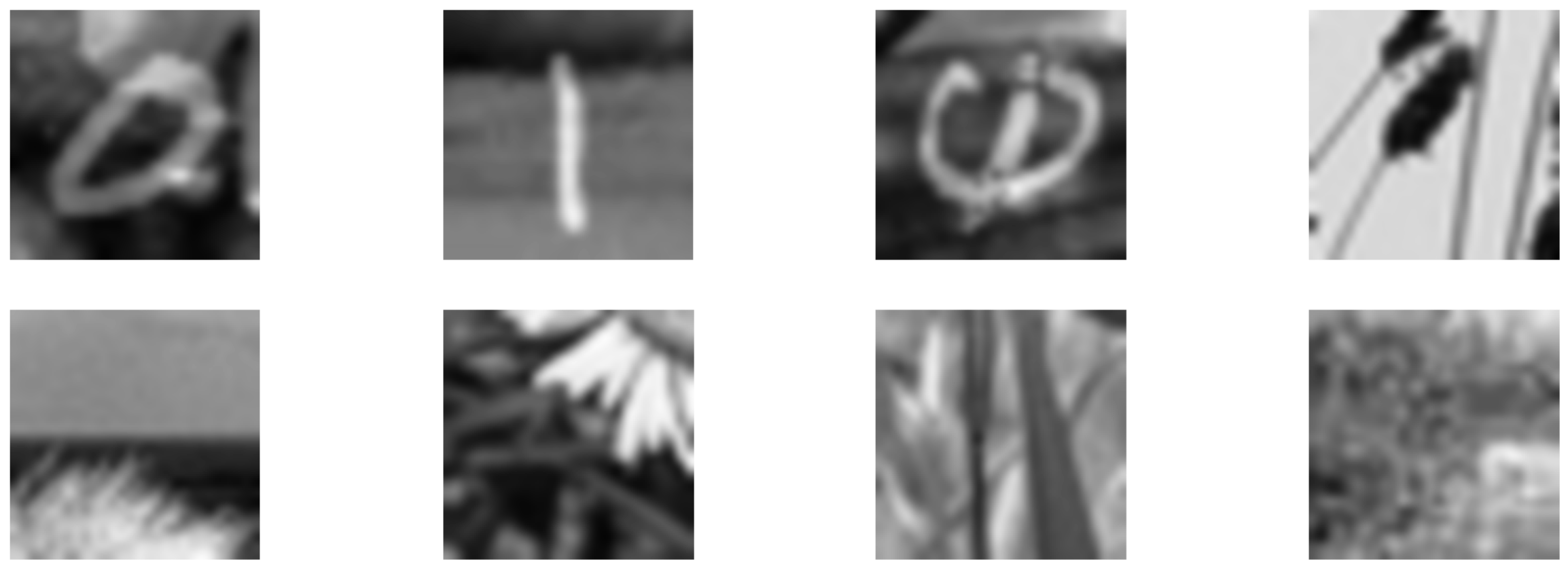}
    \caption{
    Foreground (top) and background images (bottom) for the mixed corrupted MNIST dataset
    } 
    \label{fig: MNIST foreground}
\end{figure}

We plot the 5000 images of digits 0 or 1 superimposed on grass or cloud images using their inner product with the patterns learned in cICA, ICA, cPCA, and PCPCA.
The plots are shown in Figure \ref{fig:mixed_corrupted_silhouette}. 
The algorithm cICA has the highest silhouette score (0.61), followed by cPCA (0.52), then PCPCA (0.44), then ICA (0.30). 
We also report the performance of each of the patterns for classifying the digits 0 or 1 from the corrupted images using the sign of their inner product with the pattern. 
The classification accuracies for cICA, cPCA, and PCPCA are in Table \ref{tab:accuracy each pattern}.
Both foreground cICA patterns can separate the digits 0 and 1 images with more than 0.9 accuracy, while cPCA and PCPCA only have one pattern that achieves this. See Section~\ref{app: mixed mnist} of the Appendix for details.

\begin{figure}[htbp]
    \centering
    \subfigure[]{\includegraphics[height = 3.15cm]{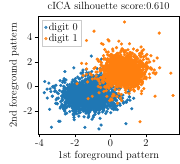}} 
    %\hspace{0.02\linewidth}
    \subfigure[]{\includegraphics[height = 3.15cm]{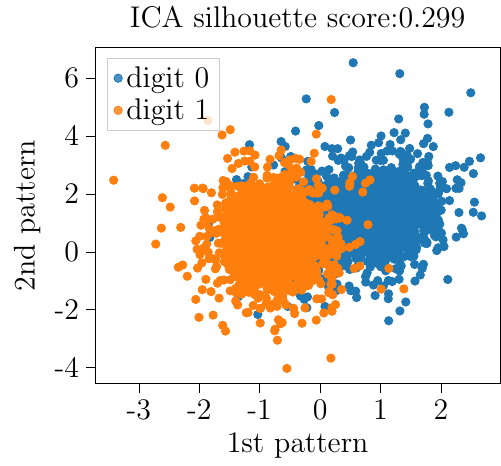}} 
   % \hspace{0.032\linewidth}
    \subfigure[]{\includegraphics[height = 3.15cm]{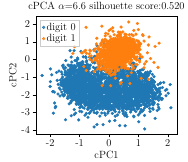}}
  %  \hspace{0.007\linewidth}
    \subfigure[]{\includegraphics[height = 3.15cm]{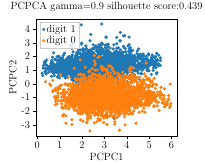}} 
    \caption{Dimensionality reduction plots of the mixed corrupted MNIST data via (a) cICA (b) ICA (c) cPCA (d) PCPCA.
    }
    \label{fig:mixed_corrupted_silhouette}
\end{figure}

\begin{table}[htbp]
    \centering
    \begin{tabular}{|c|c|c|}
    \hline
    method & first pattern (\%) &  second pattern (\%)\\
    \hline
    cICA & 94 & 93\\
    \hline
    cPCA & 71 & 94 \\
    \hline
    PCPCA & 50 & 94 \\
    \hline 
    \end{tabular}
    \caption{Classification accuracies for identifying digits $0$ or $1$ from corrupted images from each of the top two foreground patterns.}
    \label{tab:accuracy each pattern}
\end{table}

% \subsection*{Supporting Information Appendix (SI)}

% Authors should submit SI as a single separate SI Appendix PDF file, combining all text, figures, tables, movie legends, and SI references. SI will be published as provided by the authors; it will not be edited or composed. Additional details can be found in the \href{https://www.pnas.org/authors/submitting-your-manuscript#manuscript-formatting-guidelines}{PNAS Author Center}. The PNAS Overleaf SI template can be found \href{https://www.overleaf.com/latex/templates/pnas-template-for-supplementary-information/wqfsfqwyjtsd}{here}. Refer to the SI Appendix in the manuscript at an appropriate point in the text. Number supporting figures and tables starting with S1, S2, etc.

% Authors who place detailed materials and methods in an SI Appendix must provide sufficient detail in the main text methods to enable a reader to follow the logic of the procedures and results and also must reference the SI methods. If a paper is fundamentally a study of a new method or technique, then the methods must be described completely in the main text.

% \subsubsection*{SI Datasets}

% Supply .xlsx, .csv, .txt, .rtf, or .pdf files. This file type will be published in raw format and will not be edited or composed.

% \subsubsection*{SI Movies}

% Supply Audio Video Interleave (avi), Quicktime (mov), Windows Media (wmv), animated GIF (gif), or MPEG files. Movie legends should be included in the SI Appendix file. All movies should be submitted at the desired reproduction size and length. Movies should be no more than 10MB in size.

\section{Summary}

We have presented contrastive independent component analysis (cICA), a tool to explore patterns and visualize data in one setting relative to another. 
Unlike existing contrastive methods, cICA can model background patterns that each contribute to the foreground in different relative amounts $\lambda_i'/\lambda_i$.
We designed an algorithm for cICA based on
a new hierarchical tensor decomposition (HTD).
The algorithm uses linear algebra to decompose symmetric $p \times p \times p \times p$ tensors of rank at most $p^2$, encouraging orthogonality between rank-1 components.
We use cICA to find salient patterns that describe a foreground dataset relative to a background, testing the results on synthetic, semi-synthetic, and real-world datasets. 
We saw that it can extract foreground patterns of interest and is competitive with other methods. 

We investigated the identifiability of cICA, via the uniqueness of its associated coupled tensor decomposition, seeing improvements relative to cPCA and PCPCA.
This echoes the improved identifiability of ICA over PCA: a general linear mixing can be recovered uniquely via ICA, whereas PCA requires an orthogonal mixing. 

We conclude with two directions for further study.
This cICA model describes observations as a linear mixing of independent latent variables. Dropping the linearity assumption, we may seek patterns that have nonlinear signatures across the observed variables. This would combine the nonlinear contrastive methods of~\cite{abid2019contrastive,severson2019unsupervised,weinberger2022moment,lopez2024toward} with 
approaches to find interpretable patterns, generalizing the vectors $\fb_i$.
Finally, dropping the independence assumption on the latent variables would connect cICA to other latent variable models such as those arising in causal disentanglement~\cite{yang2021causalvae,squires2023linear}.

% \matmethods{Please describe your materials and methods here. This can be more than one paragraph, and may contain subsections and equations as required.

% \subsection*{Subsection for Method}
% Example text for subsection.
% }

% \showmatmethods{} % Display the Materials and Methods section

\subsection*{Acknowledgements} 
We thank Salil Bhate for helpful discussions.
AM and AS were partially supported by the NSF (DMS-2306672 and DMR-2011754).

\bibliographystyle{alpha}
\bibliography{references}

\newcommand{\etalchar}[1]{$^{#1}$}
\begin{thebibliography}{DLDMV01}

\bibitem[ABGO24]{abo2024non}
Hirotachi Abo, Maria~Chiara Brambilla, Francesco Galuppi, and Alessandro Oneto.
\newblock Non-defectivity of {S}egre--{V}eronese varieties.
\newblock {\em Proceedings of the American Mathematical Society, Series B}, 11(51):589--602, 2024.

\bibitem[AGJ14]{anandkumar2014sample}
Animashree Anandkumar, Rong Ge, and Majid Janzamin.
\newblock Sample complexity analysis for learning overcomplete latent variable models through tensor methods.
\newblock {\em arXiv preprint arXiv:1408.0553}, 2014.

\bibitem[AHJ85]{ans1985architectures}
B~Ans, J~H{\'e}rault, and C~Jutten.
\newblock Architectures neuromim{\'e}tiques adaptatives: {D}{\'e}tection de primitives.
\newblock {\em Proceedings of Cognitiva}, 85:593--597, 1985.

\bibitem[AY25]{auddy2025large}
Arnab Auddy and Ming Yuan.
\newblock Large-dimensional independent component analysis: Statistical optimality and computational tractability.
\newblock {\em The Annals of Statistics}, 53(2):477--505, 2025.

\bibitem[AZ19]{abid2019contrastive}
Abubakar Abid and James Zou.
\newblock Contrastive variational autoencoder enhances salient features.
\newblock {\em arXiv preprint arXiv:1902.04601}, 2019.

\bibitem[AZBZ17]{abid2017contrastive}
Abubakar Abid, Martin~J Zhang, Vivek~K Bagaria, and James Zou.
\newblock Contrastive principal component analysis.
\newblock {\em arXiv preprint arXiv:1709.06716}, 2017.

\bibitem[AZBZ18]{abid2018exploring}
Abubakar Abid, Martin~J Zhang, Vivek~K Bagaria, and James Zou.
\newblock Exploring patterns enriched in a dataset with contrastive principal component analysis.
\newblock {\em Nature Communications}, 9(1):2134, 2018.

\bibitem[BMS02]{bartlett2002face}
Marian~Stewart Bartlett, Javier~R Movellan, and Terrence~J Sejnowski.
\newblock Face recognition by independent component analysis.
\newblock {\em IEEE Transactions on neural networks}, 13(6):1450--1464, 2002.

\bibitem[CC02]{chiantini2002weakly}
Luca Chiantini and Ciro Ciliberto.
\newblock Weakly defective varieties.
\newblock {\em Transactions of the American Mathematical Society}, 354(1):151--178, 2002.

\bibitem[CJ10]{comon2010handbook}
Pierre Comon and Christian Jutten.
\newblock {\em Handbook of Blind Source Separation: Independent component analysis and applications}.
\newblock Academic Press, 2010.

\bibitem[Com94]{comon1994independent}
Pierre Comon.
\newblock Independent component analysis, a new concept?
\newblock {\em Signal Processing}, 36(3):287--314, 1994.

\bibitem[COV17]{chiantini2017generic}
Luca Chiantini, Giorgio Ottaviani, and Nick Vannieuwenhoven.
\newblock On generic identifiability of symmetric tensors of subgeneric rank.
\newblock {\em Transactions of the American Mathematical Society}, 369(6):4021--4042, 2017.

\bibitem[CS93]{cardoso1993blind}
Jean-Fran{\c{c}}ois Cardoso and Antoine Souloumiac.
\newblock Blind beamforming for non-{G}aussian signals.
\newblock In {\em IEE Proceedings F (Radar and Signal Processing)}, volume 140, pages 362--370. IET, 1993.

\bibitem[DDS{\etalchar{+}}09]{deng2009imagenet}
Jia Deng, Wei Dong, Richard Socher, Li-Jia Li, Kai Li, and Li~Fei-Fei.
\newblock Imagenet: A large-scale hierarchical image database.
\newblock {\em 2009 IEEE conference on computer vision and pattern recognition}, pages 248--255, 2009.

\bibitem[Den12]{deng2012mnist}
Li~Deng.
\newblock The {MNIST} database of handwritten digit images for machine learning research [best of the web].
\newblock {\em IEEE Signal Processing Magazine}, 29(6):141--142, 2012.

\bibitem[DLCC07]{de2007fourth}
Lieven De~Lathauwer, Josphine Castaing, and Jean-Franois Cardoso.
\newblock Fourth-order cumulant-based blind identification of underdetermined mixtures.
\newblock {\em IEEE Transactions on Signal Processing}, 55:2965--2973, 2007.

\bibitem[DLDMV01]{de2001independent}
Lieven De~Lathauwer, Bart De~Moor, and Joos Vandewalle.
\newblock Independent component analysis and (simultaneous) third-order tensor diagonalization.
\newblock {\em IEEE Transactions on Signal Processing}, 49(10):2262--2271, 2001.

\bibitem[Dom18]{domino2018use}
Krzysztof Domino.
\newblock The use of fourth order cumulant tensors to detect outlier features modelled by a t-student copula.
\newblock {\em arXiv preprint arXiv:1804.00541}, 2018.

\bibitem[EK04]{1306473}
J.~Eriksson and V.~Koivunen.
\newblock Identifiability, separability, and uniqueness of linear {ICA} models.
\newblock {\em IEEE Signal Processing Letters}, 11(7):601--604, 2004.

\bibitem[Flu83]{flury1983some}
Bernhard Flury.
\newblock Some relations between the comparison of covariance matrices and principal component analysis.
\newblock {\em Computational Statistics \& Data Analysis}, 1:97--109, 1983.

\bibitem[Flu84]{flury1984common}
Bernhard~N Flury.
\newblock Common principal components in k groups.
\newblock {\em Journal of the American Statistical Association}, 79(388):892--898, 1984.

\bibitem[Flu87]{flury1987two}
Bernhard~K Flury.
\newblock Two generalizations of the common principal component model.
\newblock {\em Biometrika}, 74(1):59--69, 1987.

\bibitem[GW20]{geng2020npsa}
Xiurui Geng and Lei Wang.
\newblock {NPSA}: Nonorthogonal principal skewness analysis.
\newblock {\em IEEE Transactions on Image Processing}, 29:6396--6408, 2020.

\bibitem[Hac12]{hackbusch2012tensor}
Wolfgang Hackbusch.
\newblock {\em Tensor spaces and numerical tensor calculus}, volume~42.
\newblock Springer, 2012.

\bibitem[Har70]{harshman1970foundations}
Richard~A Harshman.
\newblock Foundations of the {PARAFAC} procedure: Models and conditions for an “explanatory” multi-modal factor analysis.
\newblock {\em UCLA Working Papers in Phonetics}, 16(1):84, 1970.

\bibitem[HCO99]{hyvarinen1999fast}
Aapo Hyvarinen, Razvan Cristescu, and Erkki Oja.
\newblock A fast algorithm for estimating overcomplete {ICA} bases for image windows.
\newblock In {\em IJCNN'99. International Joint Conference on Neural Networks. Proceedings (Cat. No. 99CH36339)}, volume~2, pages 894--899. IEEE, 1999.

\bibitem[HGC15]{higuera2015self}
Clara Higuera, Katheleen~J Gardiner, and Krzysztof~J Cios.
\newblock Self-organizing feature maps identify proteins critical to learning in a mouse model of {D}own syndrome.
\newblock {\em PloS one}, 10(6):e0129126, 2015.

\bibitem[HM16]{hyvarinen2016unsupervised}
Aapo Hyvarinen and Hiroshi Morioka.
\newblock Unsupervised feature extraction by time-contrastive learning and nonlinear {ICA}.
\newblock {\em Advances in neural information processing systems}, 29, 2016.

\bibitem[HST19]{hyvarinen2019nonlinear}
Aapo Hyvarinen, Hiroaki Sasaki, and Richard Turner.
\newblock Nonlinear {ICA} using auxiliary variables and generalized contrastive learning.
\newblock In {\em The 22nd International Conference on Artificial Intelligence and Statistics}, pages 859--868. PMLR, 2019.

\bibitem[JA95]{alexander1995poly}
A.~Hirschowitz J.~Alexander.
\newblock Polynomial interpolation in several variables.
\newblock {\em Journal of Algebraic Geometry 4(4) (1995)}, 1995.

\bibitem[JMM{\etalchar{+}}01]{jung2001imaging}
T-P Jung, Scott Makeig, Martin~J McKeown, Anthony~J Bell, T-W Lee, and Terrence~J Sejnowski.
\newblock Imaging brain dynamics using independent component analysis.
\newblock {\em Proceedings of the IEEE}, 89(7):1107--1122, 2001.

\bibitem[KKMP21]{kileel2021landscape}
Joe Kileel, Timo Klock, and Jo{\~a}o M~Pereira.
\newblock Landscape analysis of an improved power method for tensor decomposition.
\newblock {\em Advances in Neural Information Processing Systems}, 34:6253--6265, 2021.

\bibitem[Kol15]{kolda2015symmetric}
Tamara~G Kolda.
\newblock Symmetric orthogonal tensor decomposition is trivial.
\newblock {\em arXiv preprint arXiv:1503.01375}, 2015.

\bibitem[KP19]{kileel2019subspace}
Joe Kileel and Joao~M Pereira.
\newblock Subspace power method for symmetric tensor decomposition and generalized {PCA}.
\newblock {\em arXiv preprint arXiv:1912.04007}, 2019.

\bibitem[KTC24]{kozhasov2024probabilistic}
Khazhgali Kozhasov and Josu{\'e} Tonelli-Cueto.
\newblock Probabilistic bounds on best rank-1 approximation ratio.
\newblock {\em Linear and Multilinear Algebra}, 72(17):3000--3028, 2024.

\bibitem[Lan11]{landsberg2011tensors}
Joseph~M Landsberg.
\newblock {\em Tensors: geometry and applications}, volume 128.
\newblock American Mathematical Society, 2011.

\bibitem[LF22]{lyu2022finite}
Qi~Lyu and Xiao Fu.
\newblock On finite-sample identifiability of contrastive learning-based nonlinear independent component analysis.
\newblock In {\em International Conference on Machine Learning}, pages 14582--14600. PMLR, 2022.

\bibitem[LHH{\etalchar{+}}24]{lopez2024toward}
Romain Lopez, Jan-Christian Huetter, Ehsan Hajiramezanali, Jonathan~K Pritchard, and Aviv Regev.
\newblock Toward the identifiability of comparative deep generative models.
\newblock In {\em Causal Learning and Reasoning}, pages 868--912. PMLR, 2024.

\bibitem[LJE20]{li2020probabilistic}
Didong Li, Andrew Jones, and Barbara Engelhardt.
\newblock Probabilistic contrastive principal component analysis.
\newblock {\em arXiv preprint arXiv:2012.07977}, 2020.

\bibitem[LM08]{lim2008cumulant}
Lek-Heng Lim and Jason Morton.
\newblock Cumulant component analysis: a simultaneous generalization of {PCA} and {ICA}.
\newblock {\em CASTA2008}, 18, 2008.

\bibitem[LNSU18]{li2018orthogonal}
Zhening Li, Yuji Nakatsukasa, Tasuku Soma, and Andr{\'e} Uschmajew.
\newblock On orthogonal tensors and best rank-one approximation ratio.
\newblock {\em SIAM Journal on Matrix Analysis and Applications}, 39(1):400--425, 2018.

\bibitem[McC18]{mccullagh2018tensor}
Peter McCullagh.
\newblock {\em Tensor methods in statistics: Monographs on statistics and applied probability}.
\newblock Chapman and Hall/CRC, 2018.

\bibitem[Rob16]{robeva2016orthogonal}
Elina Robeva.
\newblock Orthogonal decomposition of symmetric tensors.
\newblock {\em SIAM Journal on Matrix Analysis and Applications}, 37(1):86--102, 2016.

\bibitem[Rou87]{rousseeuw1987silhouettes}
Peter~J Rousseeuw.
\newblock Silhouettes: a graphical aid to the interpretation and validation of cluster analysis.
\newblock {\em Journal of Computational and Applied Mathematics}, 20:53--65, 1987.

\bibitem[SCJ{\etalchar{+}}23]{suresh2023comparative}
Hamsini Suresh, Megan Crow, Nikolas Jorstad, Rebecca Hodge, Ed~Lein, Alexander Dobin, Trygve Bakken, and Jesse Gillis.
\newblock Comparative single-cell transcriptomic analysis of primate brains highlights human-specific regulatory evolution.
\newblock {\em Nature Ecology \& Evolution}, 7(11):1930--1943, 2023.

\bibitem[SGN19]{severson2019unsupervised}
Kristen~A Severson, Soumya Ghosh, and Kenney Ng.
\newblock Unsupervised learning with contrastive latent variable models.
\newblock In {\em Proceedings of the AAAI Conference on Artificial Intelligence}, volume~33, pages 4862--4869, 2019.

\bibitem[SHH{\etalchar{+}}06]{shimizu2006linear}
Shohei Shimizu, Patrik~O Hoyer, Aapo Hyv{\"a}rinen, Antti Kerminen, and Michael Jordan.
\newblock A linear non-{G}aussian acyclic model for causal discovery.
\newblock {\em Journal of Machine Learning Research}, 7(10), 2006.

\bibitem[SRK09]{salmi2009sequential}
Jussi Salmi, Andreas Richter, and Visa Koivunen.
\newblock Sequential unfolding {SVD} for tensors with applications in array signal processing.
\newblock {\em IEEE Transactions on Signal Processing}, 57(12):4719--4733, 2009.

\bibitem[SSBU23]{squires2023linear}
Chandler Squires, Anna Seigal, Salil~S Bhate, and Caroline Uhler.
\newblock Linear causal disentanglement via interventions.
\newblock In {\em International Conference on Machine Learning}, pages 32540--32560. PMLR, 2023.

\bibitem[SSDU24]{sturma2024unpaired}
Nils Sturma, Chandler Squires, Mathias Drton, and Caroline Uhler.
\newblock Unpaired multi-domain causal representation learning.
\newblock {\em Advances in Neural Information Processing Systems}, 36, 2024.

\bibitem[WBWL22]{weinberger2022moment}
Ethan Weinberger, Nicasia Beebe-Wang, and Su-In Lee.
\newblock Moment matching deep contrastive latent variable models.
\newblock {\em arXiv preprint arXiv:2202.10560}, 2022.

\bibitem[WS24]{wang2024identifiability}
Kexin Wang and Anna Seigal.
\newblock Identifiability of overcomplete independent component analysis.
\newblock {\em arXiv preprint arXiv:2401.14709}, 2024.

\bibitem[YLC{\etalchar{+}}21]{yang2021causalvae}
Mengyue Yang, Furui Liu, Zhitang Chen, Xinwei Shen, Jianye Hao, and Jun Wang.
\newblock Causal{VAE}: Disentangled representation learning via neural structural causal models.
\newblock In {\em Proceedings of the IEEE/CVF Conference on Computer Vision and Pattern Recognition}, pages 9593--9602, 2021.

\bibitem[YWS15]{yu2015useful}
Yi~Yu, Tengyao Wang, and Richard~J Samworth.
\newblock A useful variant of the {D}avis--{K}ahan theorem for statisticians.
\newblock {\em Biometrika}, 102(2):315--323, 2015.

\bibitem[ZHPA13]{zou2013contrastive}
James~Y Zou, Daniel~J Hsu, David~C Parkes, and Ryan~P Adams.
\newblock Contrastive learning using spectral methods.
\newblock {\em Advances in Neural Information Processing Systems}, 26, 2013.

\bibitem[ZTB{\etalchar{+}}17]{zheng2017massively}
Grace~XY Zheng, Jessica~M Terry, Phillip Belgrader, Paul Ryvkin, Zachary~W Bent, Ryan Wilson, Solongo~B Ziraldo, Tobias~D Wheeler, Geoff~P McDermott, Junjie Zhu, et~al.
\newblock Massively parallel digital transcriptional profiling of single cells.
\newblock {\em Nature Communications}, 8(1):14049, 2017.

\end{thebibliography}

\newpage
\appendix

\section {Comparison of HTD with other tensor decompositions}\label{app:comparison}
\subsection{Comparison of HTD with other hierarchical tensor decompositions}

We compare HTD in Algorithm~\ref{alg:hierarchical} to other hierarchical tensor decompositions. 
The goal of hierarchical tensor decomposition~\cite[Chapter 11]{hackbusch2012tensor} is to efficiently represent a tensor that lives in a high-dimensional space. Given a tensor of order $d$, a hierarchical decomposition is based on a hierarchy of vector spaces given by a dimension partition tree on indices $\{ 1, \ldots, d\}$, such as those in Figure~\ref{tree2}. 

\begin{figure}[htbp] \centering 
\subfigure[]
{\begin{forest}
  [{$\{1,2,\ldots,d\}$}
  [{$\{1\}$}]
  [{$\{2,\ldots,d\}$}
  [{$\{2\}$}]
  [{$\vdots$}
  [{$\{d-1\}$}]
  [{$\{d\}$}]
  ]
  ]
  ]
\end{forest}}
\hspace{15mm}
\subfigure[]
{\begin{forest}
  [{$\{1,2,3,4\}$}
    [{$\{1,2\}$}
      [{$\{1\}$}]
      [{$\{2\}$}]
    ]
    [{$\{3,4\}$}
      [{$\{3\}$}]
      [{$\{4\}$}]
    ]
  ]
\end{forest}}
\caption{The dimension partition trees used in (a) the PARATREE algorithm of~\cite{salmi2009sequential} and (b) our HTD from Algorithm~\ref{alg:hierarchical}.}
\label{tree2}
\end{figure}

Hierarchical tensor representations in~\cite[Chapter 11]{hackbusch2012tensor} start at the leaves of the tree, which are labeled by single indices. One finds subspaces $U_i \subseteq \RR^{n_i}$ such that the tensor is well-approximated by a tensor in the lower-dimensional space $U_1 \otimes \cdots \otimes U_d \subset \RR^{n_1} \otimes \cdots \otimes \RR^{n_d}$. Proceeding from leaves to the root, when two indices $\{ i \}$ and $\{ j \}$ combine to form the subset $\{ i , j \}$, the representation finds a subspace $U_{ij} \subset U_i \otimes U_j$ that well-approximates the tensor. 
This repeats until we have a low-dimensional subspace $U_{1 \cdots d} \subseteq \RR^{n_1} \otimes \cdots \otimes \RR^{n_d}$ such that the tensor $T$ lies in this subspace to reasonable accuracy. Fixing ranks fixes the allowable dimension of the subspaces $U_I$ for the subsets $I \subseteq [d]$ in the tree. See~\cite[Figure 11.1]{hackbusch2012tensor}.

The PARATREE model starts at the root of the tree. For example, if the root is the splitting of $\{ 1, 2, 3 \}$ into $\{ 1 \} \cup \{ 2,3 \}$ (i.e. Figure~\ref{tree2} in the case $d=3$) then one computes a decomposition of the flattened tensor in $\RR^{n_1} \otimes \RR^{n_2 n_3}$ to give a sum $\sum_{i=1}^{r_1} \mathbf{u}_i \otimes \mathbf{x}_i$, with $\mathbf{u}_i \in \RR^{n_1}$ and $\mathbf{x}_i \in \RR^{n_2 n_3}$. The second step is the splitting of indices $\{ 2, 3\} = \{ 2 \} \cup \{ 3\}$. This decomposes each vector $\mathbf{x}_i = \sum_{j=1}^{r_2} \mathbf{v}_{i,j} \otimes \mathbf{w}_{i,j}$, 
where $\mathbf{x}_i \in \RR^{n_2 n_3}$ is viewed as a matrix of size $n_2 \times n_3$. This results in the decomposition 
\begin{equation}
    \label{eqn:paratree_3}
    T = \sum_{i=1}^{r_1} \mathbf{u_i} \otimes \left( \sum_{j=1}^{r_2} \mathbf{v}_{i,j} \otimes \mathbf{w}_{i,j} \right) . 
\end{equation} 
This pattern can be continued for larger $d$, see~\cite[Equation 9]{salmi2009sequential}. 

Our HTD takes a symmetric $p \times p \times p \times p$ tensor as input. We use the dimension partition tree in Figure \ref{tree2}(b). HTD can be viewed as a symmetric analog of the PARATREE model, but differs in that it uses a different dimension partition tree, and leverages the symmetry of the tensor and decomposition to produce a rank $r$ decomposition, rather than the rank $r_1 r_2$ (or, more generally, rank $r_1 \cdots r_{d-1}$) decomposition obtained from~\eqref{eqn:paratree_3}. Compared to the hierarchical tensor representations of~\cite[Chapter 11]{hackbusch2012tensor}, it differs in that the tensor is symmetric and it uses the dimension partition tree from root to leaves rather than leaves to root.

\subsection{Comparison of HTD with other linear algebra-based tensor decompositions}

%We compare HTD in Algorithm~\ref{alg:hierarchical} to other linear algebra-based tensor decompositions.

Jennrich's Algorithm \cite{harshman1970foundations} decomposes an order 3 tensor $T = \sum_{i=1}^r \mathbf{u}_i\otimes \mathbf{v}_i \otimes \mathbf{w}_i$, requiring $\mathbf{u}_1,\ldots,\mathbf{u}_r$ to be linearly independent and $\mathbf{v}_1,\ldots,\mathbf{v}_r$ to be linearly independent. It computes two matrices $M_z = T(:,:,z), M_{z'} = T(:,:,z')$ for random unit norm vectors $z,z'$ and then computes eigendecompositions of $M_zM_{z'}^{+}$ and $M_{z'}M_z^+$. 
The decomposition of $T$ can then be recovered via pairing the eigenvalues of the two eigendecompositions. 
When applying Jennrich's algorithm to an order-4 symmetric tensor, we need to flatten the 3rd and 4th dimensions of the tensor to form an order-3 tensor first. 
It can decompose a symmetric $p \times p \times p \times p$ tensor of rank at most $p$ due to the linear independence requirement and it takes $O(p^4)$ operations, where the most costly step is forming the matrices $M_z$ and $M_{z'}$.

Orthogonal symmetric decomposition \cite{kolda2015symmetric} decomposes a symmetric tensor $T = \sum_{i=1}^r \mathbf{u}_i^{\otimes d}$ where $\mathbf{u}_1,\ldots,\mathbf{u}_r$ are orthogonal. It takes a random $S \in (\RR^{p})^{\otimes (d-2)}$ and computes the eigendecomposition of $T(S,:,:)$. The vectors $\mathbf{u}_1,\ldots,\mathbf{u}_r$ are eigenvectors of the matrix $T(S,:,:)$. As in Jennrich's algorithm, it can also decompose a symmetric tensor in $(\RR^p)^{\otimes 4}$ with rank at most $p$ due to the orthogonal requirement and it takes $O(p^4)$ operations where the most costly step is forming the matrice $T(S,:,:)$.

In comparison, HTD can decompose a symmetric $p \times p \times p \times p$ tensor of rank up to $p^2$. The algorithm has a computational complexity of $O(p^4r)$ for a rank $r$ tensor, due to the complexity of the eigendecomposition of the flattening. 
HTD recovers the orthogonal symmetric decomposition when the tensor is orthogonally decomposable.

\section{Detailed proof of Theorem \ref{thm:nearly_orthogonal}}
\label{app: thm properties of HTD}
\begin{theorem*}[2.4]
Fix vectors $\fb_1, \ldots, \fb_\ell \in \mathbb{R}^p$ with
$
|\langle \fb_i, \fb_j \rangle| \leq \epsilon$ for all $i \neq j$.
Let
$$
T = \sum_{i=1}^\ell \nu_i \fb_i^{\otimes 4},
$$
where $\nu_1 > \cdots > \nu_\ell$, $\ell \leq p$, and $\fb_1^{\otimes 2},\ldots, \fb_\ell^{\otimes 2}$ are linearly independent. 
Fix $\hat{T}$ with
$
\| \hat{T} - T \|_F \leq \delta$. 
Let $\fc_i$ be the output patterns of the HTD algorithm with input tensor $\hat{T}$ and $\mu_i$ the corresponding recovered scalars ordered so that $\mu_1 > \cdots > \mu_\ell$.
Then for any $i \in [\ell]$,
$$
|\nu_i - \mu_i| \leq (2|\nu_i|L + K)\epsilon^2 + \left( \frac{|\nu_i|}{\nu} 2^\frac{5}{2} + 1 \right) \delta + o(\epsilon^2) + o(\delta)
$$
$$
\min\left\{ \|\fb_i - \fc_i\|, \|\fb_i + \fc_i\| \right\} \leq 2^{3/2} L \epsilon^2 + \frac{8 \delta}{\nu} + o(\epsilon^2) + o(\delta).
$$
where
$$
K = \sqrt{8} \sum_{i=1}^\ell |\nu_i|(i-1), \quad L = 2^{3/2} \frac{K}{\nu} + 2\ell-2, \quad \nu = \min_{i \neq j} \{ |\nu_i - \nu_j|, |\nu_i| \}.
$$
\end{theorem*}

We prove Theorem~\ref{thm:nearly_orthogonal}  via the following lemma.

\begin{lemma}\label{lem:close eigendecomp}
Fix $\fb_1, \dots, \fb_\ell \in \mathbb{R}^p$ such that $|\langle \fb_i, \fb_j \rangle| \leq \epsilon$ for all $i \neq j$.
Let $\mathbf{B}_i$ be the vectorization of $\fb_i^{\otimes 2}$.
Define
$
M = \sum_{i=1}^\ell \nu_i \mathbf{B}_i^{\otimes 2}
$.
Then there exists a matrix $M'$ with eigendecomposition
$
M' = \sum_{i=1}^\ell \nu_i  \mathbf{B}_i'^{\otimes 2}
$
such that for all $i \in [\ell]$,
$$
\|\mathbf{B}_i - \mathbf{B}_i'\| \leq 2(\ell-1)\epsilon^2 + O(\epsilon^4) \quad \text{and} \quad 
\|M - M'\|_F \leq \sqrt{8} \sum_{i=1}^\ell |\nu_i|(i-1) \epsilon^2 + O(\epsilon^4).
$$
\end{lemma}

\begin{proof} 

We generate orthogonal vectors via Gram-Schmidt:
$$
\mathbf{B}_j'' = \mathbf{B}_j - \sum_{i=1}^{j-1} \langle \mathbf{B}_i', \mathbf{B}_j \rangle \mathbf{B}_i', \quad \mathbf{B}_j' = \frac{\mathbf{B}_j''}{\|\mathbf{B}_j''\|}.
$$
The vectors $\mathbf{B}_i$ satisfy $\|\mathbf{B}_i\| = 1$ for all $i$ and $\langle \mathbf{B}_i, \mathbf{B}_j \rangle \leq \epsilon^2$ for $i \neq j$.
We will prove by induction on $j$ that
$$
|\langle \mathbf{B}_j', \mathbf{B}_k \rangle| \leq \epsilon^2 + O(\epsilon^4) \quad \text{for all } k > j.
$$
When $j=1$, $\mathbf{B}_1' = \mathbf{B}_1$, so the result follows immediately.
Assume the result is true for $j-1$. Then,
\begin{align*}
|\langle \mathbf{B}_j'', \mathbf{B}_k \rangle| &= | \langle \mathbf{B}_j, \mathbf{B}_k \rangle - \sum_{i=1}^{j-1} \langle \mathbf{B}_i', \mathbf{B}_j \rangle \langle \mathbf{B}_i', \mathbf{B}_k \rangle | \\
&\leq |\langle \mathbf{B}_j, \mathbf{B}_k \rangle| + \sum_{i=1}^{j-1} |\langle \mathbf{B}_i', \mathbf{B}_j \rangle| |\langle \mathbf{B}_i', \mathbf{B}_k \rangle| \\
&\leq \epsilon^2 + (j-1)(\epsilon^2 + O(\epsilon^4))^2 \\
&= \epsilon^2 + O(\epsilon^4).
\end{align*}
The inner product with $\mathbf{B}_j'$ is obtained from that with $\mathbf{B}_j''$ via 
$$
|\langle \mathbf{B}_j', \mathbf{B}_k \rangle| = \frac{|\langle \mathbf{B}_j'', \mathbf{B}_k \rangle|}{\|\mathbf{B}_j''\|},
$$
so we obtain
$$
|\langle \mathbf{B}_j', \mathbf{B}_k \rangle| \leq \frac{\epsilon^2 + O(\epsilon^4)} {\|\mathbf{B}_j\|-\|\mathbf{B}_j-\mathbf{B}_j''\|}\leq \frac{\epsilon^2 + O(\epsilon^4)}{1 - (j-1)\epsilon^2 + O(\epsilon^4)} = \epsilon^2 + O(\epsilon^4),
$$
which proves the inductive step.
By Gram-Schmidt and the triangle inequality
$$
\|\mathbf{B}_j''- \mathbf{B}_j\| = \|\sum_{i=1}^{j-1} \langle \mathbf{B}_i', \mathbf{B}_j \rangle \mathbf{B}_i'\|\leq \sum_{i=1}^{j-1}|\langle \mathbf{B}_i', \mathbf{B}_j \rangle| \leq (j-1)\epsilon^2 + O(\epsilon^4) \leq (\ell-1)\epsilon^2+  O(\epsilon^4).
$$
Thus, we bound the distance between $\mathbf{B}_j'$ and $ \mathbf{B}_j$ via the triangle inequality and $\mathbf{B}_j' = \frac{\mathbf{B}_j''}{\|\mathbf{B}_j''\|}$ by \begin{align*}
\|\mathbf{B}_j' - \mathbf{B}_j\| & \leq \|\mathbf{B}_j'-\mathbf{B}_j'' \|+ \| \mathbf{B}_j''-\mathbf{B}_j\|      \\
& = |\frac{1-\|\mathbf{B}_j''\|}{\|\mathbf{B}_j''\|}| + \| \mathbf{B}_j''-\mathbf{B}_j\|\\
& \leq \frac{\|\mathbf{B}_j-\mathbf{B}_j''\|}{1- \|\mathbf{B}_j-\mathbf{B}_j''\|} + \|\mathbf{B}_j-\mathbf{B}_j''\| \\
& \leq 2(j-1)\epsilon^2 + O(\epsilon^4).
\end{align*}
Finally, we bound the Frobenius norm of the difference between $M$ and $M'$ by \begin{align*}
\|M - M'\|_F &= \| \sum_{i=1}^\ell \nu_i \mathbf{B}_i^{\otimes 2} - \sum_{i=1}^\ell \nu_i \mathbf{B}_i'^{\otimes 2}\|_F \\
&\leq \sqrt{2} \sum_{i=1}^\ell |\nu_i| \|\mathbf{B}_i-\mathbf{B}_i'\| \\
&\leq \sqrt{8} \sum_{i=1}^\ell |\nu_i|(i-1)\epsilon^2 + O(\epsilon^4). \qedhere 
\end{align*}
\end{proof}

\begin{proof}[Proof of Theorem~\ref{thm:nearly_orthogonal}]
Fix
$
M = \sum_{i=1}^r \nu_i \mathbf{B}_i^{\otimes 2}
$
and 
$M' = \sum_{i=1}^r \nu_i \mathbf{B}_i'^{\otimes 2}$
as in Lemma \ref{lem:close eigendecomp}.
Fix $\hat{M} = \Mat(\hat{T})$ and let
$$ \hat{M}= \sum_{i=1}^r \hat{\nu_i} \hat{\mathbf{B}_i}^{\otimes 2}
$$
be its eigendecomposition.
By the triangle inequality and Lemma \ref{lem:close eigendecomp}, we have
$$
\|\hat{M} - M'\|_F \leq \|\hat{M} - M\|_F + \|M - M'\|_F = 
\delta + \|M - M'\|_F \leq \delta + K\epsilon^2 + O(\epsilon^4),
$$
where $K = \sqrt{8} \sum_{i=1}^\ell |\nu_i|(i-1)$.
By Weyl's theorem,
$$
|\nu_i - \hat{\nu_i}| \leq \|\hat{M} - M'\|_{\text{op}} \leq \|\hat{M} - M'\|_F.
$$
By the variant of the Davis-Kahan theorem in \cite{yu2015useful},
$$
\|\hat{\mathbf{B}}_i - \mathbf{B}_i'\| \leq \frac{2^\frac{3}{2} }{\nu}\|\hat{M} - M'\|_F
\quad \text{where} \quad \nu = \min_{j \neq i} \{|v_i|, |v_i - v_j|\} .
$$
Thus, we bound the distance between $\mathbf{B}_i$ and $\hat{\mathbf{B}}_i$, using the triangle inequality, by
\begin{align*}
  \|\mathbf{B}_i - \hat{\mathbf{B}}_i\| & \leq \|\mathbf{B}_i - \mathbf{B}_i'\| + \|\mathbf{B}_i' - \hat{\mathbf{B}}_i\|  \\
  & \leq 2(\ell-1)\epsilon^2 + \frac{2^\frac{3}{2}}{\nu} \delta + \frac{2^\frac{3}{2} K}{\nu} \epsilon^2 + O(\epsilon^4)\\
&= L\epsilon^2 + 2^\frac{3}{2} \frac{\delta}{\nu} + O(\epsilon^4),
\end{align*}
where $L = 2^{3/2} \frac{K}{\nu} + 2\ell-2.$
The top eigenvector of $\Mat(\hat{\mathbf{B}}_i)$ is $\fc_i$. Suppose its eigenvalue is $\alpha$.
The top eigenpair of $\Mat(\mathbf{B}_i)$ is $(\fb_i, 1)$.
Therefore, again by the Davis-Kahan theorem, we have
$$
\min\left\{ \|\fb_i - \fc_i\|, \|\fb_i + \fc_i\| \right\} \leq 2^\frac{3}{2} \|\mathbf{B}_i - \hat{\mathbf{B}}_i\|\leq 2^\frac{3}{2}L\epsilon^2 + 8\frac\delta\nu + O(\epsilon^2).
$$
By Weyl's theorem,
$$
|\alpha - 1| \leq \|\mathbf{B}_i - \hat{\mathbf{B}}_i\|_{op}\leq \|\mathbf{B}_i - \hat{\mathbf{B}}_i\|_F \leq  L\epsilon^2 + 2^\frac{3}{2} \frac{\delta}{\nu} + O(\epsilon^4).
$$
The algorithm of HTD implies
$$
\mu_i = \hat \nu_i\alpha^2.
$$
Hence, we obtain, by the triangle inequality,
\begin{align*}
|\mu_i - \nu_i| &\leq |\mu_i - \hat{\nu}_i| + |\hat{\nu}_i - \nu_i| \\
& \leq |\hat{\nu}_i||1-\alpha^2| + |\hat{\nu}_i - \nu_i|\\
& \leq (|\hat{\nu}_i-\nu_i|+ |\nu_i|) |1-\alpha|(2+|1-\alpha|) + |\hat{\nu}_i - \nu_i|\\
& \leq 2|1-\alpha||\nu_i| + |\hat{\nu}_i - \nu_i| +  o(\epsilon^2) + o(\delta)\\
& \leq 2|\nu_i|L\epsilon^2 + 2^\frac{5}{2} |\nu_i| \frac{\delta}{\nu} + \delta + K\epsilon^2+ o(\epsilon^2) + o(\delta). \qedhere 
\end{align*}
\end{proof}

\section{Detailed proof of Theorem \ref{thm: error for T} and \ref{thm: end to end error}}

\subsection{Proof of Theorem \ref{thm: error for T}}
Suppose we are in the setting of cICA, where the foreground and background datasets are described by ICA models
$$
\by = A \mathbf{z}, \qquad \bx = A \mathbf{z}' + B \mathbf{s}
$$
and the population cumulant tensors are
$$
\kappa_4(\by) = \sum_{i=1}^r \lambda_i \fa_i^{\otimes 4}, \quad \kappa_4(\bx) = \sum_{i=1}^r \lambda_i' \fa_i'^{\otimes 4} + \sum_{i=1}^\ell \nu_i \fb_i^{\otimes 4}.
$$
Let $\hat{\kappa}_4(\by), \hat{\kappa}_4(\bx)$ be the sample cumulant tensors for the two datasets.

\begin{theorem*}
Let
$T = \sum_{i=1}^\ell \nu_i \fb_i^{\otimes 4}$
and let $\hat{T}$ be the tensor obtained after Steps 1 and 2 
of Algorithm \ref{alg:generic a b} with input sample cumulant tensors $\hat{\kappa}_4(\bx), \hat{\kappa}_4(\by)$.
Let
$\rho~= \max_{i\neq j} |\langle \fa_i, \fa_j \rangle|, M_y = \Mat(\kappa_4(\by))
$
and 
$
\Delta_M = \| M_y - \Mat(\hat{\kappa_4}(\by)) \|_2. 
$
Let $\sigma_r(M_y)$ denote the $r$-th largest singular value of $M_y$.
Define
$$
\Delta_A = \frac{\Delta_M}{\sigma_r(M_y) - \Delta_M}, \quad \lambda = \min_i |\lambda_i|,\quad \lambda' = \lambda(1-(r-1)\rho).
$$
Under the assumptions that $(r-1)\rho= o(1)$,
that $\Delta_M < \frac{\lambda}{45} + O(\rho)$,
and moreover that $\max_i |\lambda_i'| \frac{2\sqrt{\Delta_A} + 3 \Delta_A}{\lambda'} = o(1)$,
we have
% $$
% \Delta_A = \frac{\Delta_M}{\lambda} + O(\Delta_M^2)
% $$
% and
$$
\|\hat{T} - T\|_F \leq \|\hat{\kappa}_4(\bx) - \kappa_4(\bx)\|_F + \beta\sqrt{\Delta_M} + O(\Delta_M),
$$
where $\beta= (\sum_{i=1}^r|\lambda_i'|\sqrt{\frac 2 {\lambda'}} +  |\lambda_i'^2| 2 \lambda'^{-\frac 3 2}).$
\end{theorem*}

\begin{proof}
Let $\fa_i'$ be the estimate of $\fa_i$ obtained via Step 1 of Algorithm \ref{alg:generic a b}, and $\mu_i$ be the estimate of $\lambda_i'$ via Step 2 of Algorithm \ref{alg:generic a b}.
We can bound the difference between the true tensor $T$ and the recovered tensor $\hat{T}$ as
\begin{align*}
 &\|\hat{T} - T\|_F\\
 =& \|\hat{\kappa}_4(\bx) - \sum_{i=1}^r \mu_i \fa_i'^{\otimes 4} - \kappa_4(\bx) + \sum_{i=1}^r \lambda_i' \fa_i^{\otimes 4}\|_F \\
\leq& \|\hat{\kappa}_4(\bx) - \kappa_4(\bx)\|_F + \|\sum_{i=1}^r \mu_i (\fa_i^{\otimes 4} - \fa_i'^{\otimes 4}) \|_F + \|\sum_{i=1}^r (\lambda_i' - \mu_i) \fa_i^{\otimes 4}\|_F \\
\leq& \|\hat{\kappa}_4(\bx) - \kappa_4(\bx)\|_F + \sum_{i=1}^r |\mu_i| \|\fa_i^{\otimes 4}  - \fa_i'^{\otimes 4} \| + \sum_{i=1}^r |\lambda_i' - \mu_i|  \\
\leq& \|\hat{\kappa}_4(\bx) - \kappa_4(\bx)\|_F +\sum_{i=1}^r 2|\mu_i| \|\fa_i  - \fa_i' \| + \sum_{i=1}^r |\lambda_i' - \mu_i|,
\end{align*}
where the first two inequalities follow from the triangle inequality and the last inequality follows from
\begin{align*}
&\|\fa_i^{\otimes 4} - \fa_i'^{\otimes 4}\|^2 = 2 - 2 \langle \fa_i, \fa_i' \rangle^4\\
= &2 - 2\left(1 - \frac{1}{2} \|\fa_i - \fa_i'\|^2\right)^4 \\
\leq &2 - 2 + 4 \|\fa_i - \fa_i'\|^2 \quad \text{(using } (1-x)^4 \geq 1-4x \text{ for small } x\text{)} \\
=& 4\|\fa_i - \fa_i'\|^2.
\end{align*}

By \cite[Lemma S.32]{kileel2021landscape}, we have
$\sigma_r(M_y) \geq \lambda \sigma_r(G_2)$, where $G_2\in \RR^{r\times r}$ is the matrix with $(i,j)$ entry $\langle \fa_i,\fa_j\rangle^2.$
By the proof of \cite[Lemma 6]{kileel2021landscape}, we have 
$\sigma_r(G_2)\geq 1-\rho_2$ where $\rho_s = \sup_{\|x\|=1} \sum_{i=1}^r |\langle x, \fa_i\rangle |^s-1$ for $s>0$
and $\rho_s\leq (r-1) \rho^{\lfloor s/2 \rfloor}$.
Thus, we can lower bound $\sigma_r(M_y)$ by 
$$
\sigma_r(M_y) \geq \lambda - \lambda(r-1)\rho =\lambda' = \lambda + O(\rho) .
$$
Let $\tau = \frac{1}{6} - 4 \rho_2- 6\rho_4 = \frac{1}{6}+ O(\rho)$. 
By \cite[Theorem 7]{kileel2021landscape}, if $\Delta_A < \frac{2\tau}{2+4\tau +12}$, 
we can bound the distance between the true component $\fa_i$ and learned component $\fa_i'$ by $$\|\fa_i-\fa_i'\|\leq \sqrt{\frac{\Delta_A}{2}}.$$
The condition is satisfied when $\frac{\Delta_M}{\lambda-\Delta_M+ O(\rho)}\leq \frac{1}{44}+ O(\rho)$. This explains our second assumption $\Delta_M \leq \frac{\lambda}{45}+O(\rho).$

By \cite[Lemma S.31]{kileel2021landscape}, the distance between the numbers $\frac 1 {\lambda_i'}$ and $ \frac 1 {\mu_i}$ is bounded from above by 
\begin{align*}
|\frac 1 {\lambda_i'} - \frac 1 {\mu_i}| &
\leq \frac{\sqrt{8}}{\sigma_r(M_y)}\|\fa_i-\fa_i'\| + \Delta_A (\frac{2}{\sigma_r(M_y)}+ \frac{1}{\sigma_r(M_y)-\Delta_M})\\
&\leq \frac{1}{\sigma_r(M_y)} (2\sqrt{\Delta_A} + 3\Delta_A)\\
&\leq \frac{1}{\lambda'} (2\sqrt{\Delta_A} + 3\Delta_A).
\end{align*}
This implies that 
\begin{align*}
|\lambda_i'-\mu_i| &\leq |\lambda_i' \mu_i|\frac{1}{\lambda'} (2\sqrt{\Delta_A} + 3\Delta_A)\\
& \leq (|\lambda_i'^2| + |\lambda_i'||\lambda_i'-\mu_i|)\frac{1}{\lambda'} (2\sqrt{\Delta_A} + 3\Delta_A).
\end{align*}
Rearranging, we obtain 
$$|\lambda_i'-\mu_i| (1-|\lambda_i'|\frac{1}{\lambda'} (2\sqrt{\Delta_A} + 3\Delta_A)) 
\leq |\lambda_i'^2| \frac{1}{\lambda'} (2\sqrt{\Delta_A} + 3\Delta_A).$$
To obtain an upper bound on $|\lambda_i'-\mu_i|$
from the above inequality, we need 
$|\lambda_i'|\frac{1}{\lambda'} (2\sqrt{\Delta_A} + 3\Delta_A)<1,$  which is our third assumption in the  statement.
Thus, the distance between the true coefficient $\lambda_i'$ of the rank one component $\fa_i^{\otimes 2}$ and the learned coefficient $\mu_i$ is bounded by 
\begin{align*}
|\lambda_i'-\mu_i| &\leq |\lambda_i'^2| \frac{1}{\lambda'} (2\sqrt{\Delta_A} + 3\Delta_A) (1+ |\lambda_i'|\frac{1}{\lambda'} (2\sqrt{\Delta_A} + 3\Delta_A) + O(\Delta_A)) \\
& = |\lambda_i'^2| \frac{2}{\lambda'}\sqrt{\Delta_A} + O(\Delta_A).
\end{align*}
Plugging the bounds on $\|\fa_i'-\fa_i\|$ and $|\lambda_i'-\mu_i|$ into the bound on $\|\hat{T}-T\|_F$, we obtain
\begin{align*}
\|\hat{T}-T\|_F & \leq \|\hat{\kappa}_4(\bx) - \kappa_4(\bx)\|_F +\sum_{i=1}^r 2|\mu_i| \|\fa_i  - \fa_i' \| + \sum_{i=1}^r |\lambda_i' - \mu_i|\\
&\leq \|\hat{\kappa}_4(\bx) - \kappa_4(\bx)\|_F + \sum_{i=1}^r 2 (|\lambda_i'-\mu_i|+ |\lambda_i'|) \sqrt{\frac{\Delta_A}{2}}+ \sum_{i=1}^r |\lambda_i'^2| \frac{2}{\lambda'}\sqrt{\Delta_A} + O(\Delta_A) \\
& \leq \|\hat{\kappa}_4(\bx) - \kappa_4(\bx)\|_F + \sum_{i=1}^r|\lambda_i'|\sqrt{2\Delta_A} + \sum_{i=1}^r |\lambda_i'^2| \frac{2}{\lambda'}\sqrt{\Delta_A} + O(\Delta_A).
\end{align*}
Note that 
$\sigma_r(M_y)\leq \lambda'$ so $\Delta_A = \frac{\Delta_M}{\sigma_r(M_y) - \Delta_M} \leq \frac{\Delta_M}{\lambda'}+ O(\Delta_M^2)$. 
Hence, replacing $\Delta_A$ by $
\Delta_M$, we obtain
\[
\|\hat{T}-T\|_F \leq \|\hat{\kappa}_4(\bx) - \kappa_4(\bx)\|_F + (\sum_{i=1}^r|\lambda_i'|\sqrt{\frac 2 {\lambda'}} +  |\lambda_i'^2| \frac{2}{\lambda'^{\frac 3 2}})\sqrt{\Delta_M} + O(\Delta_M). 
 \qedhere
 \]
\end{proof}

\subsection{Detailed proof of Theorem \ref{thm: end to end error}}

We restate the theorem for convenience.

\begin{theorem*}
Suppose we have $N_1$ samples for the background dataset and $N_2$ samples for the foreground dataset.
We can shift and scale our latent variables $z_i,z_i',s_j$ for $i, i' \in [r], j \in [\ell]$, so we assume without loss of generality that 
\begin{itemize}
\item $\mathbb{E}[z_i] = \mathbb{E}[z_i'] = \mathbb{E}[s_j] = 0$,
\item $\mathbb{E}[z_i^2] = \mathbb{E}[z_i'^2] = \mathbb{E}[s_j^2] = 1$.
\end{itemize}
Assume moreover that
the fourth cumulants of $z_i, z_i', s_j$ are nonzero, and
that the variables $z_i, z_i', s_j$ are sub-Gaussian.
Suppose $\fc_i$ 
are the output patterns of the cICA algorithm, with corresponding recovered scalars $\mu_i$, obtained from the tensor of foreground patterns~$T = \sum_{i=1}^\ell\nu_i \fb_i^{\otimes 4}$.
Under the assumptions of Theorem \ref{thm:nearly_orthogonal} and Theorem \ref{thm: error for T}, 
we have
$$
|\nu_i - \mu_i| \leq O(\epsilon^2) + \widetilde{O}(\delta),
$$
$$
\min\{\|\fb_i - \fc_i\|, \|\fb_i + \fc_i\|\} \leq O( \epsilon^2) + \widetilde{O}(\delta)
$$
where
% $$
% K = \sqrt{8} \sum_{i=1}^\ell |\nu_i|(i-1), \quad L = 2^{3/2} \frac{K}{\nu} + 2\ell-2, 
% $$
$$
\delta =   \frac{p^2 \ell'^2}{N_2} + \sqrt{ \frac{p \ell'^4}{N_2} } + \sqrt{\frac{pr'^2}{N_1} + \sqrt{ \frac{r'^4}{p N_2} }},
$$
% $$
% \nu = \min_{i \neq j} \{|\nu_i - \nu_j|, |\nu_i|\}, \quad
% \ell' = \max\{r+\ell,p\}, \quad r' = \max\{r,p\},
% $$
and
$\widetilde{O}$ absorbs polylog terms.
\end{theorem*}

We prove the theorem via the following lemmas.

\begin{lemma}\label{lem:norm}
Let $A \in \RR^{p\times r}$ be a matrix with columns $\fa_1, \ldots,\fa_r$, where $\|\fa_i\| = 1$ for all $i$, and
$
\max_{i \neq j} |\langle \fa_i, \fa_j \rangle| \leq \rho.
$
Then
$$
\|A\|_2 = 1 + O(\rho).
$$
\end{lemma}
\begin{proof}
Let $C = A\T A$.
For any $v \in \mathbb{R}^r$, we have
$$
\|(C - I_r) v\| \leq \rho \|v\|_1 \leq \sqrt{r} \rho \|v\|,
$$
thus
$$
\|C - I_r\|_2 \leq \sqrt{r} \rho.
$$
Let $\sigma$ be the top eigenvalue of $C$. Then
$
\sigma = \|A\|_2^2.
$
By Weyl's theorem, we have
$$
|\sigma - 1| \leq \|C - I_r\|_2 \leq \sqrt{r} \rho,
$$
so
$
\sigma = 1 + O(\rho),
$
and hence
\[
\|A\|_2 = \sqrt{1 + O(\rho)} = 1 + O(\rho).
\qedhere
\]
\end{proof}

Suppose $T$ is a symmetric tensor in $(\mathbb{R}^p)^{\otimes 4}$.
Its operator norm is
$$
\|T\| = \sup_{\|v_1\| = \|v_2\| = \|v_3\| = \|v_4\| = 1} |T(v_1,v_2,v_3,v_4)|
$$
where
$$
T(v_1, v_2, v_3, v_4) = \sum_{i=1}^p \sum_{j=1}^p \sum_{k=1}^p \sum_{\ell=1}^p T_{ijkl} (v_1)_i (v_2)_j (v_3)_k (v_4)_\ell.
$$

\begin{lemma}\label{lem:norms inequalities}
Suppose $T$ is a symmetric tensor in $(\mathbb{R}^p)^{\otimes 4}$.
Then, we have
$$
\|\Mat(T)\|_2 \leq p \|T\| \quad \text{and} \quad \|T\|_F \leq p^{\frac{3}{2}}\|T\|.
$$
\end{lemma}

\begin{proof}
Let $B\in \mathbb{R}^{p^2}$ such that $\|B\| = 1$ and
$$
B\T \Mat(T) B = \|\Mat(T)\|_2.
$$
The matrix $\Mat(B)$ is symmetric since it lies in the column span of $\Mat(T)$. 
Let
$\Mat(B) = \sum_{i=1}^p \lambda_i \fb_i^{\otimes 2}
$
be its eigendecomposition.
Then, we have
\begin{align*}
B\T \Mat(T) B &= \sum_{i=1}^p \sum_{j=1}^p \lambda_i \lambda_j T(\fb_i, \fb_i, \fb_j, \fb_j)\leq \left( \sum_{i=1}^p \lambda_i \right)^2 \|T\|.
\end{align*}
Note that $\|B\| = 1$, so
$
\sum_{i=1}^p \lambda_i^2 = 1$.
By the AM–GM inequality,
$
|\sum_{i=1}^p \lambda_i| \leq \sqrt{p}$.
Thus 
$$
\|\Mat(T)\|_2 = B\T \Mat(T) B \leq p \|T\|.
$$
The quantity \(\min_{T\neq 0}\frac{\|T\|}{\|T\|_F}\) is the best rank-one approximation ratio, see \cite{li2018orthogonal,kozhasov2024probabilistic}.
For fourth-order tensors of size $p$, we have \(\|T\|_F \le p^{3/2}\|T\|\) since $T$ can be written as a sum of at most $p^3$ tensors whose vectorizations are orthogonal, see \cite[Theorem 3.5]{li2018orthogonal} or \cite[Theorem 1.1]{kozhasov2024probabilistic}.
\end{proof}

\color{black}

We will use the following sample complexity result of ICA from \cite[Theorem 2]{anandkumar2014sample}.

\begin{theorem}\label{thm: sample complexity of ICA}
Consider $N$ samples $x^i = A h^i$, $i \in [N]$, from the ICA model with mixing matrix $A \in \mathbb{R}^{d \times k}$.
Suppose $\|A\| \leq O(1 + \sqrt{k/d})$ and the entries of $h \in \mathbb{R}^k$ are independent subgaussian variables with $\mathbb{E}[h_j^2] = 1$ and constant nonzero 4th order cumulant. Define $m = \max(d,k)$.
For the 4th order cumulant $\kappa_4$ in (8) and its empirical estimate $\hat{\kappa}_4$, if $n \geq d$, we have with high probability
$$
\|\hat{\kappa}_4 - \kappa_4\| \leq \widetilde{O}\left( \frac{m^2}{N} + \sqrt{\frac{m^4}{d^3 N}} \right).
$$
\end{theorem}

\begin{proof}[Proof of Theorem \ref{thm: end to end error}]

We have $\|A\| = 1 + O(\rho)$ and $\|B\| = 1+ O(\epsilon^2)$ by Lemma \ref{lem:norm}. Using the triangle inequality, we obtain
$$
\| (A,B) \| \leq \|A\| + \|B\| \leq 2 + O(\epsilon^2) + O(\rho).
$$
Thus, we have $\|A\| = O(1)$ and $\|(A,B)\| = O(1)$.

We obtain that 
the following bounds on the operator norm of the difference between the sample cumulants and true cumulants hold with high probability:
$$
\|\kappa_4(\by) - \hat{\kappa}_4(\by)\| = \widetilde{O}\left( \frac{r'^2}{N_1} + \sqrt{ \frac{r'^4}{p^3 N_1} } \right),
$$
$$
\|\kappa_4(\fx) - \hat{\kappa}_4(\fx)\| = \widetilde{O}\left( \frac{\ell'^2}{N_2} + \sqrt{ \frac{\ell'^4}{p^3 N_2} } \right),
$$
by Theorem \ref{thm: sample complexity of ICA}, under the assumptions on $z_i, z_i', s_j$ in the statement, and using $
\|A\| = O(1)$ and $\|(A,B)\| = O(1)$.
Let
$
T = \sum_{i=1}^\ell \nu_i \fb_i^{\otimes 4},
$
and let $\hat{T}$ be the tensor obtained after Steps 1 and 2 of Algorithm \ref{alg:generic a b}.
We can bound the distance between the true $T$ and the recovered $\hat{T}$ by 
\begin{align*}
\|\hat{T} - T\|_F &\leq \|\hat{\kappa}_4(\fx) - \kappa_4(\fx)\|_F + \beta \sqrt{\Delta_M} + O(\Delta_M) \\
&\leq p^{\frac{3}{2}}\|\hat{\kappa}_4(\fx) - \kappa_4(\fx)\|+
\beta \sqrt{p \|\hat{\kappa}_4(\by) - \kappa_4(\by)\|} + O(\Delta_M) \\
&= \widetilde{O}\left( \frac{p^{\frac{3}{2}} \ell'^2}{N_2} + \sqrt{ \frac{ \ell'^4}{N_2} } + \sqrt{p(\frac{r'^2}{N_1} + \sqrt{ \frac{r'^4}{p^3 N_1} })} \right)\\
& = \widetilde{O}\left( \frac{p^\frac{3}{2} \ell'^2}{N_2} + \sqrt{ \frac{\ell'^4}{N_2} } + \sqrt{\frac{pr'^2}{N_1} + \sqrt{ \frac{r'^4}{p N_1} }}  \right),
\end{align*}
using Theorem \ref{thm: error for T}.
Hence, we obtain the final bounds via Theorem \ref{thm:nearly_orthogonal} that
$$
|\nu_i - \mu_i| \leq (2|\nu_i|L + K) \epsilon^2 + \widetilde{O}(\delta) = O(\epsilon^2) + \widetilde{O}(\delta),
$$
and
$$
\min\{ \|\fb_i - \fc_i\|, \|\fb_i + \fc_i\| \} \leq {2^{3/2}}L \epsilon^2 + \widetilde{O}(\delta) = O(\epsilon^2) + \widetilde{O}(\delta),
$$
where
\[
\delta = \widetilde{O}\left( \frac{p^\frac{3}{2} \ell'^2}{N_2} + \sqrt{ \frac{\ell'^4}{N_2} } + \sqrt{\frac{pr'^2}{N_1} + \sqrt{ \frac{r'^4}{p N_1} }}  \right).\qedhere
\]
\end{proof}

\color{black}
\section{Proportional cICA}
\label{app: proportional cICA}
In this section, we present a variant of cICA called proportional cICA.
% \subsection{Setup}
Recall that the cICA model expresses the background $\by$ and foreground $\bx$ as 
\begin{equation}
\label{eqn:cica1}
    \by = A \mathbf{z}\qquad \text{and} \qquad \bx = A \mathbf{z}' + B \mathbf{s}. 
\end{equation}
Proportional cICA assumes assumes $\bz' = \gamma \bz$ for some scalar $\gamma > 0$. This assumption also appears in cPCA \cite{abid2017contrastive}. There, the choice of the hyperparameter $\gamma$ is not unique. However, in our setting—which involves the fourth-order cumulants $\kappa_4(\by)$ and $\kappa_4(\bx)$, under the assumption that $r + \ell \leq {p+1 \choose 2}$—the value of $\gamma$ is uniquely determined, with a closed-form expression, see Theorem~\ref{thm: unique gamma}. The details of the ensuing algorithm for computing matrix $B$ are as follows.

\begin{algorithm}[htb]
%\scriptsize
\caption{Recover $B$ from the background and foreground cumulants when $\mathbf{z'} = \gamma \mathbf{z}$}
\label{alg:z z' prop}
\begin{algorithmic}[1]
\renewcommand{\algorithmicrequire}{\textbf{Input:}}
\Require $\kappa_{4}(\bx),\kappa_{4}(\by)$ and $\ell$ as in~\eqref{eqn:tensor_decomp}.
\State \textbf{Compute $\gamma$} using the following theorem.
\State \textbf{Recover $B$:} Compute rank $\ell$ symmetric decomposition of $\kappa_{4}(\bx)-\gamma^4 \kappa_4(\by)$, using Algorithm~\ref{alg:hierarchical}.
\renewcommand{\algorithmicrequire}{\textbf{Output:}}
\Require Mixing matrix $B$.
\end{algorithmic}
\end{algorithm}

\begin{theorem}
\label{thm: unique gamma}
Consider proportional cICA with $\bz' = \gamma \bz$, for $\gamma > 0$. 
For generic $\fa_1,\ldots,\fa_r$ and $\fb_1,\ldots,\fb_\ell$ with $r+\ell \leq {p+1\choose 2}$ and $r\neq 8$,
the hyperparameter $\gamma$ is the unique value  $(\frac{1}{\lambda_i} (\fa_i\T VD^{-1} V\T \fa_i)^{-1})^{\frac{1}{4}}$, where $i$ is any index between $1$ and $r$,
 $\lambda_i$ is the coefficient of $\fa_i^{\otimes 4}$ in $\kappa_4(\bx)$ and $VD V\T$ is the thin eigendecomposition of~$\Mat(\kappa_4(\bx))$. 
\end{theorem}

\begin{proof}
The flattenings of the cumulants $\kappa_4(\by)$ and $\kappa_4(\bx)$ are, respectively,
\[ M_\by := \sum_{i=1}^r \lambda_i \mathbf{A}_i^{\otimes 2}, \qquad 
M_\fx := \gamma^4 \left( \sum_{i=1}^r \lambda_i \mathbf{A}_i^{\otimes 2} \right) +\sum_{j=1}^\ell \nu_j \mathbf{B}_j^{\otimes 2},\] 
where $\mathbf{A}_i, \mathbf{B}_j \in \RR^{p^2}$ vectorize the matrices $\fa_i^{\otimes 2}$ and $\fb_i^{\otimes 2}$, respectively and we use that $\lambda_i' = \gamma^4 \lambda_i$.
We have $\rank M_\by=r$ and $\rank M_\bx= r+\ell$, by the assumptions in the statement.

Let \(A \in \RR^{p^2 \times r}\) have columns \(\mathbf{A}_1,\ldots,\mathbf{A}_r\) and define \(D'=\gamma^4\text{Diag}(\lambda_1,\ldots,\lambda_r)\). 
Then 
\(\rank (M_\bx-AD'A\T )=\rank (\sum_{j=1}^\ell \nu_j\mathbf{B}_j^{\otimes 2})=\ell\).  
Suppose that \(VD V\T\) is the thin eigendecomposition of \(M_\bx\).
We have 
\[V\T (M_\bx-AD'A\T )V=D-(V\T A)D'(V\T A)\T.\]
We have that \(\rank D=r+\ell\), the upper bound \(\rank (V\T A)D'(V\T A)\T =\rank V\T  M_\by V\leq r\), and finally that \(\rank (D-(V\T A)D'(V\T A)\T )= \rank (V\T (M_\bx-AD'A\T )V) \leq \ell\).
Hence
\[D'=(A\T V D^{-1}V\T A)^{-1},\]
by Lemma~\ref{lem:unique_C}.
Matrices \(A,\text{Diag}(\lambda_1,\ldots,\lambda_r), V,D \) can be recovered uniquely from tensor decomposition of \(\kappa_4(\by)\) and eigendecomposition of \(M_\bx\). So \(D' \) can be recovered uniquely. Hence \(\gamma\) is unique: it is \(\gamma^4 \lambda_i= (\fa_i\T V D^{-1} V\T \fa_i)^{-1}\) for any $i\in [r]$. 
\end{proof}

One can test proportionality by seeing whether the values 
$\left(\frac{1}{\lambda_i}(\fa_i\T VD^{-1} V\T \fa_i)^{-1} \right)^{\frac{1}{4}}$
from Theorem~\ref{thm: unique gamma} are approximately equal as $i$ varies. In practice, exact proportionality may not hold, and learning $\gamma$ via the above Theorem could be challenging. 
An alternative is to use a sweep of $\gamma$ values and choose $\gamma$ according to visualization plots, a similar method to that used in cPCA~\cite{abid2017contrastive}.
We implement the proportional cICA algorithm and report its performance in Section \ref{app:simulations_details}. %of the Appendix. % along with the numerical experiments details.

\section{Practicalities and interpretation of cICA}
\label{app: practicalities}
In this section, we discuss the practicalities of cICA: preprocessing the input to speed up the algorithm and how to choose the ranks $r$ and $\ell$. We also discuss how to interpret coordinates when viewing cICA as a dimensionality reduction method.

\subsection{Choosing the ranks}
\label{app:r_and_l}

When computing the tensor decompositions in cICA,     a key step is to determine the ranks $r$ and~$\ell$. To choose the ranks, we can use the flattenings of the cumulants, the matrices $\Mat ( \kappa_4(\bx)), \Mat ( \kappa_4(\by)) \in \RR^{p^2 \times p^2}.$ 
If the expressions for the cumulant tensors $\kappa_4(\bx)$ and $\kappa_4(\by)$ in~\eqref{eqn:tensor_decomp} hold exactly, and if $r + \ell\leq {p+1\choose 2}$ and the vectors $\fa_i,\fb_j$ are generic, then
\[r = \text{rank}(\text{Mat}(\kappa_4(\by))) \quad \text{and} \quad r+\ell = \text{rank}(\text{Mat}(\kappa_4(\bx))).\]
For non-exact cumulants, such as sample cumulants, we do not work with the exact ranks of the flattening matrices, but instead examine plots of the eigenvalues in descending magnitude (see Appendix) to choose an appropriate cut-off.
We choose~$r$ such that the decrease of the eigenvalue plot of $\mathrm{Mat}(\kappa_4(\by))$ slows down, choose $q$ such that the decrease of the eigenvalue plot of $\mathrm{Mat}(\kappa_4(\bx))$ slows down, and calculate $\ell = q - r$.
The algorithm cICA has hyperparameters $r$ and $\ell$; proportional cICA has one hyperparameter $\ell$.

We discuss how the results may be affected by an incorrect choice of $r$ and $\ell$ and justify our way of ordering the foreground patterns $\fb_1, \ldots, \fb_\ell$ in~\eqref{eqn:kb}.
Let the true ranks be $r$ and $\ell$ and assume that we have used $r'$ and $\ell'$ in the input to Algorithm~\ref{alg:generic a b}. 
\begin{itemize}
    \item If $\ell'>\ell$, then $\ell' - \ell$ foreground patterns are noise.
\item   If $\ell'<\ell$, then $\ell - \ell'$ foreground patterns are not recovered.
\item       If $r'<r$, then background patterns are mixed with foreground patterns, as follows. Assuming without loss of generality that we have recovered $\fa_1, \ldots, \fa_{r'}$, the third step of Algorithm~\ref{alg:generic a b} decomposes the tensor $\sum_{i=r' + 1}^{r} \lambda_{i}'\fa_{i}^{\otimes 4} +\sum_{j=1}^{\ell} \nu_j\fb_j^{\otimes 4}$ via HTD, as in Algorithm~\ref{alg:hierarchical}. If the orthogonality hypotheses of Proposition~\ref{prop:orthogonal_case} hold, then the recovered foreground patterns are recovered together with some background patterns that are incorrectly interpreted as foreground patterns.
If the approximate orthogonality hypotheses of Theorem~\ref{thm:nearly_orthogonal} hold, then the foreground patterns are recovered approximately, together with background patterns that are classed as foreground patterns. 
Without an orthogonality condition, the recovered foreground patterns $\fb_1,\ldots,\fb_\ell$ will be polluted but still roughly collinear to the true foreground patterns for small $r-r'$  or when the dimension of the dataset is large, resulting in almost orthogonality between random vectors.
\item  If $r'>r$, then foreground patterns are mixed with background noise, as follows. Some background patterns from  Algorithm~\ref{alg:generic a b} will be noise, say $\fa_{r+1}',\ldots,\fa_{r'}'$. Step 2 of Algorithm~\ref{alg:generic a b} computes the coefficients of the tensors $(\fa_{r+1}')^{\otimes 4},\ldots,(\fa_{r'}')^{\otimes 4}$ in $\kappa_4(\bx)$, though they are not true rank one components of $\kappa_4(\bx)$. In Step 3, the tensor to be decomposed has the form $\sum_{i=1}^{r'-r} \mu_i(\fa_{r+i}')
^{\otimes 4} +\sum_{i=1}^{\ell} \nu_i\fb_i^{\otimes 4}$ for some $\mu_1,\ldots,\mu_{r'-r}\in \RR$.
As in the case $r' < r$, the foreground patterns can still be exactly or approximately recovered, under the hypotheses of Proposition~\ref{prop:orthogonal_case} and Theorem~\ref{thm:nearly_orthogonal} respectively, albeit with some background noise recovered as foreground patterns. \end{itemize}

The above discussion shows that when $r' \neq r$, the vectors $\fb_1, \ldots, \fb_\ell$ obtained from Algorithm~\ref{alg:generic a b} could represent foreground patterns, background patterns, or noise. We order the vectors according to~\eqref{eqn:kb}. 
The denominator of~\eqref{eqn:kb} is the variance of the linearly transformed 
background dataset $Y\fb$. The numerator is that of the transformed dataset $ X\fb$. 
Their ratio
enables us to select the most relevant foreground patterns, as follows. 
      \begin{itemize}
        \item If $\fb$ is a foreground pattern, we expect $\fb\T  \kappa_2(\by)\fb $ to be small relative to $\fb\T  \kappa_2(\bx) \fb$, hence a large $k(\fb)$.
        \item If $\fb$ is a background pattern, we  expect $\fb\T  \kappa_2(\by)\fb \approx \alpha\fb\T  \kappa_2(\bx) \fb$ for some constant $\alpha$ and hence $k(\fb) \approx\alpha$. 
        \item If $\fb$ is foreground noise, we  expect  a small $\fb\T  \kappa_2(\bx)\fb$, hence small $k(\fb)$.
        \item If $\fb$ is background noise, we  expect a small $\fb\T  \kappa_2(\by)\fb$, hence a large $k(\fb)$. To prevent the background noise from showing up in the recovered foreground pattern, we require $r'\leq r$.
    \end{itemize}
    In practice, we consider those patterns for which $k(\fb)$ exceeds a certain threshold or take the patterns with the two highest values of $k(\fb)$.

 \subsection{Visualization}
\label{app:visualization}

 We discuss how to interpret coordinates when using cICA for dimensionality reduction. The following proposition relates the projections $\fb_i^T\bx$ for $i\in [\ell]$ to the latent  variables~$s_i$. 
 
\begin{proposition}\label{prop:visualization s_i}
Consider the cICA model in \eqref{eqn:cica1}. Suppose $\|\fb_i\|=1$ for $i\in [\ell]$. Assume that 
 for some small \( \epsilon > 0 \) that \( |\langle \fb_i, \fb_j \rangle| < \epsilon\) and \(|\langle \fb_i, \fa_k \rangle| < \epsilon\) for \( i\neq j \in [\ell]\), \(k \in [r]
 \). 
Then, for each \( i \in [\ell] \), 
\[|s_i-\fb_i^T\bx| = (r C_{\bz'}+ (\ell - 1) C_\mathbf{s})O(\epsilon),\]
 where \( C_{\bz'} \) and \( C_\mathbf{s} \) are upper bounds on the magnitudes of  random variables in $\bz'$ and $\mathbf{s}$.
 In particular, \( \fb_i^T\bx\) approximates the component \( s_i \) with an error linear in \( \epsilon \).
\end{proposition}

\begin{proof}
Recall from~\eqref{eqn:cica1} that $\bx = A \bz' + B \mathbf{s}$. Hence
\begin{align*}
   \fb_i^T \bx &= (\fb_i^T A) \bz' + (\fb_i^T B) \mathbf{s} \\
   &= \sum_{k=1}^r \langle \fb_i, \fa_k \rangle z'_k + \sum_{j=1,j\neq i}^\ell \langle \fb_i, \fb_j \rangle s_j + s_i.
\end{align*}
The almost orthogonality conditions of the proposition then imply that
\begin{align*}
| s_i-\fb_i^T \bx|  & \leq  \sum_{k=1}^r |\langle \fb_i, \fa_k \rangle| |z'_k| +  \sum_{j=1}^\ell |\langle \fb_i, \fb_j \rangle| |s_j| \\ &\leq   (rC_{\bz'}  + (\ell-1)C_\mathbf{s})\epsilon. \qedhere
\end{align*}

\end{proof}
 
The almost orthogonality conditions in Proposition \ref{prop:visualization s_i} are strong requirements.  However, they can be relaxed -- 
if \( |\langle \fb_i, \fb_j \rangle| < \epsilon\) for chosen $i,j\in [\ell]$ and sources $s_i$ and $s_j$  have wider variance than $(\fb_i\T A)\bz'$ and $(\fb_j\T A)\bz'$, then plotting $\fb_i\T X$ against  $\fb_j\T X$ still approximates the plot of $s_i$ against $s_j$.

If \((\fb_i\T A) \mathbf{z}'\) and \((\fb_j\T A) \mathbf{z}'\) are uncorrelated, we expect the plot of \(X\fb_i\) against \(X\fb_j\) to show axis-aligned clusters; otherwise, clusters may not be axis-aligned. We specify the condition for \((\fb_i\T A) \mathbf{z}'\) and \((\fb_j\T A) \bz'\) to be uncorrelated, assuming that all variables in the tuple \(\bz'\) have the same variance.

\begin{proposition}
Consider the cICA model in~\eqref{eqn:cica1}. Suppose that the independent variables $\mathbf{z'}$ is a tuple of independent random variables with the same variance. Then  $(\fb_i\T A) \mathbf{z}'$ and $(\fb_j\T A) \mathbf{z}'$ are uncorrelated if and only if $ \langle \fb_i\T A, \fb_j\T A \rangle = 0$.
\end{proposition}

\begin{proof}
Write $\mathbf{u}=\fb_i\T A$ and  $\mathbf{v}=\fb_j\T A$. By the bilinearity of the covariance 
\begin{align*}
    \mathrm{Cov}(\mathbf{u} \mathbf{z}',\mathbf{v} \mathbf{z}') =& \sum_{1\leq i,j\leq r} u_iv_j\mathrm{Cov}(z'_i,z'_j)\\
    =&\sum_{1\leq i\leq r} u_iv_i\mathrm{Var}(z'_i)\\
    =&\mathrm{Var}(z'_1)\sum_{1\leq i\leq r} u_iv_i.
\end{align*}
The last expression is zero if and only if $\langle \mathbf{u}, \mathbf{v}\rangle =0$. \qedhere
\end{proof}

\section{Details of numerical experiments}
\label{app:simulations_details}
All experiments are run on an Apple M2 Pro with 16 GB memory. Each run of each algorithm takes at most 1 minute.

\subsection{Choices of Methods in Algorithm \ref{alg:generic a b}}\label{app: algorithm choice}
We describe the details of the synthetic data setup in Section~\ref{sec:SPM HTD best}.
Our setup involves a background dataset of three independent uniform random variables and a foreground dataset with five sources: three uniform random variables and two mixtures of beta distributions $0.5 B(2,5) + 0.5 B(5,4)$.
The foreground mixing matrix $B\in \RR^{5\times 2}$ consists of the last two columns of the identity matrix $I_5$. 
The background mixing matrix $A\in \RR^{5\times 3}$ is 
$$
\begin{pmatrix}
 0.74280923 &  0.91366784&  0.52707773\\
 -0.61857537&  0.32868577&  0.83815881\\
 0.23109269 & -0.2120887 & -0.08650875\\
 -0.0153426 &  0.07115626& -0.07315634\\
 0.10936053 & 0.08445063 &  0.08272407
\end{pmatrix}.
$$
We show in Figure~\ref{fig:four clusters} of the main text that projecting the foreground dataset using the matrix 
$B$ reveals four distinct clusters and we illustrate the performance of our algorithm SPM-HTD and the variants SPM-SPM, HTD-HTD. 
Here, we report the performance of other combinations of tensor decompositions methods, ICA methods and HTD in Figure \ref{fig:other methds four clusters}. 
Only the two methods JADE-HTD and FastICA-HTD find the four clusters in the foreground dataset.

\begin{figure}[htbp]
    \centering
    \includegraphics[width=1\linewidth]{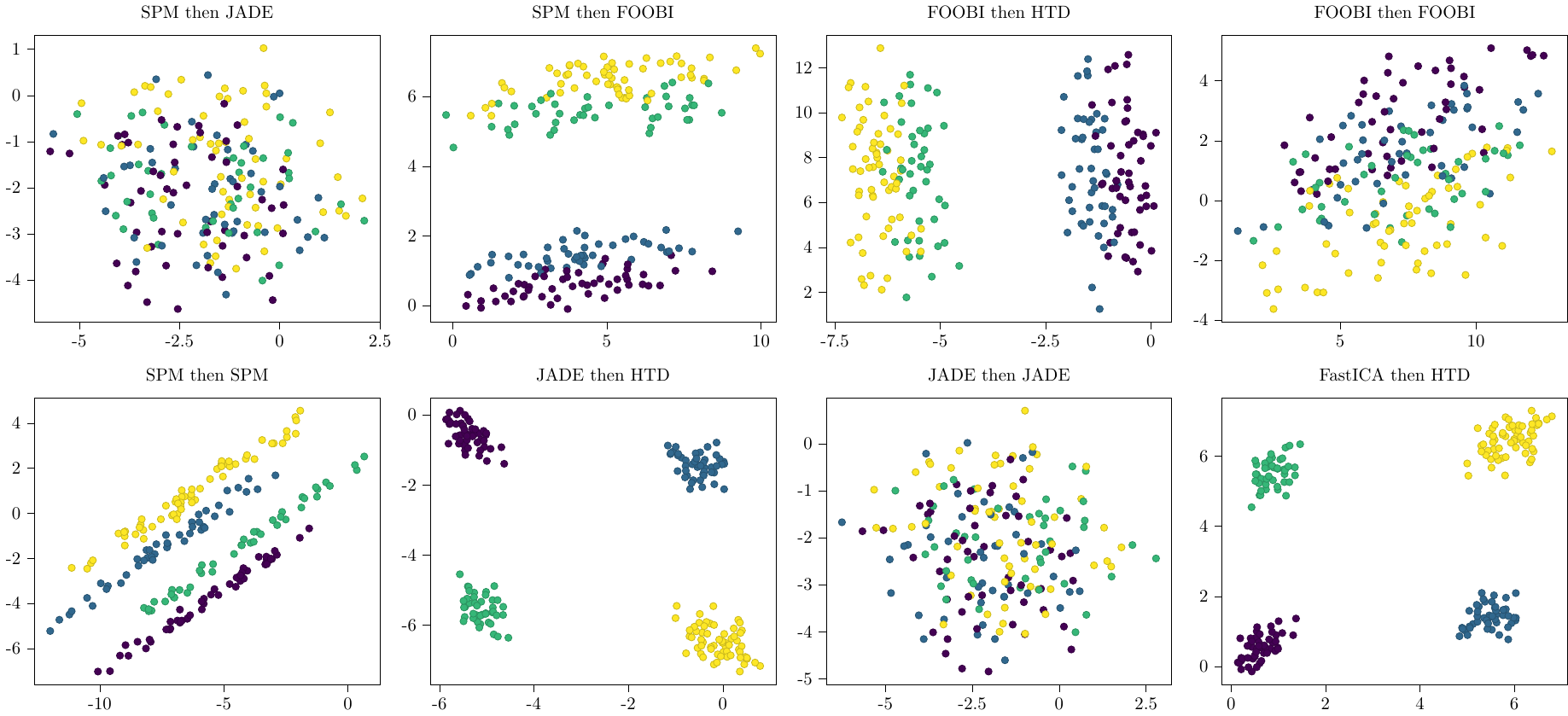}
    \caption{The performance of SPM-JADE, SPM-FOOBI, FOOBI-HTD, FOOBI-FOOBIM, SPM-SPM, JADE-HTD, JADE-JADE and FastICA-HTD on synthetic data. Only JADE-HTD and FastICA-HTD find the four clusters in the foreground dataset.}
    \label{fig:other methds four clusters}
\end{figure}

To demonstrate the necessity of our proposed three-step decomposition (Algorithm~\ref{alg:generic a b}) instead of separately decomposing the foreground and background tensors, we introduce a comparison method called SPM-SPM-Separate. Here, SPM is applied separately to the foreground and background cumulant tensors. The resulting patterns are matched using cosine similarity to identify the foreground patterns.

We vary the sample size of both datasets from 100 to 1000. 
For each sample size, we repeat the experiment 20 times by randomly drawing datasets, applying all eleven methods to estimate the matrix $B$, and computing the silhouette score on the foreground data projected via the estimated $B$. A higher silhouette score indicates that the estimated matrix $B$ accurately recovers the four clusters. To mitigate randomness, we record the best silhouette score from 20 independent runs for each method and then average these best scores across experiments.
Apart from the methods in Figure 3, we also report the performance of the method in 
Figure~\ref{fig:SPM-HTD with separate}. 
The method, SPM-SPM-Separate yields the lowest scores. This confirms the need to use the three-step decomposition procedure described in Algorithm~\ref{alg:generic a b} over separate foreground and background tensor decompositions.

\begin{figure}[htbp]
    \centering
    \includegraphics[scale = 0.6]{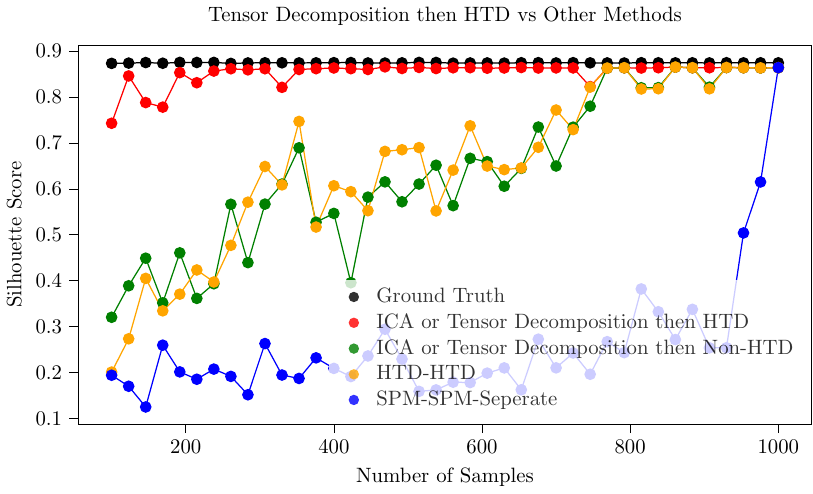}
    \caption{
    We study the accuracy of different approaches to cICA as the number of samples varies.
    We compare methods using ICA or tensor decomposition followed by HTD against  HTD-HTD, methods using ICA or tensor decomposition methods followed by non-HTD alternatives, and SPM-SPM-Separate, in which SPM is applied separately to the foreground and background datasets. 
    Performance is evaluated using the silhouette score, which measures how effectively the estimated matrix $B$ recovers the four clusters shown in the top-right plot of Figure~\ref{fig:four clusters}.
The SPM-SPM-Separate method performs worst among all methods, emphasizing the importance of employing the three-step decomposition procedure in Algorithm~\ref{alg:generic a b}. Methods using ICA or tensor decomposition followed by HTD consistently outperform both ICA or tensor decomposition methods followed by non-HTD approaches, and the HTD-HTD combination. These results justify our decision to use SPM in Step 1 and HTD in Step 3 of our algorithm.
}
    \label{fig:SPM-HTD with separate}
\end{figure}

\color{black}
\subsection{Salient patterns}
\subsubsection{Synthetic data}\label{app: synthetic}

We describe the details of the synthetic data setup in Section~\ref{sec:simulated_patterns} that produced Figure~\ref{fig:general setting}.
We consider $p\in [4,12]$.
Our samples come from the distributions~\eqref{eqn:cica1}, where matrices $A \in \RR^{p \times p}$ and $B \in \RR^{p \times (p-1)}$ are random with unit vector columns, and the columns of $B$ are assumed to be orthogonal. We assume the orthogonality of the columns of $B$ to facilitate comparison with the methods cPCA and PCPCA, which require this assumption. 

For testing Algorithm~\ref{alg:generic a b} in Figure~\ref{fig:general setting}(a) and (b) in the main text,
variables $s_i$ are exponential distributions $\exp(\theta_i)$ where $\theta_i=2$ when $i$ is odd and $\theta_i=1.5$ when $i$ is even. Variables $z_i$ and $z'_i$ are exponential distributions  $\exp(\nu_i),\exp(\nu_i')$ where $\nu_i=2,\nu_i'=1$ when $i$ is odd and $\nu_i=1,\nu_i'=2$ when $i$ is even. We generate $10^5$ data points for both the foreground and background data and apply cICA to the sample cumulant tensors.
cICA has randomness due to the subspace power method. We apply our algorithm 100 times and get 100 recovered foreground mixings $B \in \RR^{p \times (p-1)}$.

We also test Algorithm~\ref{alg:z z' prop} here. The result is shown in Figure~\ref{fig:proportional setting}.

\begin{figure}[htb]
\centering
% \subfigure[]{\input{Interpretability_plots/cosine_similarity_general.tex}}
% \subfigure[]{\input{Interpretability_plots/rel_frob_general}}
\subfigure[]{\tikzset{every picture/.style={font=\large}}
\begin{tikzpicture}[scale = 0.7]
\begin{axis}
[
legend cell align={left},
legend style={
 font = \small,
  fill opacity=0.8,
  draw opacity=1,
  text opacity=1,
  at={(0.03,0.03)},
  anchor=south west,
  draw=none
},
tick align=outside,
tick pos=left,
x grid style={darkgray176},
xlabel={p},
xmin=3.6, xmax=12.4,
xtick style={color=black},
xtick={2,4,6,8,10,12,14},
xticklabels={
  2,
  4,
  6,
  8,
  10,
  12,
  14
},
y grid style={darkgray176},
ylabel={cosine similarity},
ymin=0.549149037158856, ymax=1.0185209519212,
ytick style={color=black},
ytick={0.5,0.6,0.7,0.8,0.9,1,1.1},
yticklabels={
  0.5,
  0.6,
  0.7,
  0.8,
  0.9,
  1.0,
  1.1
}
]
\addplot [semithick, purple]
table {%
4 0.996099546722979
5 0.995319781634145
6 0.987693247860524
7 0.99380434011233
8 0.99718586488655
9 0.982379409100363
10 0.985596530986762
11 0.984486815136689
12 0.995172164499091
};
\addlegendentry{prop cICA max}
\addplot [semithick, blue]
table {%
4 0.897363172668478
5 0.785905056923316
6 0.760757299099082
7 0.643504948979757
8 0.603974074747416
9 0.655231285202517
10 0.570484124193508
11 0.657110840524977
12 0.626601210918819
};
\addlegendentry{PCPCA max}
\addplot [semithick, green]
table {%
4 0.939750516463874
5 0.859410492637685
6 0.909004816158987
7 0.75718277982045
8 0.848546794537172
9 0.781108347246622
10 0.736913646623597
11 0.790277289790488
12 0.697952315190582
};
\addlegendentry{cPCA max}
\end{axis}
\end{tikzpicture}}
\subfigure[]{\tikzset{every picture/.style={font=\large}}
\begin{tikzpicture}[scale = 0.7]

\begin{axis}[
legend cell align={left},
legend style={font = \small,fill opacity=0.8, draw opacity=1, text opacity=1, at={(0.03,0.03)}, anchor=south west, draw=none},
tick align=outside,
tick pos=left,
x grid style={darkgray176},
xlabel={p},
xmin=3.6, xmax=12.4,
xtick style={color=black},
xtick={2,4,6,8,10,12,14},
xticklabels={
  2,
  4,
  6,
  8,
  10,
  12,
  14
},
y grid style={darkgray176},
ylabel={relative Frobenius error},
ymin=0.0324309053628954, ymax=0.969430550966571,
ytick style={color=black},
ytick={0,0.2,0.4,0.6,0.8,1},
yticklabels={
  0.0,
  0.2,
  0.4,
  0.6,
  0.8,
  1.0
}
]
\addplot [semithick, purple]
table {%
4 0.0883227408657753
5 0.0967493500324941
6 0.156886915576006
7 0.111316305074053
8 0.0750217983448807
9 0.187726348175408
10 0.169726067610357
11 0.17614303768989
12 0.0982632739217318
};
\addlegendentry{prop cICA min}
\addplot [semithick, blue]
table {%
4 0.453071357142607
5 0.654362197986228
6 0.691726392298166
7 0.844387412293958
8 0.889972949310915
9 0.830383904946962
10 0.926839657984585
11 0.828117334047565
12 0.864174506776474
};
\addlegendentry{PCPCA min}
\addplot [semithick, green]
table {%
4 0.347129611344599
5 0.530263156107069
6 0.426603290753863
7 0.696874766625324
8 0.550369340466615
9 0.661651951940562
10 0.725377630446933
11 0.647646061069644
12 0.777235723329053
};
\addlegendentry{cPCA min}
\end{axis}

\end{tikzpicture}}
\caption{
The similarity of the recovered vs. true foreground patterns (i.e. the accuracy of recovering matrix \(B\)), measured via 
cosine similarity in (a) 
and relative Frobenius error in (b)
. 
The \(x\)-axis is the number of variables \(p\), which ranges from 4 to 12.
For cPCA and PCPCA, we 
test 100 hyperparameter values and plot the one with the lowest error. }
\label{fig:proportional setting}
\end{figure}

We let $z_i,z'_i$ be exponential distributions  $\exp(\nu_i),\exp(\nu_i')$ where $\nu_i=\nu_i'=1$. 
We learn the hyperparameter $\gamma'$ 
via Theorem~\ref{thm: unique gamma} of the Appendix.
The true $\gamma'$ is 1 and the recovered $\gamma'$ are all in the range $[0.94,1.08]$.

We describe the details of our comparison. For cPCA~\cite{abid2017contrastive}, we test 100 log-evenly spaced hyperparameters $\alpha$ between 0 and 1000 with $p-1$ components. Each run returns a matrix of size $p\times (p-1)$, whose columns are contrastive principal components with norm 1.
For PCPCA, we test 100 evenly spaced hyperparameters $\gamma$ between 0 and 0.9 and fix $p-1$ components. Each run returns a matrix of size $p\times (p-1)$. We normalize the columns to unit norm, to compare PCPCA with the other algorithms.

Since the columns of $B$ that are recovered are only unique up to permutation and sign, we describe how to align the outputs. 
Let \(B' \in \RR^{p \times (p-1)}\) be a recovered matrix.
Rather than searching over all ways to match the columns of \(B\) to those of \(B'\), we use a greedy algorithm to approximate the matching, as follows. 
We fix the first column of \(B\), denoted \(\fb_1\). 
We choose one of the columns of \(B'\) whose cosine similarity with \(\fb_1\) has the largest absolute value. 
We set this to be the first column of \(B'\), changing its sign if the cosine similarity is negative.
Then we select among the remaining columns, the one with the largest absolute cosine similarity with \(\fb_2\) and set this as the second column of \(B'\)(again, changing the sign if the cosine similarity is negative). 
We continue until we reach the last column. Then we compute the relative Frobenius error and mean cosine similarity which are, respectively,
\[
\sqrt{\sum_{i=1}^p\sum_{j=1}^{p-1}(b_{ij}-b_{ij}')^2/(p-1)}
\qquad \text{and} \qquad 
\frac{1}{p-1}\sum_{i=1}^{p-1} \langle \fb_i, \fb_i' \rangle.\]

\subsubsection{Corrupted MNIST dataset with continuous strength}
\label{app:corrupted MNIST continuous}
For the hyperparameters of cICA, we choose the number of components to be 30, which explains \(85\%\) of the variance. 
We then choose \(r,\ell\) for cICA and \(\ell\) for proportional cICA.
We order the eigenvalues of \(\mathrm{Mat}(\kappa_4(\by))\) and \(\mathrm{Mat}(\kappa_4(\bx))\) according to their absolute values and plot parts of the ordered eigenvalues in Figure \ref{fig:mnist_r_l_continuous}. Based on these plots, we choose \(r=65\) and \(r+\ell=130\).

\begin{figure}[htbp]
    \centering
    \scalebox{.8}{
   \subfigure[]{% This file was created with tikzplotlib v0.10.1.
\begin{tikzpicture}[scale = 0.7]

\definecolor{darkgray176}{RGB}{176,176,176}
\definecolor{steelblue31119180}{RGB}{31,119,180}

\begin{axis}[
tick align=outside,
tick pos=left,
x grid style={darkgray176},
xlabel={n-th eigenvalue},
xmin=48.55, xmax=80.45,
xtick style={color=black},
xtick={45,50,55,60,65,70,75,80,85},
xticklabels={
45,50,55,60,65,70,75,80,85
},
y grid style={darkgray176},
ymin=0.269238423719678, ymax=0.713843798701155,
ytick style={color=black},
ytick={0.25,0.3,0.35,0.4,0.45,0.5,0.55,0.6,0.65,0.7,0.75},
yticklabels={
0.25,0.3,0.35,0.4,0.45,0.5,0.55,0.6,0.65,0.7,0.75
}
]
\addplot [semithick, steelblue31119180]
table {%
50 0.693634463474725
51 0.689335246540913
52 0.663502558035099
53 0.636285523458932
54 0.634662200853531
55 0.605686613222113
56 0.59021351874416
57 0.588702931854172
58 0.574212903056048
59 0.542309233441365
60 0.509771607621794
61 0.50666054006316
62 0.471176539105032
63 0.469979897378173
64 0.469674888980654
65 0.436973494876533
66 0.429726509899752
67 0.415177272734636
68 0.401055518575199
69 0.390969101100651
70 0.379114860100457
71 0.367529638275477
72 0.363532967923284
73 0.350155475373931
74 0.341212815677229
75 0.330319350472526
76 0.322696787957564
77 0.299870445315847
78 0.293287778469812
79 0.289447758946109
};
\end{axis}

\end{tikzpicture}}
    \subfigure[]{% This file was created with tikzplotlib v0.10.1.
\begin{tikzpicture}[scale = 0.7]

\definecolor{darkgray176}{RGB}{176,176,176}
\definecolor{steelblue31119180}{RGB}{31,119,180}

\begin{axis}[
tick align=outside,
tick pos=left,
x grid style={darkgray176},
xlabel={n-th eigenvalue},
xmin=97.55, xmax=151.45,
xtick style={color=black},
xtick={90,100,110,120,130,140,150,160},
xticklabels={
90,100,110,120,130,140,150,160
},
y grid style={darkgray176},
ymin=0.787688835756546, ymax=1.53895043591135,
ytick style={color=black},
ytick={0.7,0.8,0.9,1,1.1,1.2,1.3,1.4,1.5,1.6},
yticklabels={
0.7,0.8,0.9,1,1.1,1.2,1.3,1.4,1.5,1.6
}
]
\addplot [semithick, steelblue31119180]
table {%
100 1.50480218135886
101 1.48190430023463
102 1.45740542649241
103 1.43889462194469
104 1.40022866420784
105 1.38731650667569
106 1.38133723949079
107 1.36040317035024
108 1.35258349779583
109 1.34300734570978
110 1.34258508939491
111 1.32926359220611
112 1.3221750467757
113 1.29164722151098
114 1.23714689770608
115 1.23199513437619
116 1.20510598462259
117 1.19762800962202
118 1.15682204121555
119 1.15634098232709
120 1.14742699786038
121 1.12572676297742
122 1.11484368244516
123 1.11387409956478
124 1.09913900770943
125 1.09328761511024
126 1.07111333242213
127 1.03212974616621
128 1.02940484214573
129 1.01644597244257
130 1.01305708935222
131 1.00299212223676
132 0.99624065195762
133 0.993533730368712
134 0.953512938337384
135 0.950814998960117
136 0.943630215046441
137 0.943410274825998
138 0.929393459336709
139 0.91154279743744
140 0.893669225142683
141 0.885124520039848
142 0.879252523989554
143 0.868070859394346
144 0.866285028704615
145 0.84968438229824
146 0.843168659141304
147 0.834023795053422
148 0.825587783109639
149 0.821837090309037
};
\end{axis}

\end{tikzpicture}}}
    \caption{Absolute values of eigenvalues of $\mathrm{Mat}(\kappa_4(\by))$ (left) and $\mathrm{Mat}(\kappa_4(\bx))$ (right).}
    \label{fig:mnist_r_l_continuous}
\end{figure}

We fix the random seed to be 0 for cICA. 
We check that the absolute values of the foreground-to-background cumulant ratios for the background patterns $\fa_1, \ldots, \fa_r$ range from $7.2\times 10^{-3}$ to $91$.

For cPCA, we run the experiment for $\alpha=1$.
We run PCPCA for $\gamma'=0.9$.

\subsubsection{Human and monkey gene expression data}
\label{app:monkey_and_human}

We describe the patterns obtained from the comparison of human and monkey gene expression in Section~\ref{sec:human_and_monkey}. The selected 15 highest variance genes among the 139 selected genes in~\cite{suresh2023comparative} are EIF3K, NDUFA13, SARNP, MYL10, TAF9, PRCD, BBS5, MRPS14, RING1, AGPAT5, FLOT1, BTBD7, MASTL, KANK1, BDP1. The 15 highest variance genes among the remaining $3244 = 3383 - 139$ genes are LUC7L3, RBKS, RBM7, AP4S1, CLCN1, CLASP1, ADTRP, CNNM3, NDUFAF7, CNIH4, RPUSD2, NELFCD, RPP14, ROMO1, RNF181. 

For cICA, we fix the random seed to be 0.
We use the plots of the eigenvalues of the flattenings of $\kappa_4(\by),\kappa_4(\bx)$ to choose $r=22$ and $\ell=46-22=24$.
The absolute values of the foreground-to-background cumulant ratios for the background patterns $\fa_1, \ldots, \fa_r$ range from $4.6\times 10^{-2}$ to $55$. Hence the shared gene patterns between human and monkey have different strength across the two datasets.

The top two foreground patterns are:
\begin{align*}
\fb_1\T = [& -0.04 , -0.041, -0.09 , -0.051, -0.12 ,  0.075,  0.01 , -0.004,
        0.002,  0.007,\\
& -0.07 , -0.061,  0.95 ,  0.192, -0.009, -0.007,
       -0.002, -0.001, -0.076, -0.042,\\
& -0.008, -0.04 ,  0.005, -0.058,
        0.012, -0.012, -0.05 , -0.006, -0.046, -0.005] \\
\fb_2\T = [&0.615, -0.166,  0.185,  0.119,  0.113, -0.099, -0.118,  0.011,
        0.045, -0.025,\\
&0.098,  0.141, -0.482, -0.339,  0.054,  0.028,
       -0.005,  0.03 ,  0.247, -0.017,\\
&-0.031,  0.043,  0.012,  0.043,
        0.015,  0.04 ,  0.025,  0.002,  0.236, -0.016],
\end{align*}
where the coordinates are labeled by the 30 genes in the order listed above.
The 15 genes with the largest absolute values of the top foreground pattern include 10 genes among the 139 selected in \cite{suresh2023comparative}.
The 15 genes with the largest absolute values of the second foreground pattern include 13 genes from \cite{suresh2023comparative}.
Therefore, the foreground patterns obtained via cICA demonstrate consistency with the finding in~\cite{suresh2023comparative} that this subset of 139 genes captures human-specific information.

For ICA, we run HTD for $r=46$ and rank the patterns according to 
\eqref{eqn:kb}.
We denote by $(\fb_1<15)$ (resp. $ (\fb_2<15)$) the number of genes in the top 15 with largest absolute value in $\fb_1$ that are among the 139 selected genes. % edit

We run cPCA for 100 $\alpha$ between 0 to 1000 and choose $\alpha$ that achieves the highest value of $(\fb_1<15) + (\fb_2<15)$. The highest value is obtained at $\alpha= 0.17$. Note that our parameters for proportional cICA are square of the cPCA parameters, since if $\mathbf{z}=\lambda \mathbf{z}'$, then $\kappa_2(\bz)=\lambda^2 \kappa_2(\bz')$ and $\kappa_4(\bz)=\lambda^4 \kappa_4(\bz')$.
We run PCPCA for 100 evenly spaced $\gamma'$ values between 0 and 0.9.
The best score of $(\fb_1<15) + (\fb_2<15)$ is obtained for~$\gamma'=0$.

We also run the algorithm for 100 log-evenly spaced $\gamma$ between 0 and $10^6$ and choose $\gamma$ to achieve the highest value of $(\fb_1<15) + (\fb_2<15)$. The highest score is achieved at $\gamma=0.03$. 
We observe
that the 15 genes with the highest absolute values in $\fb_1$ (resp. $\fb_2$) have 10 (resp. 13) genes among the 15 selected genes that come from the subset of 139 in~\cite{suresh2023comparative}. 
The number of misclassified genes is 6.

\subsection{Dimensionality reduction}
\subsubsection{Mouse protein data}
\label{app:mouse}

There are 270 foreground
samples. These are the protein expression in the cortex of mice subjected to shock therapy. Of these samples, 135 have Down syndrome and 135 do not. There are 135 background samples, protein expression measurements from mice without
Down Syndrome who did not receive shock therapy. Each sample measures the expression of 77 proteins; that is, $p = 77$.

For cICA, we preprocess using PCA as described in Section \ref{sec:ranking_bs}. We take $k=15$ components, which explain $90\%$ of the variance. 
We then choose $r$ and $\ell$, as described in Appendix section \ref{app:r_and_l}.
That is, we compute the eigenvalues of $\mathrm{Mat}(\kappa_4(\by))$ and $\mathrm{Mat}(\kappa_4(\bx))$, ranking the eigenvalues by magnitude, see Figure \ref{fig:mice_r_l}. Based on these plots, we choose $r=27$ and $\ell=53 - 27 = 26$.

\begin{figure}[htbp]
\centering
\scalebox{.8}{
\subfigure[]{\tikzset{every picture/.style={font=\large}}
% This file was created with tikzplotlib v0.10.1.
\begin{tikzpicture}[scale = 0.7]

\definecolor{darkgray176}{RGB}{176,176,176}
\definecolor{steelblue31119180}{RGB}{31,119,180}

\begin{axis}[
tick align=outside,
tick pos=left,
x grid style={darkgray176},
xlabel={n-th eigenvalue},
xmin=7.55, xmax=61.45,
xtick style={color=black},
xtick={10,15,20,25,30,35,40,45,50,55},
xticklabels={
  10,
  15,
  20,
  25,
  30,
  35,
  40,
  45,
  50,
  55
},
y grid style={darkgray176},
ylabel={eigenvalue},
ymin=-3.19633546007519, ymax=136.205310558293,
ytick style={color=black},
ytick={-20,0,20,40,60,80,100,120,140},
yticklabels={
  -20,
  0,
  20,
  40,
  60,
  80,
  100,
  120,
  140
}
]
\addplot [semithick, steelblue31119180]
table {%
10 129.868872102912
11 119.685422451593
12 119.205960559256
13 108.532777989007
14 100.614495564858
15 98.1824756646108
16 88.5416603483296
17 84.6909801098414
18 79.0035571034665
19 73.0749424516109
20 60.6763199151065
21 60.0059625606454
22 53.6404052406842
23 47.1378829264854
24 46.6997022819183
25 37.2785855873179
26 37.0343768398301
27 36.2284576693594
28 26.138615726947
29 24.6096085026816
30 23.5324196012496
31 20.253086576999
32 18.3822867472698
33 18.2208943427281
34 16.2898846408265
35 13.4234189063779
36 13.2682540798257
37 11.753703024432
38 11.4659319072035
39 10.4492822267937
40 10.2229120735662
41 9.69960260046346
42 9.22208913634273
43 7.84099563215642
44 7.83935769891669
45 7.10095312703929
46 6.60708577998992
47 6.56332729686067
48 6.21287015954825
49 5.68606908051499
50 5.50632579248541
51 5.23963328006148
52 4.78815481170698
53 4.18144432902929
54 3.90732876785127
55 3.57820266469067
56 3.5586730011551
57 3.34391507880113
58 3.24054070640747
59 3.14010299530517
};
\end{axis}

\end{tikzpicture}} 
\subfigure[]{\tikzset{every picture/.style={font=\large}}
% This file was created with tikzplotlib v0.10.1.
\begin{tikzpicture}[scale = 0.7]

\definecolor{darkgray176}{RGB}{176,176,176}
\definecolor{steelblue31119180}{RGB}{31,119,180}

\begin{axis}[
tick align=outside,
tick pos=left,
x grid style={darkgray176},
xlabel={n-th eigenvalue},
xmin=17.05, xmax=81.95,
xtick style={color=black},
xtick={20,25,30,35,40,45,50,55,60,65,70,75,80},
xticklabels={
  20,
  25,
  30,
  35,
  40,
  45,
  50,
  55,
  60,
  65,
  70,
  75,
  80
},
y grid style={darkgray176},
ylabel={eigenvalue},
ymin=0.994712438245972, ymax=43.5890246135928,
ytick style={color=black},
ytick={0,5,10,15,20,25,30,35,40,45},
yticklabels={
  0,
  5,
  10,
  15,
  20,
  25,
  30,
  35,
  40,
  45
}
]
\addplot [semithick, steelblue31119180]
table {%
20 41.6529195147134
21 39.9477192314057
22 39.2128042493993
23 34.9331553838275
24 34.7467542405345
25 30.4722282043233
26 29.5736829200129
27 27.4777201432907
28 27.0565437634937
29 27.053762916787
30 25.2828897708403
31 23.7000163367084
32 23.5178740109914
33 22.3159740224102
34 21.1518017479663
35 20.1955894875519
36 18.6986931520109
37 18.3276011199393
38 18.0211669760116
39 17.4014380786991
40 16.1373323998697
41 15.3294374493943
42 14.309518240356
43 13.9916819852592
44 13.3206585329395
45 13.1792025747786
46 13.0208528470059
47 11.8537241310223
48 11.6943451476852
49 11.5971773320585
50 11.5331581149752
51 10.6576458475061
52 10.552982125453
53 9.53774045277387
54 9.18760039309507
55 8.80180839150916
56 8.29493516252332
57 8.1132324814536
58 8.0999445232012
59 7.35160657913834
60 6.92298415680288
61 6.84674341356349
62 6.53172364122439
63 6.30080093841373
64 5.88888246621785
65 5.66143011072967
66 5.5573985741066
67 5.02196915307468
68 4.95034957361574
69 4.93903981582922
70 4.74822150203524
71 4.67557965719651
72 4.22359898342084
73 4.10810124655497
74 3.78380463172212
75 3.43225080874869
76 3.19420394754959
77 3.14523180583364
78 3.09746803327391
79 2.93081753712537
};
\end{axis}

\end{tikzpicture}}}
    \caption{Absolute values of eigenvalues of  \(\mathrm{Mat}(\kappa_4(\by))\) (left) and \(\mathrm{Mat}(\kappa_4(\bx))\) (right).}
    \label{fig:mice_r_l}
\end{figure}

For cICA, we fix the random seed to be 0.
For proportional cICA, we run the algorithm for 100 log-evenly spaced $\gamma$ between 0 and $10^6$. The highest silhouette score is obtained at $\gamma=0$, equivalent to running ICA.

We run cPCA for 100 $\alpha$ between 0 to 1000. These are the default values of $\alpha$ in the code of \cite{abid2017contrastive}. 
% Note that our parameters for proportional cICA are square of the cPCA parameters, since if $\mathbf{z}=\lambda \mathbf{z}'$, then $\kappa_2(\bz)=\lambda^2 \kappa_2(\bz')$ and $\kappa_4(\bz)=\lambda^4 \kappa_4(\bz')$.
We plotted the choice with the highest silhouette score, which was achieved for $\alpha=26.2$. 

We run PCPCA for 100 evenly spaced $\gamma'$ values between 0 and $0.9 \cdot \frac{270}{135}$. 270 and 135 are the number of samples in the foreground and background datasets, respectively. Such choices of $\gamma'$ are in accordance with the setup in \cite{li2020probabilistic} and are sufficient to find the highest silhouette score.
The best score was obtained when $\gamma'=0.9 \cdot \frac{270}{135}$. In \cite{li2020probabilistic}, the authors take a further step to scale the probabilistic contrastive principal components, before calculating the silhouette score. The silhouette score obtained after this additional step is 0.450.

\subsubsection{Corrupted MNIST data with discrete strength}
\label{app: mixed mnist}

For the hyperparameters of cICA, we choose the number of components to be 30, which explains \(85\%\) of the variance. 
We then choose \(r,\ell\) for cICA and \(\ell\) for proportional cICA.
We order the eigenvalues of \(\mathrm{Mat}(\kappa_4(\by))\) and \(\mathrm{Mat}(\kappa_4(\bx))\) according to their absolute values and plot parts of the ordered eigenvalues in Figure \ref{fig:mnist_r_l_discrete}. Based on these plots, we choose \(r=51\) and \(r+\ell=192\).
The absolute values of the foreground-to-background cumulant ratios for the background patterns $\fa_1, \ldots, \fa_r$ range from $6.7 \times 10^{-3}$ to 16.

\begin{figure}[htbp]
    \centering
    \scalebox{.8}{
   \subfigure[]{% This file was created with tikzplotlib v0.10.1.
\begin{tikzpicture}[scale = 0.7]

\definecolor{darkgray176}{RGB}{176,176,176}
\definecolor{steelblue31119180}{RGB}{31,119,180}

\begin{axis}[
tick align=outside,
tick pos=left,
x grid style={darkgray176},
xlabel={n-th eigenvalue},
xmin=38.55, xmax=70.45,
xtick style={color=black},
xtick={35,40,45,50,55,60,65,70,75},
xticklabels={
35,40,45,50,55,60,65,70,75
},
y grid style={darkgray176},
ymin=0.500274799876602, ymax=1.33503850590372,
ytick style={color=black},
ytick={0.5,0.6,0.7,0.8,0.9,1,1.1,1.2,1.3,1.4},
yticklabels={
0.5,0.6,0.7,0.8,0.9,1,1.1,1.2,1.3,1.4
}
]
\addplot [semithick, steelblue31119180]
table {%
40 1.29709470108431
41 1.27190276735686
42 1.22744438237163
43 1.14264533044589
44 1.12980931074168
45 1.12782382192294
46 1.06685922277909
47 1.05039026467306
48 1.05000725908844
49 1.0147149740153
50 1.00186270270793
51 0.943579943881485
52 0.902130360305609
53 0.883869049557847
54 0.843228686248399
55 0.83204532870177
56 0.826824362195603
57 0.803355725031787
58 0.75658734989
59 0.732125065606922
60 0.700926302849758
61 0.683267926734056
62 0.669730697454329
63 0.639912240497276
64 0.623913585486058
65 0.618489357269592
66 0.614514611132317
67 0.56301785377124
68 0.544270342638197
69 0.538218604696016
};
\end{axis}

\end{tikzpicture}}
    \subfigure[]{\begin{tikzpicture}[scale = 0.7]

\definecolor{darkgray176}{RGB}{176,176,176}
\definecolor{steelblue31119180}{RGB}{31,119,180}

\begin{axis}[
tick align=outside,
tick pos=left,
x grid style={darkgray176},
xlabel={n-th eigenvalue},
xmin=167.55, xmax=221.45,
xtick style={color=black},
xtick={160,170,180,190,200,210,220,230},
xticklabels={
160,170,180,190,200,210,220,230
},
y grid style={darkgray176},
ymin=0.0386808217620798, ymax=0.0794592513683697,
ytick style={color=black},
ytick={0.035,0.04,0.045,0.05,0.055,0.06,0.065,0.07,0.075,0.08},
yticklabels={
  $0.035$,
  $0.040$,
  $0.045$,
  $0.050$,
  $0.055$,
  $0.060$,
  $0.065$,
  $0.070$,
  $0.075$,
  $0.080$
},
y label style={/pgf/number format/.cd, fixed, precision=3},
scaled ticks=false
]
\addplot [semithick, steelblue31119180]
table {%
170 0.0776056863862656
171 0.0769588711464868
172 0.0751873948962694
173 0.0739984142869499
174 0.0724907789487785
175 0.0723843745392608
176 0.0706215830579442
177 0.0688414986655231
178 0.0679777676493524
179 0.067782125189526
180 0.066433889660879
181 0.0654292714967938
182 0.0640964393842078
183 0.0636290393687166
184 0.0635640163225082
185 0.0630286076606247
186 0.0625951334242943
187 0.0617199103599632
188 0.060916260344184
189 0.0600659471464946
190 0.0589055403675772
191 0.0578504449792008
192 0.0572689172984808
193 0.0554593366154507
194 0.0546940786463131
195 0.0540289009928411
196 0.0539566637181563
197 0.053868798371876
198 0.0528043917494368
199 0.0523315888832203
200 0.0512929584787813
201 0.0508935223073031
202 0.0504973247586787
203 0.0498216723090145
204 0.0491793730381064
205 0.0485703971210181
206 0.0479513634848642
207 0.0472927352766839
208 0.0468145047935972
209 0.0460838579740941
210 0.0458739764289029
211 0.0450417941790789
212 0.0447941633625538
213 0.043672762337757
214 0.0434515725051748
215 0.0427692167074713
216 0.041984405688642
217 0.0417869720257349
218 0.0409568910651864
219 0.0405343867441839
};
\end{axis}

\end{tikzpicture}}}
    \caption{Absolute values of eigenvalues of $\mathrm{Mat}(\kappa_4(\by))$ (left) and $\mathrm{Mat}(\kappa_4(\bx))$ (right).}
    \label{fig:mnist_r_l_discrete}
\end{figure}

We fix the random seed to be 0 for cICA. 
For cPCA, we run experiments for 100 $\alpha$ values between 0 and 1000 and choose $\alpha=6.6$ that achieves the highest silhouette score when plotting the mixed images of digits 0 and 1 using their inner product with the first two patterns.
We run PCPCA for 100 evenly spaced $\gamma'$ between 0 and $0.9$ and choose the $\gamma'=0.9$ with the highest silhouette score when plotting with the first two patterns.
We also include ICA with $r=192$ to illustrate that cICA performs better than ICA.

\section{Additional numerical experiment}
\subsection{Single cell RNA data}
\label{app: RNA}
We study the single-cell RNA sequencing data from \cite{zheng2017massively}.
The foreground data points are gene expressions of bone marrow mononuclear cells
from patients with acute myeloid leukemia before and after they received a stem-cell transplant; the background dataset contains gene expression measurements of healthy people.
The foreground dataset includes 7525 pre-transplant patients and 4874 post-transplant patients, while the background dataset consists of 4457 healthy patients. 
Each sample contains gene expression measurements of bone marrow mononuclear cells. 
We preprocess the data by log-transforming
and subsetting to the 500 most variable genes, in accordance with previous analyses on these data \cite{zheng2017massively,abid2018exploring,li2020probabilistic}.

For cICA, the absolute values of the foreground-to-background cumulant ratios for the background patterns $\fa_1, \ldots, \fa_r$ range from $1.5\times 10^{-4}$ to 564.
The projection plots of cICA, proportional cICA, cPCA, and PCPCA are shown in Figure~\ref{fig:RNA_data}.
The method cPCA has the highest silhouette score (0.451), followed by proportional cICA (0.402), then cICA (0.344), then PCPCA (0.164). 
We also run ICA to the foreground dataset and it has silhouette score 0.202 for comparison with cICA.

\begin{figure}[htbp]
    \centering
    \subfigure[]{\includegraphics[height=0.24\textwidth]{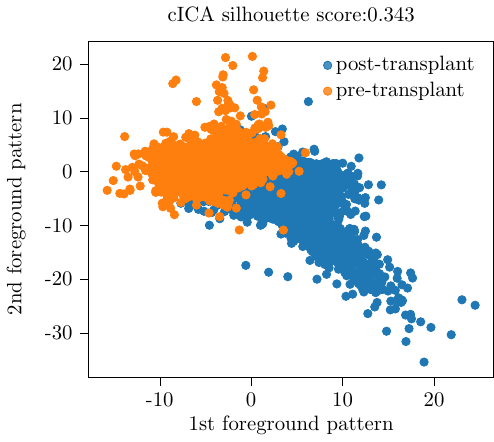}} 
    \subfigure[]{\includegraphics[height=0.24\textwidth]{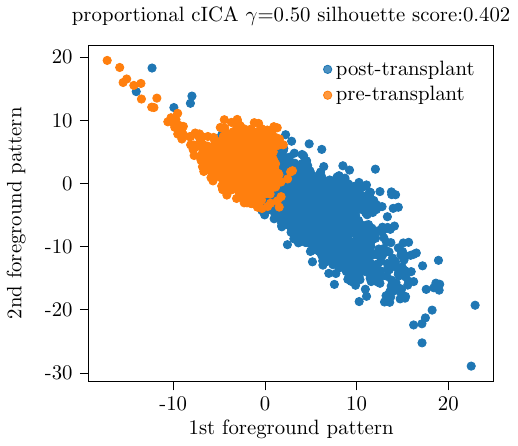}} 
    \subfigure[]{\includegraphics[height=0.24\textwidth]{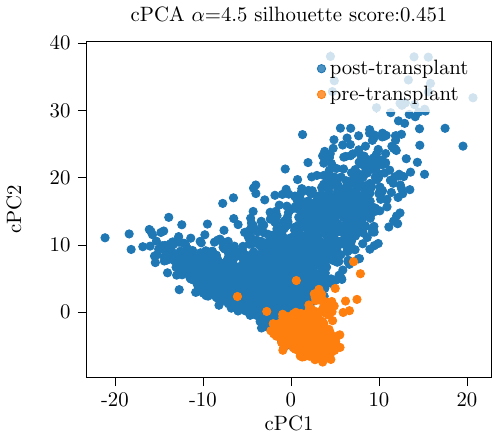}}
    \subfigure[]{\includegraphics[height=0.24\textwidth]{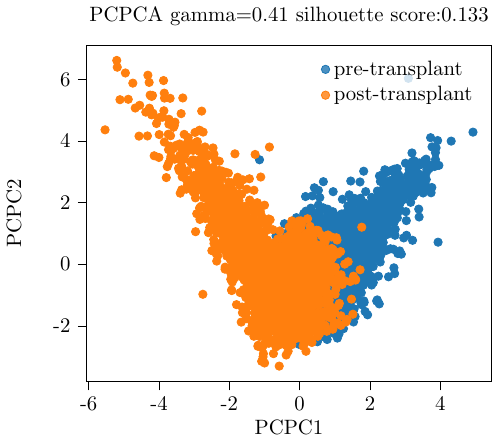}}   
    \subfigure[]{\includegraphics[height = 0.24\textwidth]{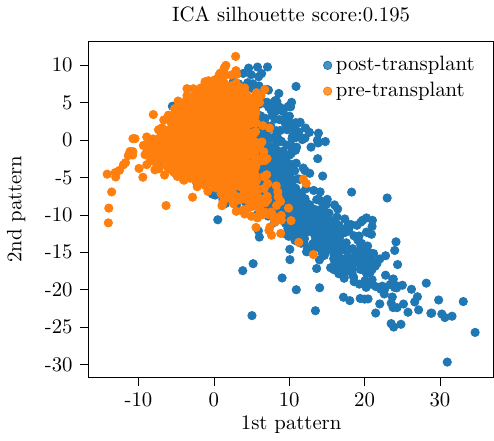}}
    \caption{Dimensionality reduction of the single-cell RNA sequencing data from \cite{zheng2017massively} via (a) cICA (b) proportional cICA (c) cPCA (d) PCPCA  (e) ICA.
    }
    \label{fig:RNA_data}
\end{figure}

\begin{figure}[htb]
    \centering
    \scalebox{.8}{
   \subfigure[]{% This file was created with tikzplotlib v0.10.1.
\begin{tikzpicture}[scale = 0.7]

\definecolor{darkgray176}{RGB}{176,176,176}
\definecolor{steelblue31119180}{RGB}{31,119,180}

\begin{axis}[
tick align=outside,
tick pos=left,
x grid style={darkgray176},
xlabel={n-th eigenvalue},
xmin=22.3, xmax=81.7,
xtick style={color=black},
xtick={20,30,40,50,60,70,80,90},
xticklabels={
20,30,40,50,60,70,80,90
},
y grid style={darkgray176},
ylabel={eigenvalue},
ymin=87.7315189951688, ymax=833.667379829479,
ytick style={color=black},
ytick={0,100,200,300,400,500,600,700,800,900},
yticklabels={
0,100,200,300,400,500,600,700,800,900
}
]
\addplot [semithick, steelblue31119180]
table {%
25 799.761204337011
26 764.268471730912
27 714.619589386712
28 675.715619225732
29 631.843013698751
30 608.86148082137
31 596.389732195454
32 582.348032700566
33 540.19299180409
34 499.177820610146
35 488.966333617294
36 477.85441894312
37 476.217758600017
38 462.038093815415
39 411.548210701254
40 405.094521821508
41 395.143216647047
42 361.931379206016
43 356.256969410272
44 353.83441934431
45 335.040849380667
46 325.914812208222
47 312.678322257301
48 299.702130790883
49 285.795276660655
50 271.147298928302
51 270.902839075263
52 261.209111376185
53 238.213335620135
54 229.334104035753
55 219.64609373883
56 211.162506983521
57 201.7450053419
58 200.250842782597
59 191.801992417359
60 191.443973827909
61 187.62187893368
62 180.976637698309
63 175.628568828996
64 169.227705637958
65 168.641823727773
66 162.605366668115
67 160.479836474987
68 156.64397923922
69 149.343630969031
70 146.683393933418
71 143.617111617534
72 142.080306070846
73 132.656174717612
74 129.379954245533
75 129.377922963905
76 128.098215038597
77 124.228258259144
78 123.191442613622
79 121.637694487637
};
\end{axis}

\end{tikzpicture}}
    \subfigure[]{% This file was created with tikzplotlib v0.10.1.
\begin{tikzpicture}[scale = 0.7]

\definecolor{darkgray176}{RGB}{176,176,176}
\definecolor{steelblue31119180}{RGB}{31,119,180}

\begin{axis}[
tick align=outside,
tick pos=left,
x grid style={darkgray176},
xlabel={n-th eigenvalue},
xmin=76.55, xmax=152.45,
xtick style={color=black},
xtick={70,80,90,100,110,120,130,140,150,160},
xticklabels={
70,80,90,100,110,120,130,140,150,160
},
y grid style={darkgray176},
ylabel={eigenvalue},
ymin=28.5982920809002, ymax=122.291606160984,
ytick style={color=black},
ytick={20,40,60,80,100,120,140},
yticklabels={
20,40,60,80,100,120,140
}
]
\addplot [semithick, steelblue31119180]
table {%
80 118.032819157344
81 115.296493656536
82 113.909978341503
83 111.171094885252
84 110.056849259444
85 109.376168091937
86 104.577020244089
87 97.8477559960069
88 97.8314532303654
89 94.3038086521535
90 93.8099953469551
91 90.2042646999485
92 88.4992514116509
93 84.0590313951597
94 83.4128483363937
95 80.3065233726855
96 80.0851536414283
97 78.988358646715
98 75.8783181982746
99 74.836274950834
100 74.0072439323678
101 70.0553973968392
102 69.3259249301069
103 68.2640144037623
104 67.1960194229575
105 65.489008869992
106 64.1358429745527
107 62.6071244930506
108 61.3064946866605
109 59.7583044230865
110 59.1664751051214
111 58.9915187211129
112 57.7832416064667
113 57.5644165742315
114 56.5465282288082
115 56.0995616533109
116 54.5963850001366
117 54.5341553547965
118 53.624436031126
119 52.0801336132572
120 51.8444309181169
121 51.2419827627316
122 50.53313066602
123 49.9797087568786
124 48.9485972071973
125 48.1571829965411
126 47.2949601372363
127 46.5598598355755
128 45.6933760203441
129 45.5726260564228
130 44.295800500771
131 44.2573358059269
132 43.5739716803956
133 42.9926411683412
134 42.7126937696804
135 42.1412426567514
136 40.7966611302435
137 40.1106261569583
138 39.9992827902904
139 39.0894195802271
140 38.2797861881583
141 37.707420084873
142 37.3577089275035
143 36.694334185097
144 35.6339330834403
145 34.9582846808736
146 34.8864614668042
147 34.1744751245028
148 33.6364712624924
149 32.8570790845404
};
\end{axis}

\end{tikzpicture}}}
    \caption{Absolute values of eigenvalues of  \(\mathrm{Mat}(\kappa_4(\by))\) (left) and \(\mathrm{Mat}(\kappa_4(\bx))\) (right).
    }
    \label{fig:RNA_r_l}
\end{figure}

For the hyperparameters of cICA and proportional cICA, we choose the number of components to be 30 which explains $54.5\%$ of the variance. 
We then choose $r,\ell$ for cICA and $\ell$ for proportional cICA.
We order the eigenvalues of $\mathrm{Mat}(\kappa_4(\by))$ and $\mathrm{Mat}(\kappa_4(\bx))$ according to their absolute values and plot out parts of the ranked eigenvalues in Figure \ref{fig:RNA_r_l}. We choose $r=53$ and $r+\ell=116$.

We fix random seed 0 for cICA and ICA. For ICA, we run the HTD algorithm for $r=116$.
For proportional cICA, we run the algorithm for 100 log-evenly spaces $\gamma$ between 0 and $10^{6}$. The highest silhouette score is 0.402, obtained when $\gamma=0.50$.

For cPCA, we plot the first two cPCA components. As above, we run 
cPCA using 100 $\alpha$ between 0 to 1000, 
the default values from \cite{abid2017contrastive}. 
The highest silhouette score is 0.457, obtained when $\alpha=3.5$. 
We run PCPCA for 100 evenly spaced $\gamma'$ between 0 and $0.9\cdot \frac{12399}{4457}$, in accordance with~\cite{li2020probabilistic}.
The numbers 12399 and 4457 are the sample sizes of the foreground and background datasets, respectively.
In accordance with the experiment in \cite{abid2017contrastive}, we run PCPCA with 4 components. The best silhouette score over any $\gamma'$ and any pair of probabilistic contrastive principal components is 0.164, obtained when $\gamma'=0.41$ using the third and fourth components.
If we normalize the probabilistic contrastive principal components and then calculate the silhouette score, the score is 0.184. 
There are three reasons why the silhouette score for cICA methods is worse than that of cPCA. 
\begin{enumerate}
    \item Due to the computational cost of forming large tensors, cICA methods is applied to the PCA transformed dataset using the top 30 principal components, which explain only 54.5$\%$ of the variance. The clustering quality is expected to be worse than when applied to the complete dataset.
    \item Our cICA methods return patterns that only exist in the foreground while cPCA learns patterns that are more prominent in the foreground than in the background.
    \item The patterns learned by cICA do not have any relation while cPCA returns perfectly orthogonal patterns. The patterns from cICA may enjoy better intrepretability but produce suboptimal plots than cPCA.
\end{enumerate}

To illustrate these arguments, we generate plots using cPCA and cICA as follows.
We apply cPCA to the PCA transformed dataset using the top 30 principal components.
The plot obtained using the top two cPCA components is shown in Figure~\ref{fig: RNA_comparison}(a). The silhouette score achieved is 0.434. 
For cICA, we apply proportional cICA to the PCA transformed dataset using the same hyperparameters as above. We select the top foreground pattern $\fb$ and the top background pattern $\fa$ ranked according to \eqref{eqn:kb}. We then use $\fb, \frac{\fa-\langle \fa, \fb \rangle \fb}{\| \fa-\langle \fa, \fb \rangle \fb \|}$as directions to plot the data. The plot is shown in \ref{fig: RNA_comparison}(b). The silhouette score obtained is 0.428, almost the same as that of cPCA.

\begin{figure}[htb]
\centering 
\subfigure[]{\includegraphics[height=0.24\textwidth]{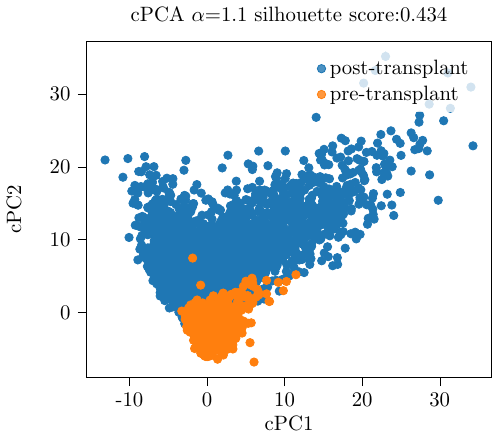}} 
\subfigure[]{\includegraphics[height = 0.24\textwidth]{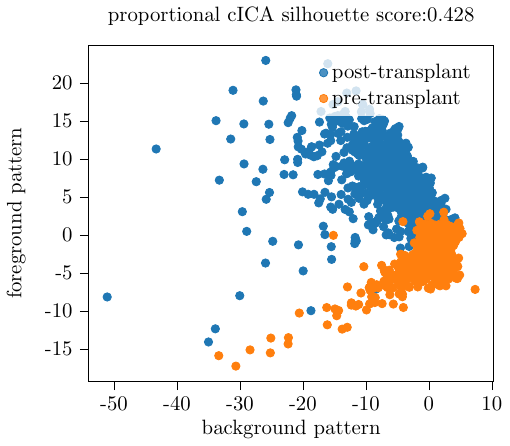}}
\caption{
(a) cPCA on the top 30 PCA components
(b) Proportional cICA plot projected to the top foreground and the top background pattern.}
\label{fig: RNA_comparison}
\end{figure}

\end{document}